\documentclass[11pt, leqno]{article}
\usepackage{verbatim}
\usepackage{hyperref}
\usepackage{amsmath}
\usepackage{enumerate, amsthm}
\usepackage{mathrsfs}
\usepackage{rotating}
\usepackage{xcolor}
\usepackage[margin=1.1 in]{geometry}
\usepackage[margin={1.5cm, 1.5cm},font=small, labelsep=endash]{caption}
\usepackage{mathtools}
\usepackage{tikz}
\usetikzlibrary{patterns}
\usetikzlibrary{arrows}
\usepackage{tikz-dimline}
\usepackage{ifthen}

\usepackage{bm}

      \usepackage{amssymb}
      \usepackage{amsmath}
      \usepackage{centernot}

      \numberwithin{equation}{section}
      \theoremstyle{plain}
      \newtheorem{theorem}{Theorem}[section]

            \newtheorem{thm}[theorem]{Theorem}

      \newtheorem{lemma}[theorem]{Lemma}
      \newtheorem{lem}[theorem]{Lemma}
      \newtheorem{corollary}[theorem]{Corollary}
      \newtheorem{proposition}[theorem]{Proposition}
      \newtheorem{prop}[theorem]{Proposition}

      \theoremstyle{definition}
      \newtheorem{defn}[theorem]{Definition}

      \theoremstyle{remark}
      \newtheorem{remark}[theorem]{Remark}

\renewcommand{\P}{\mathbb P}
\newcommand{\R}{\mathbb R}
\newcommand{\E}{\mathbb E}
\newcommand{\Z}{\mathbb Z}

\newcommand{\N}{\mathbb N}

\newcommand{\lr}[4]{#3\xleftrightarrow[#1]{#2} #4}
     \newcommand{\nlr}[4]{#3\mathrel{\mathop{\centernot\longleftrightarrow}_{#1}^{#2}} #4}

\usepackage{constants}
\usepackage{titletoc}
\usepackage{stackrel}
\usepackage{etex}
\usepackage{enumitem}
\usepackage{enumerate}
\usepackage{vwcol}
\usepackage{stackengine}

\newcommand\couprad{10} 
\newcommand\rangeofdep{40}
\newcommand\pivrange{\the\numexpr\couprad + \rangeofdep\relax}

\newconstantfamily{c}{symbol=c}
\makeatother

\usepackage{titlesec}
\titleformat{\subsection}[runin]{\normalfont\bfseries}{\thesubsection.}{.5em}{}[.]\titlespacing{\subsection}{0pt}{2ex plus .1ex minus .2ex}{.8em}
\titleformat{\subsubsection}[runin]{\normalfont\bfseries}{\thesubsubsection.}{.5em}{}[.]
\titlespacing{\subsubsection}{0pt}{2ex plus .1ex minus .2ex}{.8em}

\usepackage{multicol}
\usepackage{parcolumns}

\usepackage{float}

\title{{\textbf{\normalsize{ 
FINITE RANGE INTERLACEMENTS AND COUPLINGS
}}}}
\date{}

\usepackage{hyperref}
\hypersetup{linktocpage,
  colorlinks   = true,
  urlcolor     = blue,
  linkcolor    = blue,
  citecolor   = black
}

\begin{document}
\thispagestyle{empty}
\maketitle
\vspace{0.1cm}
\begin{center}
\vspace{-1.7cm}
Hugo Duminil-Copin$^{1,2}$, Subhajit Goswami$^3$, Pierre-Fran\c{c}ois Rodriguez$^4$,\\[1em] Franco Severo$^{5}$ and Augusto Teixeira$^6$

\end{center}
\vspace{0.1cm}
\begin{abstract}
In this article, we consider the interlacement set $\mathcal{I}^u$ at level $u>0$ on $\Z^d$, $d \geq3$, and its finite range version $\mathcal{I}^{u,L}$ for $L >0$, given by the union of the ranges of a Poisson cloud of random walks on $\Z^d$ having intensity $u/L$ and killed after $L$ steps.
As $L\to \infty$, the random set $\mathcal{I}^{u,L}$ has a non-trivial (local) limit, which is precisely $\mathcal{I}^u$.  
A natural question is to understand how the sets $\mathcal{I}^{u,L}$ and $\mathcal{I}^{{u}}$ can be related, if at all,~in such a way that their intersections with a box of large radius $R$ almost coincide. We address this question, which depends sensitively on $R$, by developing couplings allowing for a similar comparison to hold with very high probability for $\mathcal{I}^{u,L}$ and $\mathcal{I}^{{u'},2L}$, with $u' \approx u$. In particular, for the vacant set $\mathcal{V}^u=\Z^d \setminus \mathcal{I}^u$ with values of $u$ near the critical threshold, our couplings remain effective at scales $R \gg \sqrt{L}$, which corresponds to a natural barrier across which the walks of length $L$ comprised in $\mathcal{I}^{u,L}$ \textit{de-solidify} inside $B_R$, i.e.~lose their intrinsic long-range structure to become increasingly `dust-like'. These mechanisms are complementary to the \textit{solidification} effects recently exhibited in \cite{zbMATH07227743}.  By iterating the resulting couplings over dyadic scales $L$, the models $\mathcal{I}^{u,L}$ are seen to constitute a stationary finite range approximation of $\mathcal{I}^u$ at large spatial scales near the critical point $u_*$. 
Among others, these couplings are important ingredients for the characterization of the phase transition for percolation of the vacant sets of random walk and random interlacements in the companion articles \cite{RI-I,RI-II}.
\end{abstract}

\vspace{0.5cm}

\begin{flushleft}
\thispagestyle{empty}
\vspace{0.1cm}
{\footnotesize
\noindent\rule{6cm}{0.35pt} \hfill {\normalsize August 2023} \\[2em]

\begin{multicols}{2}
\small
$^1$Institut des Hautes \'Etudes Scientifiques
 \\   35, route de Chartres \\ 91440 -- Bures-sur-Yvette, France.\\ \url{duminil@ihes.fr}\\[2em]
 
$^2$Universit\'e de Gen\`eve\\
 Section de Math\'ematiques\\
 2-4 rue du Li\`evre \\
1211 Gen\`eve 4, Switzerland.\\
\url{hugo.duminil@unige.ch} \\[2em]

 $^3$School of Mathematics\\
 Tata Institute of Fundamental Research\\
 1, Homi Bhabha Road\\
 Colaba, Mumbai 400005, India. \\ \url{goswami@math.tifr.res.in} \columnbreak
 
\hfill$^4$Imperial College London\\
\hfill Department of Mathematics\\
\hfill London SW7 2AZ \\
\hfill United Kingdom.\\
\hfill \url{p.rodriguez@imperial.ac.uk}  \\[2em]

\hfill$^5$ETH Zurich\\
\hfill Department of Mathematics\\
\hfill R\"amistrasse 101\\
\hfill 8092 Zurich, Switzerland.\\
\hfill \url{franco.severo@math.ethz.ch}\\ [2em]

\hfill $^6$Instituto de Matem\'atica Pura e Aplicada\\
\hfill  Estrada dona Castorina, 110\\
\hfill  22460-320, Rio de Janeiro - RJ, Brazil.\\
\hfill  \url{augusto@impa.br}

\end{multicols}
}
\end{flushleft}

%
%
%
\newpage
\setcounter{page}{1}




\section{Introduction}
\label{Sec:intro}

Random interlacements form a prime example of a model exhibiting long-range dependence which gives rise to intriguing critical phenomena. An instance of this is the percolation transition associated to the vacant set $\mathcal{V}^u$ of random interlacements as $u>0$ varies across a critical threshold $u_*$, see \cite{MR2680403,MR2512613}, which is intrinsic to various geometric questions concerning random walk (or Brownian motion) in transient setups; see, e.g.,~\cite{zbMATH05054008,benjamini2008giant,sznitman2009, MR2838338, zbMATH06797082,MR3602841, zbMATH07227743}.

Within the framework considered in the present article, the set $\mathcal{V}^u$ is a random translation invariant subset of $\Z^d$, $d\geq3$, decreasing in $u>0$ and obtained as follows. One introduces a Poisson point process $\eta^*$ on $ W^* \times \R_+$, the space of labeled bi-infinite 
transient lazy $\Z^d$-valued trajectories modulo time-shift; see \S\ref{subsec:RI} for precise definitions, in particular \eqref{e:RI-intensity} regarding its intensity measure. 
The interlacement set $ \mathcal{I}^u= \Z^d \setminus \mathcal{V}^u$ at level $u$ is defined, for a given realization $\eta^*=\sum_i \delta_{(w_i^*, u_i)}$, as the trace of all trajectories in this 
Poisson cloud with label at most $u$, 
\begin{equation}\label{eq:def-I}
\mathcal{I}^u= \mathcal{I}^u(\eta^*)= \bigcup_{i: u_i \leq u} \text{range}(w_i^*).
\end{equation}
The use of lazy random walks in the construction is a matter of convenience and amounts to an inconsequential rescaling of $u$. 

Among its essential and also most daunting features, correlations in the occupation field of $\mathcal{I}^u$ are governed by the Green's function of the random walk, see, e.g.,~\cite[(1.68)]{MR2680403}, and thus decay like $|x-y|^{2-d}$ at large distances $|x-y|\to \infty$ for any $u>0$. A natural way to try to tame this long-range dependence is to truncate the model by introducing a finite time horizon for the trajectories; truncations of this and similar kinds have appeared  in the literature, see e.g.~\cite{MR3962876, zbMATH07577023, 10.1214/23-EJP950, 10.1214/22-EJP824}. The family of finite range models we will consider is defined as follows. Let $P_x$ denote the canonical law of the discrete-time lazy random walk on $\Z^d$ started at $x$ and $X = (X_n)_{n \ge 0}$ the corresponding process. Consider the product measure $\nu$ on $ W_+ \times \R_+$, where $W_+$ is the space of forward $\Z^d$-valued trajectories (supporting $P_x$), with
\begin{equation}
\label{eq:mu_intensity}
\nu ( B \times [0,u])  =u \sum_{x \in \Z^d} P_x[X \in  B],
\end{equation}
for measurable sets $B$. The measure $\nu$ in \eqref{eq:mu_intensity} induces a Poisson point process $ \eta$ on 
$ W_+ \times \R_+$, defined on its canonical space $(\Omega_+, \mathcal{A}_+)$, with 
$\nu$ as its intensity measure. For an arbitrary (density) 
function $f:\Z^d \to \R_+$ and $L \geq 1$, one then defines, in analogy with \eqref{eq:def-I}, 
\begin{equation}
	\label{eq:J}
	\mathcal{I}^{f ,L}= \mathcal{I}^{f ,L} (\eta) =\bigcup_{i \, : u_i \leq \frac{4d}{L} f(w_i(0))} w_i[0,L-1],
\end{equation}
for a realization $\eta = \sum_{i } \delta_{( w_i,u_i)}$, where $w_i[s,t]  \stackrel{{\rm def.}}{=} \{ x\in \Z^d: x= w_i(n)  \text{ for some $s \leq n \leq t$}\}$, for $ t \geq s \geq 0$. In words, $\mathcal{I}^{f ,L}$ comprises the trace of the first $L$ steps of a Poissonian number of trajectories, started with density proportional to $\frac{1}{L}f(\cdot)$. For $u 
\geq 0$, we write $\mathcal{I}^{u ,L} $  when $f(x)=u$ for all $x \in \Z^d$. 
The random set $\mathcal{I}^{u ,L} $ is translation invariant, and, as will be shown in Proposition~\ref{P:loclimit}, for any $u>0$ one has that
\begin{equation}
\label{eq:conv-law}
\mathcal{I}^{u ,L} \stackrel{d}{\longrightarrow} \mathcal{I}^u \text{ as $L \to \infty$.}
\end{equation} 
In view of \eqref{eq:conv-law}, the random set $\mathcal{I}^{u ,L}$ thus constitutes a finite range approximation of $\mathcal{I}^u$ in law. One thus naturally wonders in how far (if at all) the limit $L \to \infty$ can be understood in a `pathwise' sense. Existing coupling techniques, which have a long history in the area, see~\cite{MR2891880,MR2838338,PopTeix,MR3563197, PRS23, zbMATH06247265, CaioSerguei2018, 10.1214/23-EJP950}, are virtually all \textit{local} in the sense that trajectories entering the picture evolve for a time much larger than the diffusive time scale associated to the box in which the coupling is constructed. By adapting these methods, it is thus plausible (and actually true, see Section~\ref{sec:local_coup}) to expect $\mathcal I^u$ and $\mathcal I^{u,L}$ to be comparable inside a box of radius $R$ as long as $L = L(R) \gg R^2$, which in itself is already not entirely straightforward to show, cf.~Proposition~\ref{prop2:cube}~below.

In contrast, for matters relating e.g.~to the near-critical regime around $u_*$, one is often interested in pushing such comparisons much further, to scales $L=L(R)$ well below the diffusive scale $R^2$. This is related to the conjectured fractal nature of large clusters near the critical point. The underlying question thus becomes one of witnessing `extended objects' that carry long-range information at spatial scale $R$ (for instance,~random walk trajectories evolving for time $\gtrsim R^2$) `materialize' out of smaller (sub-diffusive) `particles'.
We will return to this problem at the end of this introduction, see Theorem~\ref{thm:main III}, which illustrates our main results by yielding a coupling with much smaller `localization scale' $L(R)$ than $R^2$ in the regime near $u_*$. Theorem~\ref{thm:main III} is but one application of the main couplings developed in this article, which we now present.

\subsection{Couplings and obstacles}\label{subsec:intro_LL'_coup} Our first two main results,
Theorems~\ref{thm:short_long-intro} and~\ref{thm:long_short_obstacle} below, will allow us to couple the sets $\mathcal{I}^{u ,L}$ and $\mathcal{I}^{u' ,L'}$ for $L'  \leq L$ and suitable values of $u,u'$ with
$ u \approx u'$ in such a way that their ranges almost coincide in large regions. These results will in turn lead to meaningful couplings between $\mathcal{I}^{u ,L}$ and $\mathcal{I}^{u'}$, as will be seen subsequently. We will henceforth always assume that $L \geq L' \geq 1$ are integers with $L'$ dividing $L$ 
and  such that 
\begin{equation}\label{e:couplings-params}
{L}{(\log L)^{-\gamma}}\le L'\le { L }{ (\log L)^{-10}},
\end{equation} 
for some parameter $ \gamma > 10$. The restriction on $L'$ inherent to \eqref{e:couplings-params} is not severe, one can typically extend the range of $L'$ by iterating the following results.

In attempting to compare $\mathcal{I}^{f,L}$ from \eqref{eq:J} for a given profile $f$ with $\mathcal{I}^{f',L'}$, let us first make a reasonable guess at what a good choice of $f'$ may be. A natural way to proceed is to cut the trajectories comprising $\mathcal{I}^{f,L}$ into pieces of length $L'$. Foregoing for a moment the (strong) dependence between successive starting points of the length-$L'$ walks (inherited from the longer length-$L$ trajectories) induced by this cutting procedure, one may plausibly choose $f' = P_L^{L'}(f)  $, where for all $f : \Z^d \to \R$,
\begin{equation}
\label{eq:f'}
P_L^{L'}(f) \stackrel{\textnormal{def.}}{=}  \frac{L'}{L}\sum_{ k=1}^{ {L}/{L'}} P_{(k-1)L'}( f)
\end{equation}
and $P_nf(x) = E_x[f(X_n)]$ denotes the $n$-step transition operator associated to $P_x$. Note in particular that $P_L^{L'}(\cdot)$ acts as identity map on constant functions $f(x)=u$, $x\in \Z^d$. 

The considerations leading to \eqref{eq:f'} are but a simple heuristic. For, unlike the trajectories of length $L'$ constituting $\mathcal I^{f',L'}$ with $f'=P_L^{L'}(f)$, which have independent starting points, the ones obtained 
after cutting $\mathcal I^{f,L}$ carry long-range dependence (for 
instance, most of the time a walk of length $L'$ needs to start where a walk of same length $L'$ ends). The comparison between the two sets is thus a-priori far from clear. A first idea to overcome this issue is to allow some room for homogenization by leaving a suitable gap time $1 \ll t_g\ll L'$ in the cutting procedure, all while still retaining one of two possible inclusions. Of course this does not come free of cost; in particular one should expect to end up with a slightly smaller proportion of walks of length $L'$.	
The following result turns this intuition into a theorem. We refer to the end of this introduction regarding our policy with constants $c,C$ etc., which only depend on $d \geq 3$. Let $B_N = [-N, N]^d \cap \Z^d$ for any $N \ge 0$ and recall that \eqref{e:couplings-params} is in force.

\begin{thm}\label{thm:short_long-intro} 
For all $u \in (0,\infty)$ and integer $K \ge 0$ the following holds. Given any function $f:\Z^d \to [0,u]$ 
such that $f (x) \ge (\log L)^{-\gamma}$ for all $x \in B_{K+L}$, there exists a coupling $\mathbb Q$ of two $\{0, 1\}^{\Z^d}$-valued random variables $\mathcal I_1$, $\mathcal I_2$ such that 
\begin{equation}
\label{eq:short_long-intro}
\begin{split}
&\mathcal I_1 \stackrel{\textnormal{law}}{=}  \mathcal I^{f, L},\, \mathcal I_2 
\stackrel{\textnormal{law}}{=} \mathcal I^{(1 - C(L'/L)^{1/2})P^{L'}_L(f1_{B_K}), L'}
\text{ and}\\
&	\mathbb Q \left[ {\mathcal I}_1  \supset  {\mathcal I}_2  \right] \geq 1 - C 
(u \vee 1)(K+L)^d \,e^{-c \,(L/L')^{1 / 4}}. 
\end{split}
\end{equation}
\end{thm}

Theorem~\ref{thm:short_long-intro} will follow from a more general result, Theorem~\ref{thm:short_long}, proved in Section~\ref{sec:easyCOUPLINGS}; see also Remark~\ref{R:short_long},\ref{rmk:short_length}. The discussion leading to \eqref{eq:short_long-intro} crucially relied on the fact that the inclusion of range-$L'$ trajectories into range-$L$ ones permits one to {\em forget} about gaps 
between walks. On the contrary, the opposite inclusion requires gluing shorter length-$L'$ trajectories into longer ones. This is much more difficult to achieve, and the coupling we derive to this effect in the next result, Theorem~\ref{thm:long_short_obstacle} below, is correspondingly more involved. As one of the main innovations of this article, we now introduce an obstacle set $\mathcal{O}$, which is at the heart of this coupling. We will in fact couple the models outside an enlarged obstacle set $\widetilde{\mathcal{O}}$.
In a nutshell, the walks of length $L'$ will be `wired'  into walks of length $L$ using obstacles, met frequently and by many of the random walks entering the picture, as `hubs;' cf.~Figure~\ref{F:obs}.
\bigskip
\begin{figure}[h!]
  \centering 
  \includegraphics[scale=0.80]{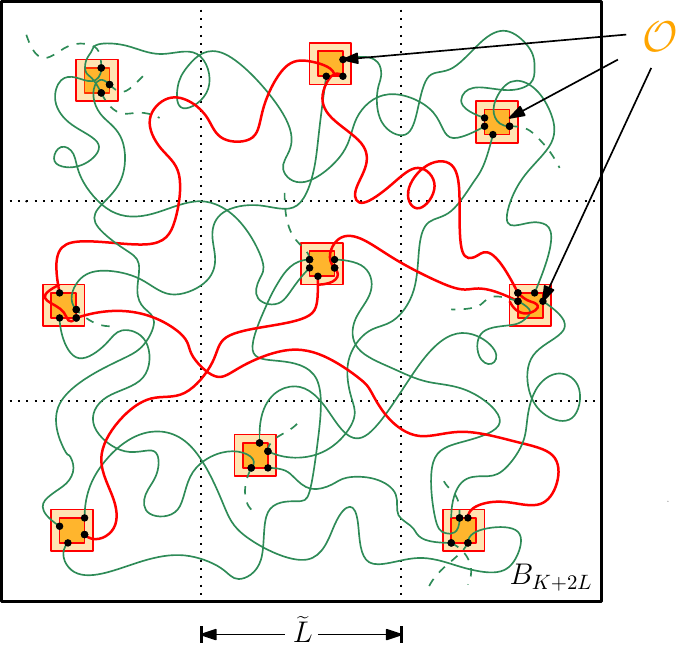}
  \caption{A longer length-$L$ trajectory (red) is obtained by `gluing' pieces of shorter length-$L'$ trajectories (green) within the enlarged obstacle set (orange).}
  \label{F:obs}
\end{figure}

We now introduce the obstacle set $\mathcal{O}$. We assume that $\mathcal O$ is a disjoint union of boxes of equal radius $\ell_\mathcal O \geq1$, called {\em obstacles}. The collection of such 
boxes is denoted by $\mathcal B_\mathcal O$, whence
\begin{equation}\label{eq:obstacles}
\mathcal{O}= \bigcup_{B\in \mathcal{B_{\mathcal{O}}}} B.
\end{equation}
For $U \subset \Z^d$, let $H_U$ denote the entrance time of $X$ in $U$, see below \eqref{eq:Greenasympt} for notation. The key features of $\mathcal{O}$ in \eqref{eq:obstacles} are encapsulated in the following:
\begin{defn}\label{def:O}
Let $K \ge 0$, $L \geq 1$, $\delta_{\mathcal{O}} \in (0,1)$ and
$M_{\mathcal{O}} \geq 1$.  
An obstacle set $\mathcal{O}\subset B_{K + 2L}$ is called $(\delta_{\mathcal{O}},M_{\mathcal{O}})$-\textit{good} if
\begin{align}
\text{(Visibility condition): } &\begin{array}{l} P_x[H_{\mathcal{O}} >  L {(\log L)^{-10 \gamma}} ] \leq \delta_{\mathcal{O}},  \text{ for all } x \in B_{K+L}; 
\end{array}\label{eq:obs-visible}\\[0.7em]
\text{(Density condition): } &\begin{array}{l}\displaystyle P_{\mu}[X_{H_{\mathcal{O}}} \in B , \, H_{\mathcal{O}} < \infty] \geq M_{\mathcal{O}} L, 
\\[0.3em]
\displaystyle 
\text{for all $B \in \mathcal B_{\mathcal{O}}$ with $B \cap U_{\mathcal O} \neq \emptyset$,
}
\end{array} \label{eq:obs-dense}
	\end{align}
where $U_{\mathcal O}= B_{K+\sqrt{L} (\log L)^{\gamma}}$, $P_{\mu}=\sum_x\mu(x)P_x$ and 
\begin{equation}
\label{eq:mu-cond} 
\mu=\alpha^{-1} 1_{U_{\mathcal O}}, \text{ with } \alpha = (\log L)^{10\gamma}. 
\end{equation}
	\end{defn}
The term \textit{obstacle} is fitting, cf.~for instance~\cite{MR1717054}: indeed verifying the conditions of Definition~\ref{def:O} will notably require some control on $g_{\mathcal{O}}(\cdot,\cdot)$, the Green's function of the walk killed on the obstacle set $\mathcal{O}$, see \eqref{eq:Greenkilled} for notation. Further note that, apart from $(\delta_{\mathcal{O}},M_{\mathcal{O}})$, a good obstacle set also depends implicitly on $L$, $K$, $\ell_{\mathcal{O}}$ and $\gamma \, ( > 10)$, cf.~above \eqref{e:couplings-params}. The parametrization in \eqref{eq:obs-dense} is chosen so that $M_{\mathcal{O}}$ will eventually correspond to a number of trajectories.
The manufacture of a good obstacle set is somewhat intricate because $\mathcal{O}$ needs to meet competing interests in satisfying \eqref{eq:obs-visible} and \eqref{eq:obs-dense} simultaneously. We return to this below the next theorem. In doing so, we will give concrete examples of obstacle sets $\mathcal{O}$ satisfying Definition~\ref{def:O}. In applications of interest, $\mathcal{O}$ itself will typically also be random.

In order to state our second main result, which exhibits a coupling with inclusions in the `hard' direction, opposite to that of Theorem~\ref{thm:short_long-intro}, we introduce in analogy with \eqref{eq:obstacles} the set 
\begin{equation}
\label{eq:obs-tilde}
	\widetilde{\mathcal{O}}= \bigcup_{B\in\mathcal B_\mathcal O}  \widetilde{B},
\end{equation}
where, for every $B\in\mathcal B_\mathcal O$, $\widetilde B$ is the concentric ball of radius $\tilde \ell_\mathcal O = \ell_\mathcal O^{1+\frac1{100}}$. We refer to the set $\widetilde{\mathcal{O}}$ defined by \eqref{eq:obs-tilde} as {\em enlarged} obstacle set in the sequel.
\begin{thm}\label{thm:long_short_obstacle} 
For all $u \in (0,\infty)$, integer $K \ge 0$ and $\varepsilon \in (0,1)$, the following holds.
Given any $f:\Z^d \to [0, u]$ such that $f\vert_{B_{K + L}} \geq (\log L)^{-\gamma} $ and $f=0$ outside $B_{K + L}$, 
and any
$(\delta_{\mathcal O},M_\mathcal O)$-good obstacle set 
$\mathcal{O} (\subset B_{K + 2L})$ 
with $C \ell_{\mathcal 
O}^{-1/100} \le \varepsilon$, there exists a coupling $\mathbb Q$ of $(\mathcal I_1, \mathcal I_2)$ such that
\begin{equation}\label{eq:long_short}
\begin{split}
&\mathcal I_1 \stackrel{\textnormal{law}}{=}  \mathcal I^{f1_{B_K}, L},\, \mathcal I_2 \stackrel{\textnormal{law}}{=} \mathcal I^{(1 + \varepsilon)P^{L'}_L(f), L'}; \text{ and, for $L \ge L_0(d, \gamma,u)$,}\\
&	\mathbb Q \big[ ({\mathcal I}_1\setminus \widetilde{\mathcal O} ) \subset  ({\mathcal I}_2\setminus  \widetilde{\mathcal O} ) \big] \geq  1 - C (u \vee 1) (K+L)^d\big( e^{-c(\varepsilon^2 M_{\mathcal{O}} \wedge (L/L')^{1/4})}  \vee   \delta_{\mathcal{O}} \big).
\end{split}
\end{equation}
\end{thm}

The inclusion \eqref{eq:long_short} complements \eqref{eq:short_long-intro} but obviously the price to pay is to avoid the enlarged obstacle set $\widetilde{\mathcal O}$. 
This is because $\widetilde{\mathcal{O}}$ delimits a region in which pieces of trajectories are glued together to form 
longer ones; cf.~Figure~\ref{F:obs}. We will soon see (\S\ref{subsec:env}) how \eqref{eq:long_short} can be employed to deduce meaningful statements e.g.~concerning boundary clusters of vacant sets. For the time being, an obvious question is to construct $(\delta_{\mathcal{O}},M_{\mathcal{O}})$-good obstacle sets, i.e.~having small $\delta_{\mathcal O}$ and large $M_{\mathcal O}$ in view of the error term appearing in \eqref{eq:long_short}.


\subsection{Constructing good obstacle sets}  
Considering the event on the left-hand side of \eqref{eq:long_short}, interesting choices for $\mathcal{O}$ (and a fortiori, $\widetilde{\mathcal O}$) ought to be as \textit{small} as possible. Indeed, the set $\mathcal{O}= B_{K+L}$ for instance is a good obstacle set (in particular, Definition~\ref{def:O} is not vacuous), however it renders \eqref{eq:long_short} moot. The obstacle sets we exhibit below do in fact have vanishing asymptotic density in $B_{K+2L}$ as $L \to \infty$; see Corollary~\ref{cor:long_short_obstacle}. These examples, which will be used in applications below, underline the strength of the coupling exhibited in Theorem~\ref{thm:long_short_obstacle}. 

Before giving concrete examples let us first highlight one key point. The obstacles comprising $\mathcal{O}$ have two 
defining properties, \eqref{eq:obs-visible} and \eqref{eq:obs-dense}, which act in opposite ways with regards to choosing scales and finding the right resolution for~$\mathcal{O}$: on the one hand, 
$\mathcal{O}$ needs to have good trapping properties, see \eqref{eq:obs-visible}, i.e.~be hard to avoid for the random walk started in the bulk, a feature which naturally improves upon increasing $\ell_{\mathcal{O}}$, the radius of an individual obstacle. On the 
other hand each box $ B\in \mathcal{O}$ needs to be small enough as to retain a high `surface 
density' of incoming trajectories, see \eqref{eq:obs-dense}, which intuitively favors mixing and facilitates the gluing. This will later be quantified in terms of a mean free path $\sqrt{t_{\mathcal{O}}}$ for the random walk among the obstacles $\mathcal{O}$, see~Lemma~\ref{L:mfp} below, which has to be sufficiently large (much larger than the typical obstacle separation; see Remark~\ref{R:MFP} below). 

We now discuss some examples of good obstacle sets, including suitable periodic arrangements of boxes, which constitute the simplest example. Eventually though, we will be interested in the disordered case where obstacles are random, so we immediately formulate a suitable relaxation of the periodicity condition.

We will let the obstacle set $\mathcal{O}$
consist of boxes $B$ of radius $\ell_{\mathcal O} \geq 1$, separated by a mesoscopic scale 
$\widetilde{L}$ with $\ell_{\mathcal O} \ll \widetilde{L} \ll L$. Suppose that
\begin{equation}\label{eq:obs-cond-scales}
(\log L)^{100\gamma} \le \ell_{\mathcal O} \le 
L^{\frac1{3d}}.\end{equation} 
We then introduce the scale $\widetilde{L}$, which will govern the typical distance between obstacles, as
\begin{equation}\label{e:obs-Ltilde-choice}
\widetilde{L} =  \lfloor ( \alpha L  \ell_{\mathcal O}^{({d-2})/{2}})^{1/d}  \rfloor, 
\text{ where } \alpha = (\log L)^{10\gamma} \text{ (cf.~\eqref{eq:mu-cond})};
\end{equation}
in fact, any positive exponent less than $d-2$ for $\ell_{\mathcal{O}}$ would do in \eqref{e:obs-Ltilde-choice}, which is related to the capacity of a box of radius $\ell_{\mathcal O}$, see~\eqref{e:cap-box} below. 
The mechanism behind the inverse proportionality of $\widetilde{L}$ as a function of the `ellipticity' lower bound $\alpha^{-1}$ introduced in \eqref{eq:mu-cond} is easy to grasp intuitively. If $\alpha^{-1}$ decreases, satisfying the density condition \eqref{eq:obs-dense} becomes harder, and retaining a given `surface density' $M_{\mathcal{O}}L$ on an individual obstacle $B$ will be eased by making the obstacles sparser, i.e.~increasing $\widetilde{L}$, which indeed grows with $\alpha$ by \eqref{e:obs-Ltilde-choice}.

Continuing with the constuction of $\mathcal{O}
$, under the assumption \eqref{eq:obs-cond-scales}, it is plain to see that 
 $\ell_{\mathcal O} \ll \widetilde L \ll L$ when $L$ is large. Now, given the mesoscopic scale $\widetilde{L}$ in \eqref{e:obs-Ltilde-choice}, let $\widetilde{\mathcal{C}}$ denote the collection of boxes 
$B(z,\widetilde{L}) \stackrel{{\rm def.}}{=} z + B_{\widetilde L}$ where, roughly speaking and as will be 
made precise momentarily, see \eqref{e:obs-tilde-C} below, $z$ ranges over all points of 
$\widetilde{ \mathbb{L}} \stackrel{{\rm def.}}{=}3\widetilde{L} \Z^d \cap B_{K + 3L/2}$. For various applications we have in mind, see e.g.~the next paragraph \S\ref{subsec:env}, see also~\cite{RI-I}, we sometimes need the enlarged obstacle set $\widetilde{\mathcal O}$ to 
avoid the boundary $\partial B$ (see Section~\ref{s:not} for notation) of a given box $B=B(x, N)$, for some $x \in \Z^d$ and $N \geq 0$. For the purposes of this exposition, the reader may choose focus on the case $B=\emptyset$ (in which case $\partial B=\emptyset$ by convention in what follows). Accordingly, we now set, for $B \in \{ B(x, N) : x \in \Z^d, N \geq 0 \} \cup  \{ \emptyset\}$,
\begin{equation}
\label{e:obs-tilde-C}
\widetilde{\mathcal{C}}= \big\{ \widetilde{C}: \, \widetilde{C}= B(z,\widetilde{L}) \text{ for some }  z \in \widetilde{ \mathbb{L}} \text{ s.t.~} B(z, 2\widetilde{L}) \cap \partial B= \emptyset \big\}.
\end{equation}
In the sequel, we usually employ the notation $\widetilde{C}$ to denote a generic element of $\widetilde{\mathcal{C}}$, which we refer to as a \textit{cell}. 

A good obstacle set will in essence comprise one obstacle (a box of radius $\ell_{\mathcal{O}}$) per cell. Thus, let $ \{y_{\widetilde{C}} : \widetilde{C} \in \widetilde{\mathcal C}\}$ denote an arbitrary collection with $y_{\widetilde{C}} \in\widetilde{C}$ for each $\widetilde{C} \in \widetilde{\mathcal C}$.
Proposition~\ref{L:final-good-obs} will imply that, under \eqref{eq:obs-cond-scales} and for all $K \geq 0$, $u \in (0,\infty)$ and $L \geq C(\gamma)$,
\begin{equation}
\label{eq:ex-good-obs}
\text{\parbox{9.0cm}{the obstacle set $\textstyle \mathcal{O}  =  \bigcup_{\widetilde{C} \in \widetilde{\mathcal{C}}}  
B(y_{\widetilde{C}}, \ell_{\mathcal O}) \subset B_{K+2L}$ is $(\delta_{\mathcal{O}},M_{\mathcal{O}})$-good with $\delta_{\mathcal{O}} = e^{-c(\log 
L)^{\gamma}} $, $M_{\mathcal{O}}= c  \ell_{\mathcal O}^{({d-2})/{2}}$. }}
\end{equation}
Although this will be too limiting for our purposes, let us emphasize that \eqref{eq:ex-good-obs} with $B=\emptyset$ in \eqref{e:obs-tilde-C} yields instances of fully `periodic' arrays of obstacles (inside $B_{K + 3L/2}$). Moreover, by computing the volume occupied by $\mathcal{O}$ and combining \eqref{eq:ex-good-obs} (see also Proposition~\ref{L:final-good-obs}) with Theorem~\ref{thm:long_short_obstacle}, one immediately arrives at the following result. As in the statement of Theorem~\ref{thm:long_short_obstacle} we assume implicitly that \eqref{e:couplings-params} holds (for some $\gamma> 10$) and that $u \in (0,\infty)$,  $K \ge 0$ and $f:\Z^d \to [0, u]$ is such that $f\vert_{B_{K + L}} \geq (\log L)^{-\gamma} $ and $f=0$ outside $B_{K + L}$. 

\begin{corollary}\label{cor:long_short_obstacle}
For any $\ell_{\mathcal{O}}$ satisfying~\eqref{eq:obs-cond-scales} and $ \varepsilon \geq c \ell_{\mathcal 
O}^{-1/100} $, one can find a good obstacle set $\mathcal{O} \subset B_{K+2L}$ and a coupling $\mathbb Q$ of $(\mathcal I_1, \mathcal I_2)$ with marginals as in \eqref{eq:long_short} such that
\begin{equation}\label{eq:long_short-cor}
\mathbb Q \big[ ({\mathcal I}_1\setminus \widetilde{\mathcal O} ) \subset  ({\mathcal I}_2\setminus  \widetilde{\mathcal O} ) \big] \geq  1 - e^{-c(\log 
L)^{\gamma}} \, \text{ and } \,  {|\widetilde{\mathcal O} |}/{|B_{K+2L}|} \leq L^{-c}.
\end{equation}
\end{corollary}
Corollary~\ref{cor:long_short_obstacle} highlights the strength of Theorem~\ref{thm:long_short_obstacle}. Indeed, the second part of \eqref{eq:long_short-cor} implies in particular that the set $\widetilde{\mathcal O} $ removed from the region of coupling has vanishing asymptotic density in $B_{K+2L}$ as $L\to\infty$. It is an interesting open problem to determine how small $|\widetilde{\mathcal O} |$ can be chosen for the inclusion in  \eqref{eq:long_short-cor} to continue to hold with high probability.

\subsection{Disorder and coupling of clusters} \label{subsec:env}
To illustrate the usefulness of Theorem~\ref{thm:long_short_obstacle}, we now discuss a specific case where $\mathcal{O}$ is random with range among the obstacle sets of the form given by~\eqref{eq:ex-good-obs}. The additional randomness comes from an auxiliary configuration $\mathcal{I} = \mathcal{I}(\omega) \subset \Z^d$ for $\omega \in \Omega$, defined on an auxiliary space $(\Omega, \mathcal{A}, \mathbf{P})$. We think of $\mathbf{P}$ as generating a random `environment' and write 
$\mathcal V = \Z^d \setminus 
\mathcal I$ with $\mathcal{I} = \mathcal{I}(\omega)$. The random `environment' of hard obstacles $\mathcal{O}=\mathcal{O}(\omega)$ will then consist of realizations of boxes fulfilling \eqref{eq:ex-good-obs} around which \textit{disconnection} in $\mathcal{V}$ occurs. As asserted in the next theorem, this yields a coupling of \textit{boundary} clusters (of a box $B$) for the vacant sets corresponding to the superposition with $\mathcal{I}$ of the two configurations to be coupled, cf.~Fig.~\ref{F:discon}. In practice, $\mathcal{I}$ corresponds to a (significant) fraction of e.g.~$\mathcal{I}^{f,L}$, which remains `frozen' while applying our coupling results, Theorems~\ref{thm:short_long-intro} and~\ref{thm:long_short_obstacle}; we describe this in detail in \S\ref{subsec:intro-L-2L}. As will become clear in the course of proving Theorem~\ref{thm:long_short} below, the reason for discarding cells close to $\partial B$
 in \eqref{e:obs-tilde-C} is that the presence of obstacles near the boundary could otherwise spoil the configuration of boundary clusters which we aim to couple.

We now make precise the relevant notion of disconnection events for $\mathcal{V}$. To this effect, we introduce one additional scale $\bar \ell_{\mathcal O}$ with $ \tilde{\ell}_{\mathcal O} < \bar \ell_{\mathcal O}  \le \widetilde L$ (recall \eqref{eq:obs-tilde} regarding $\tilde{\ell}_{\mathcal O}$). For 
any $y \in \Z^d$ and $S \subset\subset \Z^d$, define (under $\mathbf{P}$)
\begin{equation}\label{eq:disc-coup}
\begin{split}
&\mathrm{Disc}(y) = \mathrm{Disc}_{\tilde{\ell}_{\mathcal O}, \bar\ell_{\mathcal O}}(y)={\big\{\nlr{}{ \mathcal{V}}{B(y, \tilde{\ell}_{\mathcal O})}{\partial B(y, \bar \ell_{\mathcal O})}\big\}}, \mbox{ and}\\
&\mathrm{Disc}(S) = \mathrm{Disc}_{\tilde{\ell}_{\mathcal O}, \bar\ell_{\mathcal O}}(S) = \{ y \in S : \mathrm{Disc}_{\ell_{\mathcal O}, \bar\ell_{\mathcal O}}(y) \text{ occurs} \}. 
\end{split}
\end{equation}
Notice that while $\mathrm{Disc}(y)$ is an event, $\mathrm{Disc}(S)$ is a random subset of $S$ (under $\mathbf{P}$). Of interest to us will be the disconnection event (under $\mathbf{P}$)
\begin{equation}
	\label{e:obs-final2}
	\mathscr{D}= \mathscr{D}(\omega)= \big\{ \mathrm{Disc}(\widetilde{C}) \neq \emptyset \text{ for all } \widetilde{C} \in \widetilde{\mathcal{C}}\big\}.
\end{equation}
where $\widetilde{\mathcal C}$ is as in \eqref{e:obs-tilde-C} for an arbitrary box $B= B(x,N)$, which will soon play the role of the region in which we exhibit a coupling (cf.~\eqref{eq:long-j-i} below). In a nutshell, on the event $\mathscr{D}= \mathscr{D}(\omega)$ appearing in \eqref{e:obs-final2}, we can for each cell $\widetilde{C} \in \widetilde{\mathcal{C}}$
pick a box around which disconnection occurs in $\mathcal{V}(\omega)$ to form the obstacle set $\mathcal{O}= \mathcal{O}(\omega)$. Combining Theorem~\ref{thm:long_short_obstacle} with \eqref{eq:ex-good-obs} (see also Proposition~\ref{L:final-good-obs}), we then arrive at the following result, proved at the end of Section~\ref{sec:obstacle_set}. Hereinafter we use $\mathscr{C}^{\partial}_S(\mathcal V)$, for $S, \mathcal V 
\subset \Z^d$, to denote the connected component of $\partial S$ in $\mathcal{V} \cap S$.
\begin{thm}
\label{thm:long_short}
For all $u \in (0,\infty)$, $\varepsilon \in (0,1)$, integer $K, N \ge 0$, $B=B(x,N)$ for $x \in \Z^d$, and $f:\Z^d \to [0, u]$ as in Theorem~\ref{thm:long_short_obstacle}, the following holds. If \eqref{eq:obs-cond-scales} holds, $\varepsilon  \ge C\ell_{\mathcal{O}}^{-1/100} \vee (\ell_{\mathcal{O}}^{-1/(d-2)} (\log L)^{\gamma/2})$ and $ \tilde{\ell}_{\mathcal O} < \bar \ell_{\mathcal O}  \le \widetilde L$,
 there exists  a coupling $\mathbb Q_{\omega}$ of $\mathcal I_1, \mathcal I_2$ such that, $\mathbf{P}$-a.s.,
\begin{equation}\label{eq:long-j-i}
\begin{split}
		&\mathcal I_1 \stackrel{\textnormal{law}}{=}  \mathcal I^{f1_{B_K}, L} \cup \mathcal I(\omega),\ \mathcal I_2 \stackrel{\textnormal{law}}{=} \mathcal I^{(1 + \varepsilon)P^{L'}_L(f), L'} \cup \mathcal I(\omega); \text{ and for $L \geq L_0(d,\gamma)$,}\\
		& \mathbb Q_{\omega} \big[ \mathscr{C}^{\partial}_{B}\big(\mathcal V(\mathcal I_{1})  \big)  \supset  \mathscr{C}^{\partial}_{B}\big(\mathcal V(\mathcal I_{2})  
		\big)   \big] 1_{\mathscr{D}(\omega)} \geq  1 - C (u \vee 1) (K+L)^d e^{-c(L/L')^{1/4}},
	\end{split}
\end{equation}
where $\mathcal{V}(\mathcal{I}) \stackrel{{\rm def.}}{=} \Z^d \setminus \mathcal{I}$ for any $\mathcal I \subset \Z^d$.
\end{thm}

\noindent In fact, our arguments yield that one can couple not only all the boundary clusters $\mathscr{C}^{\partial}_{B}(\cdot)$ in \eqref{eq:long-j-i} but actually all the clusters intersecting the complement of the $\bar{\ell}_{\mathcal{O}}$-thickening of the obstacle set $\mathcal{O}=\mathcal{O}(\omega)$ used in the proof, which is of the form \eqref{eq:ex-good-obs}.

\medskip

Theorem~\ref{thm:long_short} also implies the following  useful `annealed' coupling. With $\mathbb Q [\cdot] \stackrel{{\rm def.}}{=} \int \mathbb Q_{\omega}[\cdot]d \mathbf{P}(\omega)$, we immediately obtain, under the assumptions of Theorem~\ref{thm:long_short}, that $\mathbb Q$ gives a coupling of three configurations 
$\mathcal I$ (with law specified by $\mathbf{P}$), $\mathcal I_1 \stackrel{{\rm law}}{=} \mathcal I^{f1_{B_K}, L}$ and $\mathcal I_2 
\stackrel{{\rm law}}{=} \mathcal I^{(1 + \varepsilon)P_{L}^{L'}(f), L'}$ such that $\mathcal I, 
\mathcal I_k$ are independent under $\mathbb Q$ for any choice of $k\in \{1,2\}$ and $L_0(d,\gamma)$,
\begin{equation}\label{eq:long-j-i_annealed}
\begin{split}
 \mathbb Q \big[ \mathscr{C}^{\partial}_{B}\big(\mathcal V(\mathcal I_{1} \cup \mathcal I )  \big)  \supset  \mathscr{C}^{\partial}_{B}\big(\mathcal V(\mathcal I_{2} \cup \mathcal I)  
		\big)   \big] \geq  \mathbf{P}[\mathscr D] - C (u \vee 1) (K+L)^d e^{-c(L/L')^{1/4}} .
	\end{split}
\end{equation}
In the next paragraph, we return to the question of coupling
$\mathcal{I}^{u ,L}$ and $\mathcal{I}^u$, see around \eqref{eq:conv-law}. This provides a concrete and simple example of environment $\omega$ with $\mathcal{I}=\mathcal{I}(\omega)$ of interest, which illuminates the use \eqref{eq:long-j-i} and \eqref{eq:long-j-i_annealed}; see also \S\ref{subsec:outlook} below for further applications.

\subsection{Coupling the vacant sets of $\mathcal I^{u, L}$ and $\mathcal I^u$} \label{subsec:intro-L-2L}

We now attend to the question of taming correlations in $\mathcal I^u$ by comparison with $\mathcal I^{u,L}$, which is $2L$-dependent (in fact, the `effective' range of dependence 
is rather of order $\sqrt L$, the typical diameter of a random walk of length $L$). As explained below, the following result can be obtained by means of Theorems~\ref{thm:short_long-intro} and~\ref{thm:long_short}.

\begin{theorem}\label{thm:main III}
Let $u\in (0,\infty)$ and $\gamma > 10$ be such that, with $D(\ell)= \exp\{ \ell^{\Cl[c]{c:disc}/{\gamma}}\}$, 
\begin{equation}
\label{eq:def_M}
\liminf_{\ell} \exp\big\{  \ell^{\Cr{c:disc}d/{\gamma}}\big\} \P[\nlr{}{ \mathcal{V}^u}{B_\ell}{\partial B_{D(\ell)}}] > 0.
\end{equation}
Then for all $v \geq u(1 + (\log L)^{-2})$, $R \ge 0$ and (dyadic) integers $L \ge L_0(d, \gamma, 
u,v)$, letting $v_\pm = v(1 \pm (\log L)^{-3})$, there exists a coupling $\mathbb{Q}$ of $(\mathcal V^{v_+,L}, \mathcal V^{v}, \mathcal V^{v_+,L})$ such that 
\begin{align}\label{eq:upper}
&\mathbb Q\big[\mathscr C^\partial_{B_R}(\mathcal V^{v_+,L})\subset\mathscr C^\partial_{B_R}(\mathcal V^{v})\subset \mathscr C^\partial_{B_R}(\mathcal V^{v_-,L})\big]\ge 1 - C(R+L)^de^{-c(\log L)^{\gamma}}.
\end{align}
\end{theorem}

In words, Theorem~\ref{thm:main III} asserts that, if \eqref{eq:def_M} holds, one can localize $\mathcal V^{v}$ up to sprinkling to scale $L$, with an error term that remains effective well below the barrier $L \approx {R}^2$; cf.~also \eqref{eq:def_M-quant} and Proposition~\ref{P:ind-step}, which imply a similar coupling as in \eqref{eq:upper} with $\mathcal V^{v,2L}$ in place of $\mathcal V^{v}$ under the sole assumption \eqref{eq:def_M} for $\ell=(\log L)^C$, which is very mild. This type of statement can be viewed as  complementary to the solidification mechanisms exhibited in \cite{zbMATH07227743}. Indeed, the properties defining our obstacle sets $\mathcal{O}$, see Definition~\ref{def:O}, bear a loose resemblance to those exhibited by `resonant sets' in the language of \cite{zbMATH07227743}, but the disconnection bounds \eqref{eq:def_M} or \eqref{eq:def_M-quant} defining the random set $\mathcal{O}$ (cf.~Theorem~\ref{thm:long_short}, which is crucially used in the proof) act in the opposite direction than the frequently used assumption $u< \bar{u}$ (which implies \textit{upper} bounds on disconnection), see e.g.~\cite{zbMATH07227743, zbMATH07114721,zbMATH07226362}. Thus our `resonances' rather have a \textit{de-solidifying} effect: indeed \eqref{eq:upper} (along with its one-step version~\eqref{eq:upper_L2L}) indicates that large clusters such as those comprised in $\mathscr C^\partial_{B_R}(\mathcal V^{v})$ can be built using random objects of well-defined `size' (length-$L$ trajectories) over a large spectrum of scales $L$, which hints at an `amorphous' structure rather than a `solidified' object.

Theorem~\ref{thm:main III} is a benchmark result. For instance, \eqref{eq:upper} immediately yields `localization estimates' for crossing probabilities of the form $\P[\lr{}{ \mathcal{V}^v}{B_r}{\partial B_{R}}]$, for any $0 \leq r < R$, by which $\mathcal{V}^v$ can be effectively replaced by $\mathcal{V}^{v_{\pm}, L}$ at suitable values of $v$ up to super-polynomial errors in $R$, with `localization scale $L$' as small as $L=L(R)= \exp\{(\log R)^{1.1/\gamma}\} $, for any (large) $\gamma >1$. As mentioned at the beginning of this introduction, see below \eqref{eq:conv-law}, this is completely outside the scope of (adaptations of) existing techniques, which require $L(R) \gg R^2$.

We 
now sketch how  \eqref{eq:upper} is deduced from the previous results. Doing so will shed light on a typical `random environment' $\mathbf{P}$ (from which the disordered obstacle set is constructed; recall the discussion following \eqref{e:obs-final2}) used in the context of Theorem~\ref{thm:long_short}.

In essence, one obtains \eqref{eq:upper} by concatenating over scales certain `recursive' couplings of similar form as \eqref{eq:upper} but relating $\mathcal V^{\cdot,L}$ and $\mathcal V^{\cdot,2L}$ instead. The statement corresponding to a single `recursive' step, see Proposition~\ref{P:ind-step} below, is worth highlighting, because it can be made fully quantitative in $L$ (this regards in particular the scales $\ell$ at which a disconnection estimate of the form \eqref{eq:def_M} is needed); see Remark~\ref{R:quanti} for more on this. To deduce Proposition~\ref{P:ind-step}, one applies Theorems~\ref{thm:short_long-intro} and~\ref{thm:long_short} (the latter in its annealed version, which is sufficient for this purpose) repeatedly together, each time replacing a small fraction of length-$L$ by length-$2L$ trajectories, thus progressively transforming $\mathcal I^{\cdot,L}$ into $\mathcal I^{\cdot,2L}$. At each step, the comparison involves trajectories of  intermediate length scale $L' \ll L$ as in \eqref{e:couplings-params}, obtained by applying Theorems~\ref{thm:short_long-intro} and Theorems~\ref{thm:long_short} respectively to the (small fraction) of length-$L$ and length-$2L$ trajectories to be exchanged, or vice versa (depending on which inclusion in \eqref{eq:upper} one is proving). In doing so, an independent `bulk' contribution, which generically consists of a mixture of trajectories of both length scales, remains untouched. This bulk part plays the role of the random environment $\mathcal{I}= \mathcal{I}(\omega)$ under $\mathbf{P}$ in the context of \eqref{eq:long-j-i_annealed}. The estimate \eqref{eq:def_M} enters naturally in supplying a lower bound for the quantity $\mathbf{P}[\mathscr D]$ in \eqref{eq:long-j-i_annealed}, which requires `abundance' of disconnections in the underlying vacant 
set $\mathcal{V}= \Z^d \setminus \mathcal{I}$, cf.~\eqref{e:obs-final2}.

\subsection{Outlook}\label{subsec:outlook}

We will in fact consider a larger class of interlacements, called {\em $\rho$-interlacements}, enabling for both varying (time-)length and spatial intensity of the underlying Poisson point process. This class will be useful not only for the proofs below, but also for subsequent applications. In particular, the couplings we derive in Theorems~\ref{thm:short_long-intro} and \ref{thm:long_short} (see also their extensions, Theorems~\ref{thm:short_long} and~\ref{P:long_short} below), play an instrumental role in the upcoming proof \cite{RI-I, RI-II} of sharpness of the phase transition of $\mathcal V^u$. Moreover, if developed further, some of our quantitative statements, see for instance Proposition~\ref{P:ind-step} below (the one-step version of Theorem~\ref{thm:main III}) will likely further improve  our understanding of the (near-)critical phase.

We now briefly discuss $\rho$-interlacements. Roughly speaking, $\rho$-interlacements and their associated interlacement set $\mathcal{I}^{\rho}$, introduced in Section~\ref{sec:local_coup}, are parametrized in terms of a (time-space) density $\rho =\rho(\ell,x) \geq0$ describing the average number of trajectories of (time-)length $\ell$ (possibly infinite) starting in $x \in \Z^d$. This supplies a framework of processes that subsumes all the models to be dealt with, including the full interlacement set $\mathcal{I}^u$ from \eqref{eq:def-I}, the length-$L$ interlacements $\mathcal{I}^{f,L}$ from \eqref{eq:J} and, in particular, their homogenous version $\mathcal{I}^{u,L}$, as well as various others encountered in practice (recall for instance the environment configuration relevant to the proof of Theorem~\ref{thm:main III}, which is more involved).

As we now briefly outline, a challenge is to accommodate the breadth of applications and the resulting variety of `random environments' $\mathcal{I}=\mathcal{I}(\omega)$ that may arise. To wit, for a complicated configuration $\mathcal{I}=\mathcal{I}(\omega)$ involving trajectories of spatially inhomogenous intensity and varying time length, the prospect of witnessing the event $\mathscr{D}(\omega)$ in \eqref{e:obs-final2} with sufficient probability, which entails exhibiting regions in which disconnection occurs in $\mathcal{V}= \Z^d \setminus \mathcal{I}$ (see \eqref{eq:disc-coup}) can seem daunting; it is, however, pivotal since these regions precisely define the obstacle set $\mathcal{O}=\mathcal{O}(\omega)$, as in the discussion leading up to Theorem~\ref{thm:long_short} or in the application to Theorem~\ref{thm:main III}.
  
As one of the benefits of $\rho$-interlacements, the mean occupation time density $(\bar{\ell}_x^\rho)_{x\in \Z^d}$ associated to $\rho$, see \eqref{eq:occtime}, yields \textit{verifiable} conditions ensuring for instance that one can identify a meaningful \textit{scalar} intensity parameter $u$ describing the (possibly complicated) configuration of trajectories involved in $\mathcal{I}$; cf.~for instance Definition~\ref{def:background} and the proof of Lemma~\ref{L:Cobst-ver}.

\bigskip

It is also instructive to draw comparisons between the couplings developed in this paper and common approaches used for the analysis of other strongly correlated models, such as the related Gaussian free field. In the latter context, stationary decompositions of the field over scales (harnessing the underlying Gaussian structure) with distinct features have a long history,  see for instance \cite{BryGuaMit04, zbMATH05044070, ADAMS2013169, Bau13, zbMATH07199901,zbMATH07322630, bauerschmidt2022discreteI, https://doi.org/10.48550/arxiv.2212.05756}. Such decompositions are appealing from the point of view of renormalisation and very useful in a variety of contexts, see e.g.~\cite{zbMATH05604370, zbMATH07103768, zbMATH07322630, DCGRS20, https://doi.org/10.48550/arxiv.2206.10724} and refs.~therein. Similarly powerful tools are at present inexistent for random walks/interlacements, as Theorem~\ref{thm:main III} vividly illustrates, and the results of this paper can be regarded as a first step in this direction. One distinctive feature is that $\mathcal{I}^u$ corresponds to a `degenerate' limit for excursion sets of occupation times $\{ \ell^u > \alpha\}$, see \cite{MR3163210}, where the threshold $\alpha=0$, which `lacks ellipticity'. 

\subsection{Organization} 
\label{subsec:organization} We now describe the organization of this article.
Section~\ref{s:not} collects some notation and various useful facts about random walk and random interlacements, along with a useful chaining result for couplings. 

Section~\ref{sec:local_coup}  introduces $\rho$-interlacements. After gathering a few generalities in~\S\ref{subsec:class-mu-gen}, we prove in \S\ref{subsec:class-mu-gen} two complementary results, Propositions~\ref{prop1:cube} and~\ref{prop2:cube}, which supply preliminary \textit{local} couplings relating $\mathcal{I}^{\rho}$ and $\mathcal{I}^{u}$ under suitable conditions on $\rho$. Here, \textit{local} is used in the same sense as in the discussion following \eqref{eq:conv-law}
; see also Remark~\ref{R:loc-coup}, which contrasts these couplings with the main results of this paper. As an immediate application, we obtain the convergence in law asserted in \eqref{eq:conv-law}, see Proposition~\ref{P:loclimit}.

Sections~\ref{sec:easyCOUPLINGS} to~\ref{sec:obstacle_set} form the core of this article. Sections~\ref{sec:easyCOUPLINGS} and~\ref{sec:hard_couple_obstacle} are dedicated to the proofs of Theorems~\ref{thm:short_long-intro} and~\ref{subsec:env}, respectively, and have a similar structure. In each case, the proof starts with a reduction step, which incorporates a notion of gaps (\S\ref{ssec:tgap}) or overlaps (\S\ref{subsec:overlaps}) between trajectories, leading to more malleable statements, see Propositions~\ref{thm:short_long'} and \ref{thm:long_short'}. The bulk of each section is devoted to the proof of each proposition. Whereas the former relies on the `cutting + homogenization' approach alluded to above Theorem~\ref{thm:short_long-intro}, the (much) more difficult proof of Proposition~\ref{thm:long_short'} requires implementing a gluing technique of short trajectories, which brings to bear the obstacle set and the conditions for it to be good that constitute Definition~\ref{def:O}. 

The main result of Section~\ref{sec:obstacle_set} is Proposition~\ref{L:final-good-obs}, which supplies a large class of good obstacle sets, including in particular \eqref{eq:ex-good-obs}. This result can be fruitfully combined with Theorem~\ref{thm:long_short_obstacle}, applied with a suitable (random) choice of obstacle set $\mathcal{O}$, thus  leading to Theorem~\ref{thm:long_short}. The proof of the latter appears at the end of Section~\ref{sec:obstacle_set}.

Finally, Section~\ref{sec:VuVuLcouple} is devoted to the proof of Theorem~\ref{thm:main III}, which illustrates Theorems~\ref{thm:short_long-intro} and~\ref{thm:long_short}. Theorem~\ref{thm:main III} is first reduced to its one-step version, Proposition~\ref{P:ind-step}, which yields a similar coupling relating $\mathcal V^{u, L}$ and $\mathcal V^{u, 2L}$, interesting in its own right. Proposition~\ref{P:ind-step} is obtained by combining Theorem~\ref{thm:short_long} (which extends Theorem~\ref{thm:short_long-intro}) with Theorem~\ref{P:long_short}. The latter corresponds to a specialization of Theorem~\ref{thm:long_short} to the case where the environment configuration $\mathcal{I}$ is of the form $\mathcal{I}^{\rho}$ introduced in Section~\ref{sec:local_coup}. This yields very handy conditions, see $(\textnormal{C}_{\textnormal{obst}})$ in Definition~\ref{def:background}, which allow to control $ \mathbf{P}[\mathscr D]$ in Theorem~\ref{thm:long_short}. The remainder of Section~\ref{sec:VuVuLcouple} contains the proof of Theorem~\ref{P:long_short}.

\medskip

Our convention regarding constants is as follows. Throughout the article $c,c',C,C',\dots$ denote 
generic positive constants (i.e.~with values in $(0,\infty)$) which are allowed to change from place to 
place. All constants may implicitly depend on the dimension $d\geq3$. Their dependence on other 
parameters will be made explicit. Numbered constants are fixed upon first appearance within the text. 

\section{Notation and useful facts}
\label{s:not}

\label{sec:prelims}

In this section, we gather a few preliminary results. In \S\ref{sec:RW} we review some facts about random walk and gather some estimates about entrance and exit laws (possibly with finite time horizon); see in particular Lemma~\ref{L:return}. In \S\ref{subsec:RI} we introduce random interlacements,  excursion decompositions and supply a basic coupling for excursions, Lemma~\ref{L:RI_basic_coupling}.
Finally, \S\ref{A:chaining} exhibits a tool, Lemma~\ref{lem:concatenation}, that allows to `concatenate' couplings between random sets with a common marginal, which is useful in order to preserve inclusions; see also Remark~\ref{R:chain},2). 

We start with some notation. We write $\mathbb{N}^\ast=\{1,2,\dots\}$, $\mathbb{N}= \mathbb{N}^\ast \cup \{0\}$ and $\mathbb{R}_+ =[0,\infty)$. We consider the lattice $\Z^d$, $d \geq3$, endowed with the usual (nearest-)neighbor graph structure, and denote by $|\cdot|$ the $\ell^{\infty}$-norm on $\Z^d$. For a set $K \subset \Z^d$, we write $K^c =\Z^d 
\setminus K$ for its complement in $\Z^d$, $\partial K$ for its 
inner vertex boundary, i.e. $\partial K = \{x \in K: \exists y\notin K \text{ s.t. } y \sim x\}$, where $x \sim y$ denote neighbors in $\Z^d$, at Euclidean distance one. 
We also write $\partial_{\text{out}} K = \partial (K^c)$ for the outer (vertex) boundary of $K$ and $\overline{K}=K \cup \partial_{\text{out}} K$. 
The set $B(K, r) = K_{r} = \bigcup_{x \in U} B(x,r)$ denotes the $r$-neighborhood of $K$ and $K \subset \subset \Z^d$ means that $K\subset \Z^d$ has finite cardinality. We use the notations $B_r(z)=B(z,r)$ interchangeably to denote $\ell^{\infty}$-balls with radius $r> 
0$ centered at $z \in \Z^d$ and abbreviate $B_r=B_r(0)$. We use $d(\cdot, \cdot)$ to refer to the $\ell^{\infty}$-distance between sets.

\subsection{Random walk}\label{sec:RW}

We endow $\Z^d$, $d \geq 3$, with symmetric weights $a:\Z^d \times \Z^d \to [0,\infty)$, where $a_{x,y} =a_{y,x}=1$ if $x \sim y$, $a_{x,x}=2d$ and $a_{x,y}=0$ otherwise, and write $a_x=\sum_{y \in \Z^d}a_{x,y} =4d$. We consider the discrete-time Markov chain on $\Z^d$ with generator $Lf(x)= a_x^{-1}\sum_{y\sim x} a_{x,y}(f(y)-f(x))$, for $f: \Z^d \to \R$, which has transition probabilities $p(x,y)= \frac{a_{x,y}}{a_x}$, $x,y \in \Z^d$.
We denote by $P_x$ the canonical law of this chain when started at $x \in \Z^d$, defined on its canonical space $(W_+, \mathcal{W}_+)$, and by $X = (X_n)_{n \geq 0}$ the corresponding canonical process. 
For $\mu: \Z^d \to \R_+$ we write $P_\mu = \sum_{x\in \Z^d} \mu(x)P_x$. With $p_n(x,y)=P_x[X_n=y]$, $x,y, \in \Z^d$, $n \geq 0$, so that $p_1=p$, one has$p_n(x,y)=p_n (0,y-x)$ by translation invariance. We denote by $P_n$ the transition operators, i.e.~for $f:\Z^d \to \R$,
\begin{equation}\label{eq:P_n-def}
P_nf(x)= \sum_y p_n(x,y)
f(y) \, \big( = E_x[f(X_n)]  \big), \quad x \in \Z^d, 
\end{equation}
which satisfies the semigroup property $P_{n+m}=P_{n}P_m$ for $n,m \geq 0$.
The random walk $X$ satisfies the following local central limit theorem estimate, see for instance \cite[Theorem~2.3.11]{LawLim10}: 
\begin{align}
\label{eq:LCLT}
 \bigg| \log \bigg( \frac{p_n(0, x)}{\overline{p}_n(x)} \bigg) \bigg| \le C\bigg(\frac{1}{n} + \frac{|x|^4}{n^3}\bigg) \quad \text{for $|x| < c n$
,}
\end{align}
where
\begin{align}
\label{eq:density}
\overline{p}_n(x) \stackrel{\text{def.}}{=} \left(\frac{d}{\pi n}\right)^{d/2}\exp\Big(-\frac{d|x|^2}{n}\Big).
\end{align}
 For $U \subset \Z^d$, we write $H_U = \inf \{n \geq 0: X_n \in K \}$ for the entrance time in $U$, $T_U = H_{\Z^d \setminus U}$ is the exit time  from $U$ and $\widetilde{H}_U = \inf \{ n \geq 1: X_n \in U\}$ the hitting time of $U$. We denote by $g_U$ the Green's density (with respect to $a_{\cdot}$) killed on $U$, i.e.
  \begin{equation}
\label{eq:Greenkilled}
 g_U(x,y) = \sum_{n \geq 0} a_y^{-1}P_x[X_n =y, \, n < H_U], \ x, y \in \Z^d,
 \end{equation}
and write $g(\cdot,\cdot)=g_{\emptyset}(\cdot,\cdot)$. By 
\cite[Theorem~1.5.4]{Law91}, one has that
\begin{equation}\label{eq:Greenasympt}
c(|x-y| \vee 1)^{2-d} \leq g(x,y)  \leq C (|x-y| \vee 1)^{2-d}, \text{ for all } x,y \in \Z^d.
\end{equation}
More is in fact true but \eqref{eq:Greenasympt} will be sufficient for our purposes. We further define, for $U \subset\subset \Z^d$,
\begin{equation}\label{eq:equilib_K} 
e_U(x)=a_xP_x[\widetilde{H}_U=\infty]1\{x \in U\},
\end{equation}
the equilibrium measure of $U$, which is supported on $\partial U$ and denote by
\begin{equation}\label{eq:cap_K} 
\text{cap}(U) = \sum_x e_{U}(x)
\end{equation}
its total mass, the capacity of $U$, which is increasing in $U$. We write $\bar{e}_U= {e_U}/{ \text{cap}(U)}$ for the normalized 
equilibrium measure. One has the last-exit decomposition, valid for all $U \subset\subset \Z^d$,
\begin{equation}
	\label{eq:lastexit}
	P_x[H_U < \infty] = \sum_y g(x,y) e_{U}(y), \quad x \in \Z^d;
\end{equation}
see, e.g.,~\cite[Lemma~2.1.1]{Law91} for a proof. Together, \eqref{eq:lastexit} and \eqref{eq:Greenasympt} readily imply that there exists $\Cl{hit-box} = \Cr{hit-box}(d) > 0$ such that for all $L \geq 1$ and $x \in \Z^d$ with $|x|>L$,
  \begin{equation}
    \label{eq:hit1}
    P_x [H_{B_L} < \infty] 
    \leq \Cr{hit-box} \textstyle \big(\frac{L}{|x|} \big)^{d-2}.
  \end{equation}
 Moreover, summing \eqref{eq:lastexit} over $x\in U$, one immediately sees that
\begin{equation}\label{eq:cap_bound_Green}
\big(\max_{x\in U}\sum_{y\in U} g(x,y)\big)^{-1}\leq  \frac{ \mathrm{cap}(U)}{|U|}\leq \big({\min_{x\in K}\sum_{y\in U} 
g(x,y)} \big)^{-1},
\end{equation}
where $|U|$ denotes the cardinality of $U$. Along with \eqref{eq:Greenasympt}, \eqref{eq:cap_bound_Green} readily gives for all $L > 0$,
\begin{equation}
\label{e:cap-box}
cL^{d-2} \leq \text{cap}(B_L) \leq C L^{d-2}.
\end{equation}
For later reference, we also record the pointwise lower bound 
\begin{equation}\label{eq:bnd_equil_box}
e_{B_L}(x) \ge 
cL^{-1}\text{ for all } x \in \partial B_L \text{ and }L>0,
\end{equation}
on the equilibrium measure of a box. One way to obtain \eqref{eq:bnd_equil_box} is to enforce the event $\{\widetilde H_{B_L} = \infty\}$ by 
first sending $X$ to distance $CL$ away from the box in a suitable coordinate direction 
and then requiring that $X$ never visits $B_L$ again. The first event has 
probability at least $\frac cL$ by a standard one-dimensional gambler's ruin estimate. On the other hand, by \eqref{eq:Greenasympt} and \eqref{eq:hit1}, the second event has probability at least $c > 0$ under $P_y$ for any $y$ at a distance at least $CL$ from 
$B_L$ provided $C$ is large enough. Applying the strong Markov property, \eqref{eq:bnd_equil_box} follows.

We now collect a few useful hitting probability estimates over the next lemmas. The following result yields pointwise comparison estimates between the equilibrium measure of a ball and either of i) the tail probability of its hitting time and ii) the normalized hitting probability 
measure
\begin{equation}\label{def:normalized_hitting_prob}
\bar h_B(x, y)	\stackrel{\text{def.}}{=} P_x [ X_{\widetilde H_{B}} = y \, | \, \widetilde H_B < \infty], \quad B \subset \Z^d, \, x, y \in \Z^d.
\end{equation}
\begin{lem}[$\xi >0$] 
  \label{L:return} 
  
 \medskip
\noindent \begin{itemize}
  \item[i)]There exists $\Cl{c:escape} = \Cr{c:escape}(\xi) > 0$ such that 
  for any sequence $ \ell_N \geq N^{2+\xi}$, $N \geq 1$ and for every $x \in \partial B$ with $B = B_N$, one has
  \begin{equation}
    \label{eq:return1}
    e_B(x) \leq  a_xP_x[\widetilde{H}_B > \ell_N] \leq (1 + \Cr{c:escape} N^{-\xi/3}) e_B(x).
  \end{equation}
  \item[ii)] There exists $\Cl{C:norm_equil} = \Cr{C:norm_equil}(\xi) > 0$ such that for all $N \geq 1$, $y \in B = B_N$ and $z \in \Z^d \setminus B_{N^{1 + \xi}}$, 
\begin{equation}
\label{normalized_eB}
(1 - \Cr{C:norm_equil} N^{-\xi}) \bar e_B(y) \leq  \bar h_B(z, y) \leq (1 + \Cr{C:norm_equil} N^{-\xi}) \bar e_B(y).
	\end{equation}
  \end{itemize}
\end{lem}
\begin{proof}We first show $i)$.
  The first inequality in \eqref{eq:return1} is immediate. For the second one, abbreviating $\ell=\ell_N$, one writes
  \begin{equation}
    \label{eq:return2}
    P_x[\widetilde{H}_B > \ell] =a_x^{-1}  e_{B}(x) + P_x[\ell < \widetilde{H}_B < \infty].
  \end{equation}
  Defining $\tilde{B} = B(0,N^{1 + \xi/3})$, one estimates the second term in \eqref{eq:return2} by
  \begin{equation}
    \label{eq:return3}
    P_x[\ell \leq \widetilde{H}_B < \infty] \leq P_x[T_{ \tilde{B}} > \ell] + q \text{ where } q = P_x[T_{ \tilde{B}} \leq \widetilde{H}_B < \infty].
  \end{equation}
One knows that, for suitably small $c\,(>0)$,
\begin{equation}  
\label{e:exittime-mgf}
\sup_{x \in B_L} E_x\big[\textstyle\exp\big\{\frac{cT_{B(0,L)}}{L^2}\big\}\big] \leq 2, \ L \geq 1.
\end{equation}
 Applying \eqref{e:exittime-mgf} with $L= N^{1 + \xi/3}$ readily yields that
 \begin{equation}\label{eq:long_excursion} a_x P_x[T_{ \tilde{B}} > \ell] \leq a_x\exp\{ - cN^{\xi/3}\} \leq \exp\{ - c'N^{\xi/3}\} e_B(x),\end{equation}
 where the last bound follows as $e_B(x) \geq \frac cN$, see \eqref{eq:bnd_equil_box}. 
 All that is left is to estimate is $q$. Applying the strong Markov property, one has
  \begin{multline}
    \label{eq:bound_q}
    q  = E_x \Big[ T_{ \tilde{B}} \leq \widetilde{H}_B, P_{X_{T_{ \tilde{B}}}}[\widetilde{H}_B < \infty] \Big]
  \\  \leq \sup_{y \in \tilde{B}^c} P_y [\widetilde{H}_B < \infty]
    P_x [ T_{ \tilde{B}} \leq \widetilde{H}_B ] 
     \overset{\eqref{eq:hit1}}\leq \Cr{hit-box} N^{-\frac{\xi(d - 2)}{3}} P_x [ T_{ \tilde{B}} \leq \widetilde{H}_B ].
  \end{multline}
Moreover, $q$ can also be written as $q = P_x[T_{ \tilde{B}} \leq \widetilde{H}_B] - P_x [\widetilde{H}_B = \infty]$. Substituting this on the left-hand side of \eqref{eq:bound_q} and rearranging terms yields that
  \begin{equation*}
    P_x [T_{ \tilde{B}} \leq \widetilde{H}_B] \leq (1 - \Cr{hit-box}N^{-(\xi/3)(d-2)})^{-1} P_x [\widetilde{H}_B = \infty] \leq  C(\xi) e_B(x).
  \end{equation*}
  Plugging this into \eqref{eq:bound_q} readily gives 
  \begin{equation}\label{eq:bound-q-final}
   q \leq C'(\xi)  
  N^{-(\xi/3)(d-2)} a_x^{-1}  e_B(x).
  \end{equation}
   Applying this bound along with \eqref{eq:long_excursion} in \eqref{eq:return3} yields the second inequality in \eqref{eq:return1}.
  
  We now show $ii)$. For $U\supset B (=B_N)$, let $e_{B,U}(x)= a_xP_x[\widetilde{H}_B> T_U]1\{x \in B\}$ denote the equilibrium measure of $B$ relative to $U$ (so $e_B=e_{B,\Z^d}$, cf.~\eqref{eq:equilib_K}), and $\bar e_{B,U}(\cdot)$ be its normalized version, a probability measure on $\Z^d$. With $U= B(0,N^{1 + \xi})$, Theorem~2.1.3 in \cite{Law91} gives 
\begin{equation}\label{eq:barHA_identity}
\Big|  \frac{\bar h_B(z, y)}{\bar{e}_{B,U}(y)} -1 \Big| \leq C N^{-\xi},
\end{equation}
for all $y \in \partial B$ and $z \in \Z^d \setminus B_{N^{1 + \xi}}$.
In order to show that $\bar{e}_{B,U}(y)$ is comparable to $\bar e_B(y)$, write
\begin{equation*}
a_x^{-1}e_{B,U}(x)= P_{x}[\tilde H_B = \infty] + P_{x}[T_U < \tilde H_B < \infty] =  a_x^{-1}e_B(x) + q
\end{equation*}
where $q$ is as in \eqref{eq:return3} but with $U = B(0,N^{1 + \xi})$ in place of $\tilde B=  B(N^{1 + \xi/3}) $. Using the 
bound on $q$ from \eqref{eq:bound-q-final}, one gets $\big| \frac{\bar{e}_{B,U}(x)}{\bar e_B(x)} -1 \big| \leq C(\xi) N^{-\xi}$ uniformly in $x \in \partial B$. Combined with \eqref{eq:barHA_identity}, this yields \eqref{normalized_eB}, thus completing the proof.
\end{proof}

The next result deals with regularity of exit distributions in the starting point.
\begin{lemma}
For all $L \geq 1$, $B=B_L$, and $z \in \partial_{{\rm out}} B$, letting $\pi_B(x,z) = P_x[X_{T_{B}}=z]$, one has
\begin{equation}
\label{e:exit-dist1}
\Big| \frac{\pi_B(x,z)}{\pi_B(y,z)} -1 \Big| \leq  \frac{C|x-y|}{L}, \quad  x,y \in B_{L/2}.
\end{equation}
\end{lemma}
\begin{proof} Let $B'= B_{L/2}$.
For fixed $z$, the function $\pi_B(\cdot,z): \overline{B'} \to [0,\infty)$ is harmonic in $B'$. Hence, by repeated use of a gradient estimate, see Theorem 1.7.1 in \cite{Law91}, one obtains for any $x,y \in B'$,
\begin{equation}
\label{e:exit-dist2}
\big| \pi_B(x,z) - \pi_B(y,z) \big| \leq  C L^{-1} |x-y| \Vert \pi_B(\cdot,z) \Vert_{\ell^{\infty}(B')}.
\end{equation}
By the Harnack inequality applied in $B'$ one knows that
\begin{equation}
\label{e:exit-dist3}
\Vert \pi_B(\cdot,z) \Vert_{\ell^{\infty}(B')} \leq C' \pi_B(x,z).
\end{equation}
Together, \eqref{e:exit-dist2} and \eqref{e:exit-dist3} readily yield \eqref{e:exit-dist1}.
\end{proof}

\subsection{Random interlacements}
\label{subsec:RI}
The interlacement point process $\eta^*$ is a Poisson point process on the space $W^* \times \R_+$, $W^* = W/\sim$, of labeled bi-infinite trajectories modulo time-shift (denoted by $\sim$), see e.g.~\cite{MR2525105} for definitions in the present context of the weighted graph $(\Z^d,a)$, cf.~\S\ref{sec:RW}. Let $W^*_U \subset W^*$ denote the set of trajectories visiting~$U \subset \mathbb{Z}^d$ and $\pi^\ast: W\to W^\ast$ denote the canonical projection corresponding to $\sim$.
The intensity measure of $\eta^*$ on $W^* \times \mathbb{R}_+$ is given by $\nu_{\infty}(\mathrm{d}w^*) \mathrm{d}u$, where $\mathrm{d}u$ denotes the Lebesgue measure and the measure $\nu_{\infty}$ on $W^*$ is specified by requiring that $1_{W^*_U} \nu_{\infty} = \pi^* \circ Q_U$, for all $U \subset \subset \Z^d$,  where $Q_U$ is the finite measure on $W$ with 
\begin{equation}\label{e:RI-intensity}
Q_U[(X_{-n})_{n \geq 0} \in {A}, \, X_0 =x , \, (X_{n})_{n \geq 0} \in {A}' ]= P_x[ {A} \, | \, \widetilde{H}_U =\infty \,]e_U(x)P_x[ {A}'],
\end{equation}
 for all $x \in \Z^d$ and ${A}, {A}' \in \mathcal{W}^+$, with $e_U$ as in \eqref{eq:equilib_K}.
We write $(\Omega^*, \mathcal{A}^*, \P)$ for the canonical space of the interlacement point process. Given a realization $\eta^* \in \Omega^*$, the interlacement set $\mathcal{I}^u(\eta^*)$ at level $u > 0$ is defined as in \eqref{eq:def-I}, and $\mathcal{V}^u = \Z^d \setminus \mathcal{I}^u$ is the corresponding vacant set. The parameter $u$, which controls the number of trajectories entering the picture, can for instance becharacterised in terms of the occupation time field $\ell^u= (\ell_x^u)_{x\in \Z^d}$, where for $x \in \Z^d$,
\begin{equation}
\label{e:ell-u}
\ell_x^u(\eta^*) = a_x^{-1} \sum_{i} \sum_{n \in \Z} 1\{ w_i(n)=x, \, u_i \leq u\}, \,\
\end{equation}
if $\eta^*=\sum_i \delta_{(w_i^*, u_i)}$, with $w_i$ such that $\pi^*(w_i)=w_i^*$. One then has that\begin{equation}
\label{e:ell-u-mean}
\E[\ell_x^u] =u, \, \text{ for all } x \in \Z^d;
\end{equation}
i.e.~$u$ is the average number of visits at $x$ by any of the trajectories with label at most $u$. The perspective \eqref{e:ell-u-mean} will be useful later, in order to associate a meaningful scalar parameter $u$ to a (possibly complicated) model comprising trajectories of different types (e.g.~varying length).

We now set up the framework to decompose trajectories into excursions. We then present a straightforward coupling result for the excursions associated to $\mathcal{I}^u$, which will be useful in the next section.
We assume henceforth that for \textit{any} realization $\eta^* =\sum_{i \geq 0} \delta_{(w_i^*, u_i)} \in \Omega^* $, the labels $u_i$, $i \geq 0$, are pairwise distinct, that $\eta^*(W_A^*\times[0,u])<\infty$ for all $u \geq 0$ and that $\eta^*(W_A^*\times \R_+)=\infty$, which is no loss of generality since these sets have full $\P$-measure.

Let $A ,U$ be finite subsets of $\mathbb{Z}^d$ with 
$\emptyset \neq A \subset U$. The infinite, resp.~doubly infinite 
transient trajectories, elements of $W^+$, resp.~$W$, are split into excursions between $A$ and $\partial_{{\rm out}} U$ by introducing the following successive return and departure times 
between these sets. Let $D_0 =0$ and $R_k = D_{k-1} + H_A \circ \theta_{D_{k-1}}$, $D_k = R_k + T_U \circ \theta_{R_k}$, for $k \geq 1$, where all of $D_k, R_j, D_j$, $j > k$ are understood to be $=\infty$ whenever $R_k=\infty$ for some $k \geq 0$.
We denote by $W_{A, \partial_{{\rm out}} U}^+$ the set of all excursions between $A$ and $\partial_{{\rm out}} U$, i.e.~all finite trajectories starting in $A$, ending in $\partial_{{\rm out}}U$ and not exiting $U$ in between. 
Given $\eta^*=\sum_{i \geq 0} \delta_{(w_i^*, u_i)}$, we order all the excursions from $A$ to $\partial_{{\rm out}} U$, first by increasing value of $\{u_i: w_i^* \in W_A^*\}$, 
then by order of appearance within a given trajectory $w_i^* \in W_A^*$. This yields a sequence of 
$W^+_{A, \partial_{{\rm out}} U}$-valued random variables under $\P$, encoding the successive 
excursions,
\begin{equation}
\label{eq:RI_Z}
\big( Z^{A, U}_n(\eta^*) \big)_{n \geq 1} = \big( w_0[R_1, D_1], \dots, 
w_0[R_{N^{A, U}}, D_{N^{A, U}}],\, w_1[R_1, D_1], \dots \big),
\end{equation}
where $N^{A,U}= N^{A,U}(w_0^*)$ is the total number of excursions from $A$ to $\partial_{{\rm out}} U$ in $w_0^*$, i.e. $N^{A, U}(w_0^*) = \sup\{j : D_j(w_0) < \infty\}$ and $w_0$ is any point in the equivalence class $w_0^*$. We will omit the superscripts $A,U$ whenever no risk of confusion arises.

We now proceed to couple the excursions \eqref{eq:RI_Z} induced by the interlacements with a suitable family of i.i.d.~excursions between $A$ and $\partial_{\rm out} U$. The corresponding result appears in Lemma~\ref{L:RI_basic_coupling} below. For $x\in \Z^d$, let $Q_x$ be the joint law of two independent simple random walks $X^1$, $X^2$ on $\Z^d$, respectively sampled from $P_x$ and from $P_{\bar{e}_A}$, the normalized equilibrium measure on $A$, see below \eqref{eq:cap_K} for notation.
Define
\begin{equation}
\label{eq:RI_Y}
Y = \begin{cases}
X^1 \circ \theta_{H_A^1}, \quad & \text{if $H_A^1\stackrel{\text{def.}}{=}\inf\{ n \geq 0: X^1_n \in A\} < \infty$}\\
X^2, & \text{otherwise,}
\end{cases}
\end{equation}
and observe that by the Markov property for the simple random walk and the defining properties of $\eta^*$, the law of $(Z_n(\eta^*))_{n \geq 1}$ is characterized as follows: 
$\P[Z_1 = w] = P_{\bar{e}_A}[ X_{[0, T_U]} = w]$ for all $w \in W^+_{A, \partial_{{\rm out}} U}$, and for all $n \geq 2$,
\begin{equation}
\label{eq:RI_Zlaw}
\P[Z_n = w \, | \, Z_1,\dots, Z_{n-1}] = Q_{Z_{n-1}^{\text{end}}}[\,Y_{[0, T_U]}  =  w], \text{ for all $w \in \partial W^+_{A, \partial_{{\rm out}} U}$},
\end{equation}
where $Z_{n-1}^{\text{end}} \in \partial_{{\rm out}} U$ is the endpoint of $Z_{n-1}$.

Let $\mu$ be a finite positive measure supported on $ A$ and write $\bar{\mu}$ for the 
normalized probability measure $\mu / \mu(A)$. The following simple result supplies a coupling between the 
(dependent) sequence $Z= (Z_n)_{n \geq 1}$ (under $\P$) given by \eqref{eq:RI_Z} and an i.i.d.~sequence of random 
variables $\widetilde{Z}= (\widetilde{Z}_n)_{n \geq 1}$ having common distribution $P_{\bar 
\mu}[X_{[0,T_U]} = \cdot \,]$ in a way that certain inclusions, tailored to our later purposes, hold with high probability under 
suitable assumptions on ${\mu}$. In the sequel, we let $Z_v = Z_{\lfloor v \rfloor \vee 1}$ for $v\in \R_+$.

\begin{lem}[Coupling $Z$ and $\widetilde{Z}$]
  \label{L:RI_basic_coupling} For all sets $A,U$ with $ \emptyset \neq A \subset U \subset \subset \Z^d$, there exists a probability measure $\mathbb{Q} = \mathbb{Q}_{A,U}$ with the following property. If, for some $\delta \in (0,1)$,
\begin{align}
&Q_y[Y_0 = x] \leq \textstyle \big(1+\frac{\delta}{4}\big)  \bar{e}_A(x), \text{ for all $y\in \partial_{{{\rm out}}} U$ and $x \in \partial A$} \label{eq:RI_cond_Q}
\end{align}
and $\mu$ is some 
(finite) measure supported on $A$ satisfying
\begin{align}
&\mu(x) \geq \textstyle \big(1-\frac{\delta}{4}\big)  e_A(x), \text{ for all $x \in A$,} \label{eq:RI_cond_mu}
\end{align}
then for every $\alpha \in (0,\infty)$ there exists an event $\mathcal{U}_{\alpha}$ with
\begin{equation}
\label{eq:RI_basic_coupling1}
\mathbb{Q}[\mathcal{U}_{\alpha}^{\mathsf{c}}] \leq C e^{-c \delta \alpha \textnormal{cap}(A)}
\end{equation}
and $\mathbb{Q}$ carries a coupling of the sequences $Z$ and $\widetilde{Z}$ such that, on $\mathcal{U}_{\alpha}$, for all $\alpha' \geq \alpha$,
\begin{equation}
\label{eq:RI_basic_coupling2}
 \{ Z_1,\dots Z_{(1-\delta)\alpha' \textnormal{cap}(A)} \} \subset  \{ \widetilde{Z}_1,\dots \widetilde{Z}_{(1+\delta)\alpha' \mu(A)} \}.
\end{equation}
\end{lem}


\begin{proof}
We use a version of the soft local time technique from \cite{PopTeix}. We consider, under a suitable probability $\mathbb{Q}$, a Poisson point 
process $\eta=\sum_{\lambda} \delta_{(z_{\lambda},u_{\lambda})}$ on 
$W_{A, \partial_{{\rm out}} U}^+ \times \R_+$ with intensity measure \begin{equation*}\nu (D \times [0,v])= v \sum_{x \in \partial A} P_x[X_{[0, 
T_U]} \in  D ]\, \end{equation*}
for $v>0$ and measurable set $D \subset W_{A, \partial_{{\rm out}} U}^+$, and introduce the random variables (functions of $\eta$)
\begin{equation}
\label{eq:RI_basic_coupling3}
\begin{split}
&\xi_1 = \inf\{ s \geq 0: \exists \lambda \text{ s.t. } s \, \bar{e}_A(z_{\lambda}(0)) \geq u_{\lambda} \},\\
&G_1(x)= \bar{e}_A(x) \xi_1, \text{ for } x \in \partial A.
\end{split}
\end{equation}
Let $(z_1,u')$ denote the $\mathbb{Q}$-a.s.~unique pair among the points $ (z_{\lambda},u_{\lambda})$ in $\text{supp}(\eta)$ such that $\xi_1  \bar{e}_A(z_1(0))$ $= 
u'$. 
For $n \geq 2$, one defines recursively (recall \eqref{eq:RI_Y}, \eqref{eq:RI_Zlaw})
\begin{equation}
\label{eq:RI_basic_coupling4}
\begin{split}
&\xi_n = \inf\big\{ s \geq 0: \exists (z_\lambda, u_{\lambda}) \notin \{(z_k,u_k)\}_{ 1\leq k < n }\text{ s.t. }  G_{n-1}(z_{\lambda}(0))+ s \, Q_{z_{n-1}^{\text{end}}}[Y_0 = z_{\lambda}(0)] \geq u_{\lambda} \big\},\\
&G_n(x)= G_{n-1}(x) +  \xi_n Q_{z_{n-1}^{\text{end}}}[Y_0 = x], \text{ for } x \in \partial A.
\end{split}
\end{equation}
Similarly, one defines sequences $(\tilde{\xi}_n)_{n \geq1}$, $(\tilde{z}_n, \tilde{u}_n)_{n \geq 1}$ and $(\widetilde{G}_n(\cdot))_{n \geq1}$, replacing all occurrences of $\bar{e}_A(\cdot)$ by $\bar{\mu}(\cdot)$ in \eqref{eq:RI_basic_coupling3} and
$Q_{z_{n-1}^{\text{end}}}[Y_0 =\cdot]$ by $\bar{\mu}(\cdot)$ in \eqref{eq:RI_basic_coupling4}. 

In view of \eqref{eq:RI_Zlaw}, it then follows by Propositions 
4.1 and 4.3 in \cite{PopTeix} that for all $n \geq 1$, the sequence $(z_1,\dots,z_n)$ has the 
same law under $\mathbb{Q}$ as $(Z_1,\dots, Z_n)$ under $\P$ and is independent of 
$(\xi_1,\dots,\xi_n)$, which are i.i.d.~mean one exponential random variables. Similarly, $(\tilde{z}_1, \dots, \tilde{z}_n) \stackrel{\text{law}}{=} (\widetilde{Z}_1, 
\dots, \widetilde{Z}_n)$, which is independent of $(\tilde\xi_1,\dots,\tilde\xi_n)$, another 
sequence of i.i.d.~mean one exponential variables. In particular $\mathbb{Q}$ provides a 
coupling between $Z$ and $\widetilde{Z}$ and by construction (see \cite{PopTeix}, Corollary 
4.4), for any $0\leq v_1 \leq v_2$,
\begin{equation}
\label{eq:RI_basic_coupling5}
\big\{ G_{v_1}(x) \leq \widetilde{G}_{v_2}(x) \text{ for all $x \in \partial A$} \big\} \subset \big\{ \{ Z_1,\dots Z_{v_1} \} \subset  \{ \widetilde{Z}_1,\dots \widetilde{Z}_{v_2} \} \big\},
\end{equation}
where $G_v(\cdot) = G_{\lfloor v \rfloor \vee 1}(\cdot)$ for $v\in \R_+$.

For the remainder of the proof, let $v_1= (1-\delta)\alpha' \textnormal{cap}(A)$ and $v_2= (1+\delta)\alpha' \mu(A)$. Note that $v_1$ and $v_2$ depend implicitly on $\alpha' >0$ and that $v_1 \leq v_2$ by \eqref{eq:RI_cond_mu}. We proceed to define an event $\mathcal{U}_{\alpha}$ satisfying \eqref{eq:RI_basic_coupling1} which will imply the event appearing on the left-hand side of \eqref{eq:RI_basic_coupling5} for any $\alpha' \geq \alpha$, thus completing the proof. Let $\mathcal{P}_t =\sup \{ k \geq 0: \sum_{n=1}^k \xi_n \leq t\}$ and define $(\widetilde{\mathcal{P}}_t)_{t \geq 0}$ similarly, with $\tilde{\xi}_n$ in place of $\xi_n$. Thus $(\mathcal{P}_t)_{t \geq 0}$ and $(\widetilde{\mathcal{P}}_t)_{t \geq 0}$ are each Poisson counting processes with unit intensity, vanishing at time $t=0$. Define
\begin{equation}
\label{eq:RI_basic_coupling6}
\mathcal{U}_{\alpha} = \big\{ \forall \alpha' \geq \alpha: \mathcal{P}_{(1-\frac{\delta}{2}) \alpha' \textnormal{cap}(A)} \geq v_1 \text{ and }  \widetilde{\mathcal{P}}_{(1+\frac{\delta}{2}) \alpha' \mu(A)} \leq v_2 \big\}.
\end{equation}
By a union bound, standard large-deviation estimates for Poisson variables, and using that $\mu(A) \geq c \,\text{cap}(A)$, which follows from \eqref{eq:RI_cond_mu}, one sees that the event $\mathcal{U}_{\alpha}$ 
defined in \eqref{eq:RI_basic_coupling6} satisfies \eqref{eq:RI_basic_coupling1}. Moreover, when 
$\mathcal{U}_{\alpha}$ occurs, using that $ (1+\frac{\delta}{4}) (1-\frac{\delta}{2}) \leq  (1-\frac{\delta}{4}) 
(1+\frac{\delta}{2})  $, which follows since $x\in\R_+ \mapsto \frac{1-x}{1+x}$ is decreasing, one obtains for all $x \in\partial A$
\begin{equation*}
\begin{split}
G_{v_1}(x)&\stackrel[\eqref{eq:RI_basic_coupling4}]
{\eqref{eq:RI_basic_coupling3}}{=}  \bar{e}_A(x) \xi_1 + \sum_{2\leq n\leq v_1} \sum_{y \in \partial_{\rm out} U}P_{Z_{n-1}}[X_{T_U} =y]Q_y[Y=x] \xi_n\\
&\stackrel{\eqref{eq:RI_cond_Q}}{\leq} \Big(1+\frac{\delta}{4}\Big)  \bar{e}_A(x) \sum_{1\leq n \leq v_1} \xi_n \stackrel{\eqref{eq:RI_basic_coupling6}}{\leq}  \Big(1+\frac{\delta}{4}\Big) \Big(1-\frac{\delta}{2}\Big) \alpha'  {e}_A(x)\\
& \stackrel{\eqref{eq:RI_cond_mu}}{\leq}  \Big(1+\frac{\delta}{2}\Big) \alpha'  \mu(A) \bar{\mu}(x) \stackrel{\eqref{eq:RI_basic_coupling6}}{\leq} \bar{\mu}(x) \sum_{1\leq n \leq v_2} \tilde{\xi}_n \stackrel{\eqref{eq:RI_basic_coupling4}}{=}\widetilde{G}_{v_2}(x),
\end{split}
\end{equation*}
as claimed.
\end{proof}
\subsection{Chaining of couplings}
\label{A:chaining}
We conclude this preliminary section with a simple result used repeatedly throughout the text, see Lemma~\ref{lem:concatenation} below, to the effect of concatenating two (or more) couplings having 
one marginal in common. The following setup 
will be more than sufficient for our purposes. Let $X$ and $Y$ be two random variables defined on the 
same probability space $(\Omega, \mathcal A, \mathbb Q)$, with $X$ taking values in some Polish space $E_X$, equipped with its Borel $\sigma$-algebra $\mathcal E_X$, and $Y$ in the measure space $(E_Y, \mathcal{E}_Y)$. Although this is not needed for what follows, in practice all relevant random variables will be in $L^1$. 

We first briefly recall the following central aspects of regular conditional distributions when conditioning on $Y$ in the above setup. One knows (see, e.g.~\cite[Theorems~8.36 and 8.37]{KlenkeA20} for a proof) that there exists a map
\begin{equation}
	\label{a:chain1}
	\kappa_{X,Y}: E_Y \times \mathcal{E}_X \to [0,1], 
\end{equation}
with the following properties: i) for each $A \in \mathcal{E}_X $, the map $y \mapsto \kappa_{X,Y}(y,A)$ is $\mathcal{E}_Y$-measurable, ii) for every $y \in E_Y$, $A \mapsto \kappa_{X,Y}(y,A)$ is a probability measure on $(E_X,  \mathcal{E}_X$ and iii) $\kappa_{X,Y}$ is a version of the conditional distribution, that is, for all $B \in \mathcal{E}_Y$,
\begin{align}
	& \int_{B} \kappa_{X,Y}(y,A)  ( \mathbb{Q} \circ Y^{-1}) (dy) = \mathbb{Q}[X \in A, \, Y \in B ]. \label{a:chain2}
\end{align}
In particular, \eqref{a:chain2} immediately yields that
\begin{align}\label{a:chain3}
	&\begin{array}{l}\text{if $\kappa(\cdot) \equiv \kappa_{X,Y}(y,\cdot)$ does not depend on $y \in E_Y$,}\\
		\text{then $\kappa \stackrel{\text{law}}{=} \mathbb{Q} \circ X^{-1}$ and $X$ and $Y$ are independent.} \end{array}
\end{align}
We now proceed to concatenate two (coupling) measures $\mathbb Q_1$ and $\mathbb Q_2$ having a common marginal.
For simplicity, all the random variables appearing in the next lemma are tacitly assumed to take values in (possibly different) Polish spaces equipped with their respective Borel $\sigma$-algebra.
\begin{lemma}[Chaining of couplings]\label{lem:concatenation}
	Let $(X, Y)$ and $(Y', Z)$ be pairs of random variables defined on $(\Omega_1, A_1, \mathbb Q_1)$ and 
	$(\Omega_2, \mathcal A_2, \mathbb Q_2)$ respectively such that $Y \stackrel{{\rm law}}{=} Y'$. Then there 
	exists a probability space $(\Omega, \mathcal 
	A, \mathbb Q)$ carrying a triplet of random variables $(X'', Y'', Z'')$ such that 
	\begin{equation}\label{a:chain3.5}
		\text{$(X'', Y'') \stackrel{{\rm law}}{=} (X, Y)$ and $(Y'', Z'') \stackrel{{\rm law}}{=} (Y', 
			Z)$.}
	\end{equation}
	In particular, $\mathbb Q$ is a coupling of (the laws of) $X$, $Y$ and $Z$. 
\end{lemma}
We refer to the marginal of $\mathbb Q$ on $(X'', Z'')$, which is a 
coupling between the laws of $X$ and $Z$, as a coupling obtained by {\em concatenating} $\mathbb 
Q_1$ and $\mathbb Q_2$ (but see Remark~\ref{R:chain},1)).

\begin{proof}
	Suppose that $X, Y$ (hence $Y'$) and $Z$ take values in $(E_X, \mathcal E_X)$, $(E_Y, \mathcal E_Y)$ and 
	$(E_Z, \mathcal E_Z)$ respectively. Define the 
	measurable space $(\Omega, \mathcal A) = (E_1 \times E_2 \times E_3, \mathcal E_1 
	\otimes \mathcal E_2 \otimes E_3)$, the random variables $X'',Y'',Z''$ to be the projections on the first, second and third coordinate, respectively, and the probability measure $\mathbb Q$ as follows. For any $A_i \in \mathcal E_i$, $i=1,2,3$,
	\begin{equation}\label{a:chain6}
		\mathbb Q [X'' \in A_1, \, Y'' \in A_2, \, Z'' \in A_3] \stackrel{\text{def.}}{=} \int_{A_1 \times A_2} \kappa_{Z,Y'}(y, A_3) \, (\mathbb Q_1 \circ (X, Y)^{-1}) (dx, dy)
	\end{equation}
	where $\kappa_{Z,Y'}$ refers to the regular conditional distribution of $Z$ given $Y'$ under $\mathbb{Q}_2$, cf.~\eqref{a:chain1}. It follows directly from \eqref{a:chain6} and \eqref{a:chain2} by integrating over various subsets of the coordinates that $\mathbb Q$ is a probability measure satisfying \eqref{a:chain3.5}. 
\end{proof}

\begin{remark} \label{R:chain}
\begin{enumerate}[label*=\arabic*)] \label{R:chain}
		\item The conclusions of Lemma~\ref{lem:concatenation} do not uniquely characterize $\mathbb{Q}$. Another measure (on the same space $(\Omega, \mathcal{A})$) with the same properties is given by replacing the right-hand side of \eqref{a:chain6} by $\int_{A_2 \times A_3} \kappa_{X,Y}(y, A_1) \, (\mathbb Q_2 \circ (Y, Z)^{-1}) (dy, dz)$.

\item \label{R:chain-appli} A typical application of Lemma~\ref{lem:concatenation} is as follows. Suppose $X,Y, Z$ are random sets (subsets of $\Z^d$, say). If for some $\varepsilon_1, \varepsilon_2 \geq 0$ (possibly $=0$),
		\begin{equation}
			\mathbb Q_1[X \subset Y] 
			\geq 1- \varepsilon_1, \  \mathbb Q_2[Y' \subset Z ] \geq 1- \varepsilon_2,
			\label{a:chain4}
		\end{equation}
		then for $\mathbb Q$ as supplied by Lemma~\ref{lem:concatenation}, by \eqref{a:chain3.5} and a union bound, one obtains that
		\begin{equation}
			\mathbb Q[X'' \subset Z'' ] \geq 1- \varepsilon_1 - \varepsilon_2.
			\label{a:chain5}
		\end{equation}
		Loosely speaking, in view of \eqref{a:chain4}-\eqref{a:chain5}, the concatenation $\mathbb{Q}$ preserves inclusions with high probability.
	\end{enumerate}
\end{remark}

\section{Random $\rho-$interlacements and local couplings}\label{sec:local_coup}
We now introduce a framework of 
interlacement processes with trajectories of varying spatial intensity and (time-)length, parametrized by an intensity measure $\rho$, 
see \eqref{e:this-is-rho} below. We call these $\rho$-interlacements. The corresponding interlacement set $\mathcal{I}^{\rho}$, see \eqref{eq:prelim2}, allows in principle for (forward) trajectories of any length started anywhere in space. 
In particular, it can be used to describe both the usual interlacement set $\mathcal{I}^u$, see 
\eqref{eq:def-I}, as well as the finite range models $\mathcal{I}^{f,L}$ from \eqref{eq:J}, but the measure $\rho$ allows for more flexibility, which will be needed in due time; recall for instance the discussion below Theorem~\ref{thm:main III}, which involves a choice of random environment $\mathcal{I}$ that can be quite involved (e.g.~non-homogenous). 

After introducing $\rho$-interlacements in \S\ref{subsec:class-mu-gen} and gathering a few 
generalities, including a simple but important re-rooting property, see Lemma~\ref{L:reroot}, we develop in \S\ref{subsec:loc-coup} two couplings, see Propositions~\ref{prop1:cube} and~\ref{prop2:cube}, which provide conditions on $\rho$ under 
which the induced interlacement set  $\mathcal{I}^{\rho}$ can be \textit{locally} 
compared to a full interlacement $\mathcal{I}^u$ at suitable intensity $u>0$. 
 Each proposition yields 
one of two possible inclusions. In particular, the mean occupation time density $\bar \ell^\rho=(\bar \ell_{x}^\rho)_{x\in \Z^d}$ corresponding to $\mathcal{I}^{\rho}$, introduced in \eqref{eq:occtime} below, plays a key role in associating a scalar parameter $u>0$ to $\mathcal{I}^{\rho}$ and facilitating a comparison with $\mathcal{I}^u$, see \eqref{eq:cubecond1.1} and \eqref{eq:cubecond2.2}. The mean occupation time density $\bar \ell=\bar \ell^\rho$ will also figure prominently later on and allow to formulate within more complicated setups stringent conditions on the `environment' (recall \S\ref{subsec:env}) that can nonetheless be verified with bounded effort, see for instance \eqref{eq:disconnect_background0} and (the proof of) Lemma~\ref{L:Cobst-ver}.

Returning to matters in the present section, similarly as in the discussion following \eqref{eq:conv-law}, the attribute `local' in the context of Propositions~\ref{prop1:cube} and~\ref{prop2:cube} below refers to the fact that the smallest length  scale 
$\ell$ in the support of $\rho$ satisfies $\ell \gg N^2$, where $N$ denotes the linear size of the box in 
which the coupling is constructed; we refer to Remark~\ref{R:loc-coup} for more on this. Roughly speaking, Propositions~\ref{prop1:cube} and~\ref{prop2:cube} are the best one can hope for when adapting available techniques to the present framework and pushing them to their limits, which already requires some efforts owing to the generality of our setup.

With Propositions~\ref{prop1:cube} and~\ref{prop2:cube} at our disposal, we focus in \S\ref{subsec:fr-models} on a case in point, the finite range models $\mathcal{I}^{u,L}$ mentioned in the introduction, see \eqref{eq:J}, and use these 
results to prove that their local limits as $L \to \infty$ is indeed $\mathcal{I}^u$, as asserted in \eqref{eq:conv-law}; see 
Proposition~\ref{P:loclimit}.

\subsection{Generalities}
\label{subsec:class-mu-gen}
Consider a (density) function 
\begin{equation}\label{e:this-is-rho}
\rho:(\mathbb{N}^\ast\cup\{\infty \}) \times \Z^d \to \mathbb{R}_{+}
\end{equation} 
(recall that $\mathbb{N}^\ast = \{1,2\dots\}$). Intuitively, $\rho(\ell, x)$ gives the intensity of 
trajectories that have length $\ell$ and start at $x$. We often think of $\rho$ as a measure on $(\mathbb{N}^\ast\cup\{\infty \}) \times \Z^d$, or on any of its factors rather than as a function, and not distinguish between the two. For instance, we routinely write $\rho (A, 
x)=\sum_{\ell \in A } \rho(\ell,x)$, for $A \subset \mathbb{N}^\ast$ etc.~in the sequel.

Recall the measurable space $(W^+, \mathcal{W}^+)$ from \S\ref{sec:RW} on which $P_x$, $x\in \Z^d$ is defined. For $\rho$ as in \eqref{e:this-is-rho}, we introduce a Poisson point process $\eta$ on the space $(\mathbb{N}^\ast\cup\{\infty \}) 
 \times W^+$ with intensity measure $\nu_\rho$ given by
\begin{equation}
\label{eq:prelim1}
\nu_\rho (\ell, A ) = \sum_{x\in\Z^d}\rho(\ell, x)P_x[X \in A], \text{ for $A \in \mathcal{W}^{+}$}
\end{equation}
and define
\begin{equation}
\label{eq:prelim2}
\mathcal{I}^{\rho} = \bigcup_{(\ell,w)\in \eta}w[0,\ell-1].
\end{equation}
In view of \eqref{eq:prelim2}, the label $\ell$ indeed corresponds to the (time-)length of a trajectory in $\mathcal{I}^{\rho}$, as indicated above. We denote by $\mathbb{P}_{\rho}$ the canonical law of $\eta$. Notice that, for any finite $K \subset \mathbb{Z}^d$, as follows immediately by comparing \eqref{eq:prelim1} with \eqref{e:RI-intensity}, one has that
\begin{equation}
\label{eq:prelim3}
\begin{split}
&\mathcal{I}^u\cap K \stackrel{\text{law}}{=} \mathcal{I}^{\rho_u} \cap K, \text{ for } \rho_u(\ell,x)=u 1_{\infty}(\ell) \, e_K(x).
\end{split}
\end{equation}
Similarly, the set $\mathcal{I}^{f, L}$ from \eqref{eq:J} is in the realm of \eqref{e:this-is-rho}, see~\eqref{eq:prelim3.1} below. We now give an alternative description of the law of $\mathcal{I}^\rho$ when restricted to a finite set $K \subset \Z^d$, which will make comparison to $\mathcal{I}^u\cap K$ as in \eqref{eq:prelim3} easier.
The following lemma roughly asserts that, to describe $\mathcal{I}^\rho \cap K$, one can replace the intensity $\rho$ with  $\rho_K$ which fasts forwards the walk until time $H_K$.
\begin{lem}[Re-rooting] \label{L:reroot}For a measure $\rho$ supported on $\mathbb{N}^\ast  \times \Z^d$ and finite $K \subset \Z^d$, defining
\begin{equation}
\label{eq:prelim4}
{\rho_K(\ell,x)}=  \sum_{\ell' \ge 0} E_x \big[ \rho(\ell+\ell', X_{\ell'}) 1_{\{\widetilde{H}_K > \ell'\}} \big] 1_{x\in K},
\end{equation}
one has, with $\geq_{\textnormal{st.}} $ denoting stochastic domination,
\begin{equation}
\label{eq:prelim5}
\mathcal{I}^{\rho}\cap K \stackrel{\textnormal{law}}{=}  \mathcal{I}^{\rho_K}\cap K\quad \text{and}\quad \mathcal{I}^{\rho} \geq_{\textnormal{st.}}  \mathcal{I}^{\rho_K}.
\end{equation}
\end{lem}

\begin{proof}
  For $x \in K$ and $\ell \in \mathbb{N}^\ast$, let $A_{\ell, x} \subset  (\mathbb{N}^\ast \times W^+)$ consist of all pairs $(t, w)$ such that $\ell \leq t < \infty$, $H_K(w) = t - \ell$ and $X_{H_K}(w)=x$. Then by definition of $\mathcal{I}^{\rho}$ in \eqref{eq:prelim2}, one has
  \begin{equation}\label{eq:reroot1}
    \mathcal{I}^{\rho} \cap K = \bigcup_{x \in K} \ \bigcup_{\ell \in \mathbb{N}} \ \bigcup_{(t, w) \in ( \text{supp}(\eta) \cap A_{\ell, x} ) } w[t - \ell, t-1] \cap K.
  \end{equation}
  The unions over $x$ and $\ell$ in \eqref{eq:reroot1} are over independent processes. We make this decomposition more explicit by introducing $\Phi_{\ell, x}$, which acts on $\eta$ by mapping every pair $(t, w) \in \text{supp}(\eta) \cap A_{\ell, x}$ to the trajectory $(\ell, \tilde{w})$, where $\tilde{w}(\cdot) = \theta_{t-\ell} w( \cdot)$, where $(\theta_nw) (\cdot)= w(n+ \cdot)$ for $w \in W^+$ and $n \geq 0$ denote the canonical shifts. With this, \eqref{eq:reroot1} can be recast as
  \begin{equation}\label{eq:reroot1bis}
    \mathcal{I}^{\rho} \cap K = \bigcup_{x \in K} \ \bigcup_{\ell \in \mathbb{N}} \ \bigcup_{(\ell, \tilde{w}) \in ( \text{supp}(\Phi_{\ell, x}(\eta)) ) }\tilde{w}[0, \ell-1] \cap K. 
  \end{equation}
 Observe now that $\big( \Phi_{\ell,x}(\eta) \big)_{\ell, x}$ are independent Poisson point processes since the sets $A_{\ell,x}$ are disjoint, and $\mathcal{I}^{\rho}$ depends on $\eta$ only through $\Phi_{\ell, x}(\eta)$, 
  $\ell \in \mathbb{N}^\ast$, $x \in K$. On the other hand, omitting the intersection with $K$ on both sides of \eqref{eq:reroot1bis} clearly yields the inclusion `$\supset$' in place of an equality.

  To conclude the proof, we compute the intensity measure $\nu_{\ell, x}$ of each $\Phi_{\ell, x}(\eta)$.
  Since $\Phi_{\ell, x}(\eta)$ is concentrated on points $(\ell, w)$ with $w(0)=x \in K$, it follows that $\nu_{\ell, x}(t, W^+) = 0$ if $t \neq \ell$ or $x \not \in K$.
  On the other hand, for $x \in K$ and $D \in \mathcal{W}^+ $, applying the Markov property,
  \begin{equation} \label{eq:expr_nuellx}
    \begin{split}
      \nu_{\ell, x}(\ell, D) &= \nu_\rho \big(\{ (t, w) \in A_{\ell, x}: w(t - \ell + \, \cdot) \in D \} \big)\\[0.5em]
      &= \sum_{\ell \leq t < \infty} \sum_{y \in \Z^d} \rho(t, y) P_y[H_K = t - \ell, X_{H_K} = x, X_{t - \ell + \, \cdot} \in D]\\
      &= P_x[X_{\cdot} \in D] \sum_{\ell \leq t < \infty} \sum_{y \in \Z^d} \rho(t, y) P_y[H_K = t - \ell, X_{H_K} = x].
    \end{split}
  \end{equation}
  Using reversibility of the simple random walk one rewrites
  \begin{equation}\label{eq:reversibility}
  \begin{split}
    \sum_{\ell \leq t < \infty} \sum_{y \in \Z^d} \rho(t, y) &P_y [H_K = t - \ell, X_{H_K} = x] \\&= \sum_{\ell \leq t < \infty} \sum_{y \in \Z^d} \rho(t, y) \frac{a_x}{a_y} P_x [\widetilde{H}_K > t - \ell, X_{t - \ell} = y],
  \end{split}
  \end{equation}
  which equals $\sum_{\ell' \geq 0} a_x E_x \big[ a_{ X_{\ell'}}^{-1}\rho(\ell + \ell', X_{\ell'}) 1_{\{\widetilde{H}_K > \ell'\}} \big]$ after 
  the substitution $\ell' = t - \ell$, thus finishing the proof on account of \eqref{eq:prelim4}.
\end{proof}
\begin{remark} Albeit notationally simpler, the formula~\eqref{eq:prelim4} for the re-rooted density $\rho_K$ could be replaced by the wordier, but more transparent (and equivalent in the present setup) definition 
\begin{equation}
\label{eq:prelim4'} \tag{\ref{eq:prelim4}'}
\frac{\rho_K(\ell,x)}{a_x} =  \sum_{\ell' \ge 0} E_x \Big[\frac{ \rho(\ell+\ell', X_{\ell'}) }{a_{X_{\ell'}}}1_{\{\widetilde{H}_K > \ell'\}} \Big] 1_{x\in K}.
\end{equation}
The uniformity of $a_x(=4d)$ allows us to effectively work with a `flat' density $\rho$ in \eqref{e:this-is-rho}-\eqref{eq:prelim1}, i.e.~with reference measure in the second argument of \eqref{e:this-is-rho} given by counting measure on $\Z^d$ rather than one with density $a_{\cdot}$. Although slightly less stringent, this choice, reflected in our formula \eqref{eq:prelim4}, cf.~also \eqref{eq:occtime} below, somewhat simplifies the exposition in the sequel.
\end{remark}

\subsection{Local couplings between $\mathcal I^\rho$ and $\mathcal{I}^u$}\label{subsec:loc-coup}

We now exhibit sufficient conditions on the intensity $\rho$ in \eqref{e:this-is-rho} ensuring that $\mathcal I^\rho$ locally resembles $\mathcal{I}^u$ for a given $u>0$. The main results appear in Proposition~\ref{prop1:cube} and~\ref{prop2:cube} and yield (local) couplings between the two objects. The proximity between the two sets involves an 'average occupation time density' field $\bar \ell_x^\rho= \bar{\ell}_x$, $x \in 
\Z^d$, for $\mathcal{I}^\rho$, which acts as a surrogate for the scalar parameter $u$ in view of \eqref{e:ell-u-mean}. It is defined as  
\begin{multline}
\label{eq:occtime}
 \bar \ell_x  =\bar \ell_x^\rho=a_x^{-1}\sum_{k > 0} \int d\nu_\rho(k,\cdot)  \sum_{0 \le \ell < k}1_{\{ X_{\ell} 
=x\}}\\
\stackrel{\eqref{eq:prelim1}, \eqref{eq:P_n-def}}{=} a_x^{-1}\sum_y \sum_{k > 0} \rho(k,y) \sum_{0\leq \ell < k }p_{\ell}(y,x)= \sum_{\ell \geq 0} E_x\big[\textstyle\frac{\rho(\ell+\mathbb{N}^\ast, X_\ell)}{a_{X_\ell}}\big]=\displaystyle \frac1{4d}\sum_{\ell \geq 0} E_x\big[{\rho(\ell+\mathbb{N}^\ast, X_\ell)}\big].
\end{multline}
The next two propositions yield the desired local couplings relating $\mathcal{I}^{\rho} \cap B$ 
to $\mathcal{I}^u \cap B$ for a box $B=B_N$ under certain assumptions on~$\rho$ and $\bar{\ell}$. The simplest instances to keep in mind are the `pure length-$L$' models $\mathcal I^{u, L}$, see in particular \eqref{eq:loc-time-uL} below. The first (and easier) of the two results yields a coupling by which $\mathcal{I}^\rho$ comprising short (i.e.~finite) walks, is covered by the long (infinite) walks of the full interlacement 
$\mathcal{I}^u$. 

\begin{prop}[Local coupling I]
	\label{prop1:cube}
	If, for some $N \geq 1$, $a>6d$ and 
	$\rho$ supported on $\mathbb{N}^\ast \times \Z^d$, 
	\begin{align}
	\label{eq:cubecond1}
	& 
	\rho(\mathbb{N}^\ast, \,\cdot\,) \leq N^{-a}	\end{align}
	and moreover, for some $u >0$ and $\delta \in (0,1)$,
	\begin{equation}\label{eq:cubecond1.1}
	\text{$\bar{\ell}_x  \leq u(1-\delta)$ for all $x \in \Z^d$,}
	\end{equation}
	then there exists a coupling of $\mathcal{I}^\rho \cap B$ and $\mathcal{I}^{u} 
	\cap B $ with $B=B_N$ such that, for all $N \ge C \delta^{-3}$,  
	\begin{equation}
	\label{eq:cube1}
	\big( \mathcal{I}^\rho \cap B \big) \subset \big( \mathcal{I}^{u} \cap B \big) \text{ with probability at least } 1- N^{-\frac a2}.
	\end{equation}
\end{prop}
Proposition~\ref{prop1:cube} is sufficient for our purposes, but the coupling constructed is far from optimal; see Remark~\ref{R:loc-coup} at the end of this section for more on this.
\begin{proof}
	The coupling will be defined under a probability $\overline{\P}$ carrying two independent Poisson point processes $({\eta}_1, {\eta}_2)$, each of them defined on the space $\R_{+} \times  W^+$ with intensity measure
	\begin{equation}
	\label{eq:nu_bar}
	\bar{\nu} ([0,v] \times  A) \stackrel{\text{def.}}{=} v \sum_{x \in \Z^d} P_x[X \in  A], \, v \geq 0.
	\end{equation}	
	For $w \in W^+$ with $ {\rho}( \mathbb{N}^\ast \times \{w(0)\}) >0$ and $0< u \leq {\rho}( \mathbb{N}^\ast \times \{w(0)\})$, let $\ell(u,w)$ denote the unique element $\ell \in \mathbb{N}^\ast$ such that $ \rho(\{0,\dots, \ell-1\} \times \{w(0)\})< u\leq \rho(\{0,\dots, \ell\} \times \{w(0)\})$.
	Then define, for $ \eta = \sum_{i} \delta_{(u_i, w_i)}$ any point measure on $\R_{+} \times  W^+$,
	\begin{equation}
	\label{eq:cube4.01}
	\overline{\mathcal{I}}^\rho({\eta})  = \bigcup_{i:\, u_i \leq  {\rho}( \mathbb{N}^\ast \times \{w_i(0)\})} w_i[0,\ell(u_i, w_i)-1]. 
	\end{equation}
It follows readily from \eqref{eq:nu_bar}-\eqref{eq:cube4.01} that $ \overline{\mathcal{I}}^\rho({\eta}_i)$, $i=1,2$, has the same law  under $\overline{\P}$ as $\mathcal{I}^\rho$ defined in \eqref{eq:prelim2} (under $\P_\rho$), for any measure $\rho$ supported on $\mathbb{N}^\ast \times \Z^d$.

	Let $\tilde{B}=  B_{2N^{3}}$ and $\tilde{\rho}: \mathbb{N} \times \Z^d \to \mathbb{R}_{+}$ be defined as
	$\tilde{\rho}(\ell,x)=\rho(\ell,x)1_{\Z^d \setminus \tilde{B}}(x).$ With $\tilde{W}^+ \subset W^+$ denoting the subset of trajectories with starting point outside $\tilde{B}$, we write $\tilde \eta$ for the restriction of $\eta$ (a point measure on $\R_{+} \times  W^+$) to points $(u, w)$ with $w \in \tilde{W}_+$, and introduce two random sets (under $\overline{\P}$)
	\begin{equation}
	\begin{split}
	\label{eq:cube4}
	{\mathcal{I}}^{{\rho}}(\eta_1,\eta_2) &\stackrel{\text{def.}}{=} \overline{\mathcal{I}}^{\tilde{\rho}_B}({\eta}_1) \cup \overline{\mathcal{I}}^\rho({\eta}_2 -\tilde{\eta}_2) \quad \text{(with $\tilde{\rho}_B$ as defined in \eqref{eq:prelim4}),} \\[0.5em]
	{\mathcal{I}}^{u} (\eta_1) &\stackrel{\text{def.}}{=}  \bigcup_{\substack{i : \, 0 \leq u_i \leq u e_B(w_i(0))}} 
	w_i[0,\infty),\
	\text{ if } \eta_1= \sum_{i} \delta_{(u_i, w_i)}.
	\end{split}
	\end{equation}
	One readily verifies using \eqref{e:RI-intensity} that ${\mathcal{I}}^{u} (\eta_1) \cap B$ has the same law under $\overline{\P}$ as $\mathcal{I}^u \cap B$ under $\P$.
	As we now briefly explain, ${\mathcal{I}}^{{\rho}}(\eta_1,\eta_2) \cap B$ has the same law under $\overline{\P}$ as $\mathcal{I}^\rho \cap B$ under $\P_\rho$. Indeed, it suffices to argue that
	\begin{equation}
	\label{eq:cube4.001}
	( \overline{\mathcal{I}}^{\tilde{\rho}_B}({\eta}_1) \cap B,  \overline{\mathcal{I}}^\rho({\eta}_2 -\tilde{\eta}_2) \cap B) \stackrel{\text{law}}{=} ( \overline{{\mathcal{I}}}^{{\rho}}(\tilde{\eta}_2) \cap B,  \overline{\mathcal{I}}^\rho({\eta}_2 -\tilde{\eta}_2) \cap B).
	\end{equation}
	To see this, first note that $\overline{\mathcal{I}}^\rho(\tilde{\eta}_2)= \overline{\mathcal{I}}^{\tilde{\rho}}({\eta}_2) $ has the same law as $\mathcal{I}^{\tilde{\rho}}$ in \eqref{eq:prelim2} by the discussion following \eqref{eq:cube4.01}. Thus, Lemma \ref{L:reroot} applies and yields that $(\overline{\mathcal{I}}^\rho(\tilde{\eta}_2) \cap B) \stackrel{\text{law}}{=} ( \mathcal{I}^{\tilde{\rho}_B}({\eta}_2) \cap B)$. Since the sets $\overline{\mathcal{I}}^\rho(\tilde{\eta}_2) $ and $ \overline{\mathcal{I}}^\rho(\eta_2 -\tilde{\eta}_2)$ are independent and
	$\eta_1$ and $\eta_2$ are i.i.d., \eqref{eq:cube4.001} directly follows.
	
	We will now show that under the assumption \eqref{eq:cubecond1}, for all $N \geq C\delta^{-3}$,
	\begin{align}
	&\sup_x \bar{\ell}_x  \leq u(1-\delta) \ \Rightarrow \   \tilde{\rho}_B( \mathbb{N}^\ast, \, x) \leq u e_B(x) \text{ for all $x \in \Z^d$,}\label{eq:cube5.1}
	\end{align}
	Before proving \eqref{eq:cube5.1}, we first explain how to deduce \eqref{eq:cube1}.
	Using \eqref{eq:cubecond1} and the fact that $1-e^{-x}\leq x$ for $x >0$, we see that
	\begin{equation}
	\label{eq:cube3}
	\overline{\P}[ \, \overline{\mathcal{I}}^\rho(\eta_2- \tilde \eta_2) \neq \emptyset\, ]=  1- e^{-\sum_{\ell \geq 0, x\in \tilde{B} } \rho(\ell,x)} \leq CN^{3d-a}, 
	\end{equation}
	which is less than $N^{-a/2}$ as $a>6d$. Due to \eqref{eq:cube5.1}, we infer immediately from \eqref{eq:cube4.01} and \eqref{eq:cube4} that ${\mathcal{I}}^{u} (\eta_1)\supset \overline{\mathcal{I}}^{\tilde{\rho}_B}(\tilde{\eta}_1)$ whenever $\bar{\ell}_x  \leq u(1-\delta)$ for all $x \in \Z^d$ (and \eqref{eq:cubecond1} holds). Together with \eqref{eq:cube3}, this implies that ${\mathcal{I}}^{{\rho}}(\eta_1,\eta_2)\subset {\mathcal{I}}^{u} (\eta_1)$ with $\overline{\P}$-probability at least $1-N^{-a/2}$, which yields \eqref{eq:cube1} since the sets ${\mathcal{I}}^{{\rho}}(\eta_1,\eta_2) \cap B$ and $ {\mathcal{I}}^{u} (\eta_1) \cap B$ have the required marginal distributions.

		It remains to show \eqref{eq:cube5.1}. Since $\tilde{\rho}(\ell,\, \cdot \,)$ vanishes in $B$ for any $\ell \geq 0$, $\tilde{\rho}_B(\ell,\, \cdot \,)$ is supported on $\partial B$, see \eqref{eq:prelim4}, hence the conclusions of \eqref{eq:cube5.1} hold trivially except for $x \in \partial B$. For such $x$, using that $\tilde{\rho}(\cdot,X_{\ell})=0$ under $P_x$ unless $\ell \geq 2N^{3}-N \, (\geq N^{3})$, 
	we find, with the hopefully obvious notation $\mathbb{N}^\ast+n=\{ n+1, n+2,\dots\}$,
	\begin{equation}
	\label{eq:cube6}
	\begin{split}
	\begin{array}{rcl}
	\tilde{\rho}_B( \mathbb{N}^\ast,\, x) \hspace{-1ex}
	&\stackrel{\eqref{eq:prelim4} }{=} & \hspace{-1ex}\displaystyle 
	\sum_{\ell \geq N^{3}}E_x\big[1_{\{\widetilde{H}_B > \ell\}} \,  \tilde{\rho}(\mathbb{N}^\ast+\ell, X_{\ell})\big] \\
	&\stackrel{}{\leq}& \hspace{-1ex} 
	\displaystyle E_x\Big[1_{\{\widetilde{H}_B > N^{3}\}} \,  \sum_{\ell  \geq 0}  
	E_{X_{N^{3}}}\big[\tilde{\rho}(\mathbb{N}^\ast+N^{3}+\ell, X_{\ell})\big]\Big]\\
	& \stackrel{\eqref{eq:occtime}}{\leq}&  \hspace{-1ex} \displaystyle a_x E_x\big[1_{\{\widetilde{H}_B > N^{3}\}} 
	\bar{\ell}_{X_{N^{3}}}\big] \stackrel{\eqref{eq:cube5.1}}{\leq}  \hspace{-1ex} \displaystyle u(1-\delta) a_x P_x[\widetilde{H}_B > N^{3}].
	\end{array}
	\end{split}
	\end{equation}
	The desired bound in \eqref{eq:cube5.1} then follows from \eqref{eq:cube6}, using Lemma \ref{L:return} with $\xi=1$ to obtain that
	$ P_x[\widetilde{H}_B > N^3] \leq\frac{1}{1-\delta} e_B(x)$ uniformly in $x \in 
	\partial B$ whenever $N \geq C(d)\delta^{-3}$.
\end{proof}

We now state a companion result to Proposition \ref{prop1:cube} with opposite inclusions. It will be 
important that this inclusion occurs with sufficiently high probability, see \eqref{eq:cube2.1} below. 
The proof involves the excursion decomposition of interlacement trajectories introduced in 
\S\ref{subsec:RI} and relies on the basic coupling from Lemma~\ref{L:RI_basic_coupling}.
\begin{prop}[Local coupling II]
\label{prop2:cube} For all $a>2$ and $\varepsilon > 0$ such that $a(1-3\varepsilon)>2$, the following holds.
	Suppose that for some $N \geq 1$ and $u > 0$, 
	the 
	measure $\rho$ in \eqref{subsec:class-mu-gen} satisfies:	
\begin{equation}
\begin{split}
	 \label{eq:cubecond2}
& 
\text{for some set $S \subset \big(\mathbb{N}^\ast \cap [N^{a(1-\varepsilon)},N^{a(1+\varepsilon)}]\big)$ with $|S|\leq N^{\varepsilon}$,} \\
&\rho(\ell,x)= \rho(\ell,x)1_{\ell \in S} \mbox{ and } \rho(S, x) \leq 
{u}N^{-a}\text{ for all $x \in \Z^d$}
 \end{split}
 \end{equation}
and (recall \eqref{eq:occtime} for notation)
\begin{equation}
	 \label{eq:cubecond2.2}
\text{$\bar{\ell}_x  \ge u(1+\delta)$ 
for all $x \in B_{N^{1+\varepsilon}}$ as well as $\bar{\ell}_x \le uD$ for all $x \in \Z^d$,}
\end{equation}
for some $\delta \in (0,1)$ and $ D > 1$. Then with $B=B_N$, there exists a coupling of 
$\mathcal{I}^\rho \cap B$ and $\mathcal{I}^{u} \cap B $ such that for some $c = c(a, \varepsilon, 
D)$ and $C = C(\varepsilon)$,  
\begin{equation}
 \label{eq:cube2.1}
\big( \mathcal{I}^{u} \cap B \big) \subset \big( \mathcal{I}^\rho \cap B \big) \text{ with probability at least } 1- e^{-c(u \wedge1) \delta^{C} N^{c}}.
\end{equation}
\end{prop}

\begin{proof} In the notation of \S\ref{subsec:RI}, we choose $A=B$  and $U =B_{N^{1 + \xi}}$ for 
	some $\xi > 0$ whose precise value will be chosen as a function of $\varepsilon$ below in 
	Lemma~\ref{L:hardcube_sizes}. 
Consider the measures $\tilde{\mu}$ on $\mathbb{N}^* \times B$ and $\mu$ on $B$ defined as 
\begin{equation}\label{def:mutilde}
	\tilde{\mu} (\ell, x) = \tilde \rho_B(\ell, x) 1_{\ell \ge t_N}, \quad \mu(\cdot) = \tilde \mu(\mathbb 
N^\ast,\, \cdot),
\end{equation} where  $\tilde \rho(\ell, x) = \rho(\ell, x) 1_{x \notin U}$, $\tilde{\rho}_B$ is obtained from $\tilde{\rho}$ according to \eqref{eq:prelim4} and $t_N=\lfloor N^{2+3\xi}\rfloor
$. In plain words, 
$\tilde \mu (\ell, x)$ represents the intensity of walks comprising $\mathcal I^{\rho}$ that (a) 
start outside $U$, i.e.~sufficiently far from $B$, (b) enter $B$ for the first 
time through $x$ and (c) have at least time $t_N$ left after doing so.

We now apply Lemma~\ref{L:RI_basic_coupling} to construct the desired coupling. Exploiting 
property~(b) as well as \eqref{eq:cubecond2} and \eqref{eq:cubecond2.2}, we will prove in
Lemma~\ref{L:hardcube_sizes} below that with the choice $\xi=c \varepsilon$ for suitable $c\in (0,1)$, the measure $\mu$ satisfies 
\begin{equation}\label{eq:mulwr_bnd}
\mu(x) \ge u \textstyle\big(1+\frac{\delta}{2}\big) e_B(x), \text{ for all $x \in B$ and 
 $N \ge C(a, \varepsilon, D)\, \delta^{-C(\varepsilon)}$}
\end{equation}
(cf.~\eqref{eq:RI_cond_mu}). On the other hand, it follows from \eqref{normalized_eB} that condition \eqref{eq:RI_cond_Q} holds for the 
pair $(B, U)$ whenever $N \ge C(\xi) \delta^{-C(\xi)}$. Recall to this end the definition of $Q_x$ from above \eqref{eq:RI_Y}.
Thus, Lemma~\ref{L:RI_basic_coupling} applies and
yields a coupling $\mathbb Q = \mathbb Q_{B, U}$ between the excursions $Z = (Z_n(\omega))_{n \ge 1}$ introduced in \eqref{eq:RI_Z} and 
an i.i.d.~sequence $\widetilde Z = (\widetilde Z_n)_{n \ge 1}$ of excursions between $B$ and 
$\partial_{{{\rm out}}} U$ under $P_{\bar \mu}$, where $\bar \mu = \mu / \mu(B)$.


We will now generate a subset of $\mathcal{I}^\rho$ that will cover $\mathcal{I}^u \cap B$ using the excursions $\widetilde Z $. This requires truncating the latter to their actual deterministic length, which could be shorter than the hitting time of $\partial_{{{\rm out}}} U$. In order to do this, we first apply a thinning procedure that recovers the length of individual trajectories. By suitable extension of the probability space, we suppose that $\mathbb Q$ carries a Poisson variable $N_B^{\mu}$ with intensity $\mu(B)$ and a family $\{U_1, U_2, \dots\}$ of i.i.d.~uniform~random variables on $[0,1]$. All of the previous random variables are independent from each other as well as independent from $Z$ and $\widetilde{Z}$. To each $\widetilde{Z}_i$, we assign a length label $\ell_i = \ell_i(U_i)$ which is the unique element $\ell \in \mathbb{N}^*$ satisfying $a_{i,\ell-1}  < U_i \leq a_{i,\ell}$, where $a_{i,\ell}= {\tilde{\mu}(\{0,\dots, 
\ell\}, \widetilde{Z}_i(0)) 
}/{ \mu(\widetilde{Z}_i(0))}.$
Note that $\ell\geq t_N$ on account of \eqref{def:mutilde}. It then follows from the thinning property of Poisson processes that 
$\hat{\omega}= \sum_{i \le N_B^{\mu}} \delta_{(\widetilde{Z}_i(0), \ell_i, 
\widetilde{Z}_i)}$ is a Poisson process on $\Z^d \times \mathbb{N}^* \times W^+_f$, where $W_{f}^+$ denotes the space of all finite-length, nearest-neighbor trajectories in $\Z^d$,
having intensity $\hat{\mu}$, where
\begin{equation*}
\hat \mu (S \times I \times A) = \sum_{x \in S} \tilde{\mu}(I, x) P_x[X_{[0, T_U]} \in A].\end{equation*} 
Consequently,
\begin{equation}
\label{def:IvcapB}
\widetilde{\mathcal{I}}^{\mu} = \widetilde{\mathcal{I}}^{\mu} (\hat{\omega}) \stackrel{\text{def.}}{=} \bigcup_{i \le N_B^{\mu}} \widetilde{Z}_i[0, (t_N-1) \wedge T^i_U ] \leq_{\textnormal{st.}} \mathcal I^{\tilde \mu} \cap B \leq_{\textnormal{st.}} \mathcal I^{\rho} \cap B,
\end{equation}
where $T^i_U $ is the exit time of the excursion $\widetilde{Z}_i$ and the second stochastic domination follows immediately on account of \eqref{def:mutilde} and Lemma~\ref{L:reroot}.
\smallskip

We now generate a copy of $\mathcal I^u \cap B$ using the excursions $Z$ (under $\mathbb Q = \mathbb Q_{B, U}$). 
By suitable extension of $\mathbb Q$, conditionally on $(Z, \widetilde Z)$, we sample an integer-valued random 
variable $N_B^u$ according the conditional distribution under $\P$ of the number of excursions coming from trajectories with label at most $u$ given $\sigma ((Z_n(\omega))_{n \ge 1})$, cf.~\eqref{eq:RI_Z}. Although not necessary, we assume for definiteness that $N_B^u$ and $N_B^{\mu}$ are independent conditionally on $(Z, \widetilde 
Z)$. With these definitions, it follows that
\begin{equation}
\label{eq:cube101}
\widetilde{\mathcal{I}}^{u} \stackrel{\text{def.}}{=} \bigcup_{i \le N_B^u} \text{range}(Z_i) \mbox{ is distributed as } \mathcal I^u \cap B \text{ under $\P$}.
\end{equation}
Together \eqref{def:IvcapB} and \eqref{eq:cube101} immediately give
\eqref{eq:cube2.1}, provided one argues that 
\begin{equation}
\label{eq:cube13}
\mathbb Q[ \widetilde{\mathcal{I}}^{u} \subset \widetilde{\mathcal{I}}^{\mu}] \ge 1- e^{-c(u\wedge 1)\delta^C N^{c}},
\end{equation}
for $c=c(\xi)$, $C=C(\xi)$. However, combining \eqref{eq:RI_basic_coupling2} in Lemma~\ref{L:RI_basic_coupling}, \eqref{def:IvcapB} and 
\eqref{eq:cube101}, we see that
\begin{multline}
\label{eq:cube13_aux}
\{ \widetilde{\mathcal{I}}^{u} \subset \widetilde{\mathcal{I}}^{\mu}\}^{\mathsf{c}} \\
 \subset \{\,\max_{i \le N_B^u} L_i \ge t_N \} \cup \{N_B^{\mu} \le  u(1 + \tfrac{\delta}{4}){\rm cap}(B) \}\cup \{ N_B^u \ge u(1 + \tfrac{\delta}{8}){\rm 
cap}(B)\} \cup \mathcal{U}_{u}^{\mathsf{c}} 
\end{multline}
where $L_i$ is the length of $Z_i$. We now bound the probabilities of each of these events 
separately. The last event on the right is handled using \eqref{eq:RI_basic_coupling1}. Consider now the third event. By construction, the quantity $N_B^u$ has the same law under $\mathbb 
Q$ of the number of excursions (under $\P$) stemming from trajectories visiting $B$ with label 
at most $u$. For each such trajectory, the number of excursions between $B$ and ${\partial_{{\rm out}}U}$ it generates is stochastically dominated using \eqref{eq:hit1} by $G$, a geometric random variable with parameter $1 -  \Cr{hit-box}N^{-\xi(d-2)}$ (with values starting at $1$). Hence, $N_B^u$ is stochastically dominated by $\sum_{1 \le j \le M_B^u} G_j$, where $M_B^u$ is a Poisson random variable with mean $u 
\,{\rm cap}(B)$ and the $G_j$'s are i.i.d.~copies of $G$, independent of $M_B^u$.
By standard 
large deviation bounds for tail probabilities of Poisson 
and geometric random variables as well along with the bound ${\rm cap}(B) \geq cN^{d-2}$ (which follows from \eqref{eq:cap_bound_Green} and \eqref{eq:Greenasympt}), one deduces that
\begin{multline}\label{eq:tail_bnd}
\mathbb Q \big [ N_B^u \ge u(1 + \tfrac{\delta}{8}){\rm cap}(B) \big] \\\le  \mathbb Q \big [ 
M_B^u \ge m \big] + \mathbb Q \big [ \sum_{1\leq j \leq m} G_j \ge u(1 + \tfrac{\delta}{8}){\rm 
cap}(B) \big] \le e^{-\delta^{C} u N^{d-2}}, \end{multline}
where $C = C(\xi, d)$ and the last bound follows with the choice $m=u(1 + \tfrac{\delta}{4}){\rm 
cap}(B)$. The second event on the right of \eqref{eq:cube13_aux} is bounded 
similarly. As for the first term, one just combines \eqref{eq:tail_bnd} with the estimate 
\eqref{eq:long_excursion} and applies a union bound.

\smallskip
All that remains is to verify is \eqref{eq:mulwr_bnd}. To this effect, one writes for all $x \in \partial B$ 
\begin{equation}
  \label{eq:cube7}
\tilde{\rho}_B (\mathbb{N} + t_N, x ) \stackrel{\eqref{eq:prelim4}}{=} \sum_{\ell \geq 0} E_x\big[1_{\{\widetilde{H}_B > \ell\}}  \tilde{\rho}( \N+ \ell+t_N , X_{\ell})\big]
\geq a_1- a_2-a_3-a_4-a_5
 \end{equation}
(the lower bound in \eqref{eq:cube7} will be explained momentarily), where one defines
\begin{align*}
&a_1 =  4d E_x\big[1_{\{\widetilde{H}_B > t_N\}} 
\bar{\ell}_{X_{2t_N}} \big], \quad (\text{with } \bar\ell_{\cdot}=\bar\ell_{\cdot}^\rho, \, \text{cf.~}\eqref{eq:occtime})
\end{align*}
and, abbreviating $\tilde \rho(I , \cdot) = \tilde \rho(I\cap \mathbb{N}^* , \cdot)$ for $I\subset \R$,
 \begin{align*}
&a_2 =  \sum_{\ell \geq 0} E_x\big[1_{\{\widetilde{H}_B > t_N\}} \tilde{\rho}\big((\ell, \ell+2t_N], \,X_{\ell+2t_N}\big)\big], \\
&a_3 = \sum_{\ell \geq 0} E_x\big[1_{\{\widetilde{H}_B > t_N\}} E_{X_{2t_N}}[ (\rho-\tilde{\rho})(\ell + \mathbb{N}^\ast, X_{\ell})]\big],\\
&a_4 = \sum_{\ell \geq 2t_N} E_x\big[1_{\{\widetilde{H}_B > t_N\}}  \tilde{\rho}([\ell,\ell+ t_N),X_{\ell})\big],
\\
&a_5 = \sum_{\ell \geq 2t_N} E_x\big[1_{\{ t_N< \widetilde{H}_B \leq \ell\}}  \tilde{\rho}(\N+\ell+t_N, X_{\ell})\big].
 \end{align*}
To see the lower bound in \eqref{eq:cube7}, one first applies the Markov property at time $2t_N$ and combines with \eqref{eq:occtime} to find that
$$
a_1' \stackrel{\text{def.}}{=} \sum_{\ell \geq 2t_N} E_x\big[1_{\{\widetilde{H}_B > t_N\}}  \tilde{\rho}(\N+\ell, X_{\ell})\big] = a_1-a_2-a_3.
$$
Next, one writes $1_{\{\widetilde{H}_B > \ell\}}$ as $1_{\{\widetilde{H}_B > t_N\}} - 1_{\{ 
	t_N< \widetilde{H}_B \leq \ell\}}$ for $\ell \geq t_N$ whence
 \begin{equation*}
\tilde{\rho}_B (\mathbb{N} + t_N , x ) \stackrel{\eqref{eq:prelim4}}{\ge} \sum_{\ell \ge 2t_N} E_x\big[1_{\{\widetilde{H}_B > t_N\}}  \tilde{\rho}( \N+ \ell+t_N , X_{\ell})\big] - a_5.
\end{equation*}
Now the first term on the right-hand side is easily seen to equal $a_1' - a_4$ in view of the identity
 \begin{align*}
 \tilde{\rho}( \N+ \ell+t_N , X_{\ell}) = \tilde{\rho}( \N+ \ell , X_{\ell}) - \sum_{0 \le 
 \ell' < t_N} \tilde{\rho}(\ell' + \ell, X_{\ell}).
 \end{align*}
We will consider each $a_i$ separately. The results are summarized in the following lemma. 
Recall that the hypothesis
 \eqref{eq:cubecond2}  depends on two parameters $a>2$ and $\varepsilon>0$ satisfying $a(1-3\varepsilon)>2$.
 
 \begin{lemma}\label{L:hardcube_sizes} Under the hypotheses of Proposition~\ref{prop2:cube}, there exists $\xi = \xi(\varepsilon)>0$ such that, uniformly in $x \in 
 \partial B$ and whenever $N \geq C(a, \varepsilon, D)\, \delta^{-C(\varepsilon)}$,
 \begin{align}
 &a_1 \geq u \big(\textstyle1+ \frac{9}{10}\delta\big) e_B(x) \text{ and }\label{eq:cubesizea_1'}\\
 &a_i \leq \textstyle \frac{u \delta}{10} \cdot e_B(x) \text{ for } i=2,\dots,5.
 \label{eq:cubesizea_4}
 \end{align}
 \end{lemma}
 The proof of Lemma~\ref{L:hardcube_sizes} is given below. Once \eqref{eq:cubesizea_1'} and \eqref{eq:cubesizea_4} are 
 proved, \eqref{eq:mulwr_bnd} follows and the proof of Proposition~\ref{prop2:cube} is complete.
 \end{proof}

\begin{proof}[Proof of Lemma~\ref{L:hardcube_sizes}] Unless otherwise specified, all subsequent estimates  are uniform in $x \in \partial B (= \text{supp}(e_B))$. By assumption in~\eqref{eq:cubecond2.2} and monotonicity, cf.~\eqref{eq:equilib_K} (also recall that $a_x=4d$), the quantity $a_1$ is larger than 
\begin{equation}
 \label{eq:cube8.2}
u(1+\delta)e_B(x)- \textstyle4d E_x\big[1_{\{ X_{2t_N} \notin B_{N^{1+ \varepsilon}}\}} 
\bar{\ell}_{X_{2t_N}} \big].
\end{equation}
To deal with the second term in \eqref{eq:cube8.2} one uses that
$\bar{\ell}_x \leq u D$ for all $x \in \mathbb{Z}^d$ as implied by \eqref{eq:cubecond2.2} and combines this for $\xi \le \varepsilon /2$ with the bound (recall that $t_N =\lfloor N^{2+3\xi} \rfloor$, see below \eqref{def:mutilde})
\begin{equation}
 \label{eq:cube8.3}
P_x [X_{2t_N} \notin B_{N^{1 + \varepsilon}}] \leq P_x [X_{2t_N} \notin B_{N^{1 + 2\xi}}] \leq C \sum_{r > N^{1+ 2\xi} } r^{d-1} e^{-c'{r^2}/{t_N}} \le C e^{-c'N^{\xi}}
\end{equation}
which follows by standard heat kernel estimates. Using that $e_B(x) \geq \frac{c}{N}$ uniformly in 
$x \in \text{supp}(e_B)$, see \eqref{eq:bnd_equil_box}, one readily bounds the expectation in 
\eqref{eq:cube8.2} to deduce overall that $a_1$ satisfies \eqref{eq:cubesizea_1'} whenever $N \geq  C( \varepsilon, \xi, D)$ and $\xi \le \varepsilon /2$.

Next we bound $a_2$. First recall from \eqref{eq:cubecond2} that for any $x \in 
\Z^d$ and interval $I \subset \mathbb{N}^*$, $\tilde{\rho}(I, x) \,(\leq \rho(I,x))$ is bounded by $u N^{-a}$ and 
vanishes when $S \cap I = \emptyset$. Therefore,
\begin{multline}\label{eq:cube_sumbnd}
a_2 \le P_x[\widetilde{H}_B > t_N] \cdot \sum_{\ell}\sup_{x \in 
\Z^d}\tilde{\rho}\big((\ell, \ell+ 2t_N], x\big) \le P_x[\widetilde{H}_B > t_N] \cdot 
u N^{- a}\cdot|S|\cdot 2t_N\\
\stackrel{\eqref{eq:cubecond2}}{\le} uP_x[\widetilde{H}_B > t_N] \cdot 2N^{2 + 3\xi + 
\varepsilon - a} \stackrel{\eqref{eq:return1}}{\le} ue_B(x)(1 + o(1))\cdot 2N^{2 + 3\xi + 
\varepsilon - a} \text{ as }N \to \infty \end{multline}
where the rate $o(1)$ is subject to the choice of $\xi$. 
Hence $a_2 = u\delta e_B(x)\cdot o(1)$ as soon as $(2 + 3\xi + \varepsilon - a) < 0$ 
which holds for all $\xi \le \varepsilon$ since $a(1 - 3\varepsilon) > 2$.

The term $a_3$ is the most delicate. Recalling that $(\rho-\tilde{\rho})(\ell,x)= \rho(\ell,x) 1_{\{x \in B_{N^{1+\xi}}\}}$, cf.~below \eqref{def:mutilde}, it follows using \eqref{eq:return1} that $a_3$ is bounded by
\begin{equation}
\label{eq:cube15}
C e_B(x)\sup_{y \in B_{N^{1+ 2\xi}}} E_y\big[\textstyle \sum_{\ell \geq 0}\rho(\ell + \mathbb{N}^\ast, X_{\ell})1_{\{X_{\ell} \in B_{N^{1+\xi}}\}} \big] +  C u N^{a \varepsilon}e^{-c'N^{\xi}};
\end{equation}
in deducing \eqref{eq:cube15}, we have also used \eqref{eq:cube8.3}, along with the fact that $ \sum_{\ell \geq 0}\rho(\ell + \mathbb{N}^\ast, \cdot) \leq u N^{a\varepsilon}$, itself a consequence of \eqref{eq:cubecond2}, to deal with the 
case that $X_{2t_N} \notin B_{N^{1+ 2\xi}}$. In view of \eqref{eq:cube15}, in order to 
obtain \eqref{eq:cubesizea_4} for $a_3$, it is more than sufficient to argue that 
\begin{equation}
\label{eq:cube16}
\sup_{y \in B_{N^{1+ 2\xi}}} E_y\big[\textstyle \sum_{\ell \geq 0}\rho(\ell + \mathbb{N}^\ast, 
X_{\ell})1_{\{X_{\ell} \in B_{N^{1+\xi}}\}} \big] \le u N^{-c(a, \varepsilon, D)} 
\end{equation}
(intuitively, this will be because $\rho$ is supported at scales $\ell \approx N^a \gg N^2$, 
which renders the condition $X_{\ell} \in B_{N^{1+\xi}}$ costly if $\xi$ is chosen small 
enough). To get \eqref{eq:cube16}, first note that contributions to \eqref{eq:cube16} from 
$\ell \leq N^{a(1-\varepsilon)}$ are easily dispensed with: using \eqref{eq:cubecond2},
$$E_y\big[ \sum_{0 \leq \ell \leq N^{a(1-\varepsilon)}}\rho(\ell + \mathbb{N}^\ast, 
X_{\ell})1_{{\{X_{\ell} \in B_{N^{1+\xi}}\}}} \big] \leq N^{a(1-\varepsilon)} \sup_z \rho(S, z) 
\leq u  N^{-a\varepsilon}.$$
Now, observing that no contributions to \eqref{eq:cube16} arise from terms $\ell \geq N^{a(1+ \varepsilon)}$, 
using again the deterministic bound $\rho(\ell + \mathbb{N}^\ast, X_{\ell}) \leq \sup_z \rho(S, z)  \leq u N^{-a} $ implied by \eqref{eq:cubecond2} and applying the on-diagonal estimate $P_x[X_\ell = y] \leq C\ell^{-d/2}$, it follows that  
for all $y \in B_{N^{1 + 2\xi}}$,
\begin{align*}
E_y\big[ \sum_{ \ell > N^{a(1-\varepsilon)}}\rho(\ell + \mathbb{N}^\ast, X_{\ell})1_{\{X_{\ell} \in B_{N^{1+\xi}}\}} \big]  & \leq u N^{a \varepsilon}\sup_{\ell > N^{a(1-\varepsilon) }}P_y[X_{\ell} \in B_{N^{1+\xi}}]\\ 
&\leq u N^{a \varepsilon+(1+\xi)d - \frac{a}{2}(1-\varepsilon)d }.
\end{align*}
Since $d \geq3 $, the last exponent is negative (i.e.~the previous line is $u N^{-c}$) 
if the condition $\frac a2(1-2\varepsilon)> 1+  \xi$ is satisfied. As $a(1-3\varepsilon)>2$ 
by assumption, this condition is met by choosing $\xi=c \varepsilon$ with $c \in (0, 1/2)$ small enough so that $1+ \xi \leq \frac{(1-2\varepsilon)}{(1-3\varepsilon)}$, which we now fix (recall that we just need $\xi \leq \varepsilon/2$ in view of previous requirements). Overall, 
\eqref{eq:cube16} thus follows and with it \eqref{eq:cubesizea_4} for $a_3$.

The bound on $a_4$ follows from \eqref{eq:cube_sumbnd} in exactly similar manner as that on $a_2$. As for $a_5$, applying the Markov property at time $t_N$ (say), one obtains that
\begin{multline*}
a_5 \le E_x\Big[1_{\{ t_N< \widetilde{H}_B < \infty\}}  \Big(\sum_{\ell \geq t_N}\tilde{\rho}(\N+\ell+2t_N, X_{\ell}) \Big) \circ \theta_{t_N} \Big]  \\
\stackrel{\eqref{eq:occtime}}{\le} C P_x[t_N< \widetilde{H}_B < \infty] \sup_{z \in \Z^d} \bar{\ell}_z \stackrel{\eqref{eq:return3}\text{ff.}}{\le} uD  C(\varepsilon) N^{-c\varepsilon}
e_B(x), 
\end{multline*}
where we also used \eqref{eq:cubecond2}-\eqref{eq:cubecond2.2} in the last step.
 \end{proof}

\subsection{Local convergence of $\mathcal{I}^{u,L}$ to $\mathcal I^u$}
\label{subsec:fr-models}
We now focus on the case of the homogenous length-$L$ models $ \mathcal{I}^{f ,L}$ introduced in  \eqref{eq:J}, which  are of class $\mathcal{I}^\rho$. Indeed, in view of \eqref{eq:J} and \eqref{eq:prelim1}-\eqref{eq:prelim2}, for any positive functgion $f$ on $\Z^d$, one has
\begin{equation}
\label{eq:prelim3.1}
\mathcal{I}^{f,L} \stackrel{\text{law}}{=} \mathcal{I}^{\rho}, \, \text{with }\rho (\ell,x)= \textstyle\frac{a_x f(x)}L 1_{L}(\ell)=4d \textstyle\frac{ f(x)}L 1_{L}(\ell), \, x \in \Z^d,
\end{equation}
which specialises to $\mathcal{I}^{u,L}$ with $f(x)=u$, $x\in\Z^d$. In the latter case, in which we denote $\rho=\rho_{u,L}$ the measure appearing in \eqref{eq:prelim3.1}, it is instructive to observe that for all $x \in \Z^d$,
\begin{equation}
\label{eq:loc-time-uL}
\bar{\ell}_x \stackrel{\eqref{eq:occtime}}{=}\frac1{4d}  \sum_{\ell \geq 0} E_x[\rho(\ell+\mathbb{N}^\ast, X_\ell)]= \sum_{\ell \geq 0} \frac{u}{L} 1\{ \ell <L\} =u.
\end{equation}
With a view towards \eqref{e:ell-u-mean} and the conditions entering Propositions~\ref{prop1:cube} and \ref{prop2:cube}, \eqref{eq:loc-time-uL} suggests that $\mathcal{I}^{u,L}$  is a good local approximation for $\mathcal{I}^u$. Indeed, one has the following result. We tacitly endow $\{0,1\}^{\Z^d}$ with the product topology and convergence in distribution, as stated below, corresponds to convergence in law of all finite dimensional marginals.
\begin{prop}[$u\geq 0$]\label{P:loclimit}
The set $\mathcal{I}^{u,L}$ under $\P_{\rho_{u,L}}$ converges in distribution to $\mathcal{I}^u$ as $L \to \infty$.
\end{prop}
\begin{proof}
It is enough to show that for all finite $K \subset \Z^d$,
\begin{equation}
\label{e:loc-limit-J-u-L1}
\textstyle \lim_L\P_{\rho_{u,L}}[\mathcal{I}^{u,L} \cap K= \emptyset] = \exp\{-u\text{cap}(K)\}.
\end{equation}
Let $a=3$, $N= \lfloor L^{1/a}\rfloor$ and $\delta \in (0,1)$. 
The condition \eqref{eq:cubecond1} of Proposition~\ref{prop1:cube} is thus in force for $\rho=\rho_{u,L}$ on account of \eqref{eq:prelim3.1} and \eqref{eq:cubecond1.1} holds with $u'= \frac{u}{1-\delta}$ in place of $u$ by \eqref{eq:loc-time-uL}. Hence, Proposition~\ref{prop1:cube} applies with these choices and one readily finds applying \eqref{eq:cube1} that $$
\textstyle \liminf_L\P_{\rho_{u,L}}[\mathcal{I}^{u,L} \cap K= \emptyset] \geq \P[\mathcal{I}^{u'} \cap K= \emptyset] = \exp\{-u'\text{cap}(K)\}.$$
A corresponding upper bound is obtained by using Proposition~\ref{prop2:cube} instead, which applies with the same choices for $a$ and $N$ and $S=\{L\}$, $D=1$. The result \eqref{e:loc-limit-J-u-L1} follows by letting $\delta \downarrow 0$. 
\end{proof}

\begin{remark}\label{R:loc-coup}
\begin{itemize}
\item[1)] In 
the upcoming sections, we will face the challenging task of deriving couplings which operate
between walks having comparable lengths $\ell \approx L$, for a given $L \geq1$, with coupling errors decaying super-polynomially in $L$ and within boxes whose linear size $N$ is unrestricted (and may well be e.g.~comparable to the typical spatial extension of the walks, or even much larger). This is essentially disjoint from the regime covered by the above results (which will still be used, see the discussion below). Indeed, in the notation of Propositions~\ref{prop1:cube} and~\ref{prop2:cube}, this means replacing $\mathcal I^u$ by $\mathcal I^{\rho'}$ where the typical side length $L \in \text{supp}(\rho)$ is comparable to that of $ \text{supp}(\rho')$, and possibly $N \gg L^{\frac12}$. In contrast, with a view 
to \eqref{eq:prelim3.1} (a case in point), the conditions \eqref{eq:cubecond1} and \eqref{eq:cubecond2} require that $N \approx L^{\frac{1}a} \ll L^{\frac12}$, with an error at best polynomial in $L$ in Proposition~\ref{prop1:cube} (cf.~\eqref{eq:cube1}). We also refer to the results of \cite[Lemma 5.3]{10.1214/23-EJP950} in this context, which yield an exact comparison between models of length $L$ and $2L$, with a polynomial error term in $L$ (or equivalently, $N$) similarly as in \eqref{eq:cube1}, which becomes effective when $N \ll L^{\frac1{7d}}$.

\item[2)] We briefly indicate in how far the above couplings will be used below. Our choice to include Proposition~\ref{prop1:cube}, which causes little effort but could be dispensed with, stems from the fact that, together with Proposition~\ref{prop2:cube}, it already yields in a self-contained fashion the proof of the convergence in law asserted in Proposition~\ref{P:loclimit}. Whereas Proposition~\ref{prop1:cube} will soon be improved for a suitable class of models of the form $\mathcal{I}^{\rho}$ (including $\mathcal{I}^{u,L}$), essentially by iterating Theorem~\ref{thm:short_long} below, which has a stand-alone proof, Proposition~\ref{prop2:cube} is non-negotiable: it will be used as a crucial input in~\S\ref{sec:disconnection}, in order to exhibit the desired disconnection events defining the obstacle set $\mathcal{O}=\mathcal{O}(\omega)$ for the random environments $\omega$ of interest.  
\end{itemize}
\end{remark}

\section{Covering length-$L'$ by length-$L$ interlacements}
\label{sec:easyCOUPLINGS}
In this section we prove Theorem~\ref{thm:short_long-intro}.
In fact, we will prove a slightly more general 
statement, Theorem~\ref{thm:short_long} below, which is often easier to use in practice; see also Remark~\ref{R:short_long},\ref{rmk:short_length}. The proof starts with a reduction step, stated in \S\ref{ssec:tgap}, see Proposition~\ref{thm:short_long'}, which introduces \textit{gaps} between (pieces of) trajectories that will later favor mixing and drive the coupling. The `ungapped' theorem is then deduced from its `gapped' version, Proposition~\ref{thm:short_long'}, in \S\ref{ssec:tgap}. The proof of Proposition~\ref{thm:short_long'} appears in \S\ref{S:proofofeasycoupling}. We now state the main result of this section.

\begin{thm}\label{thm:short_long} 
For all $u \in (0,\infty)$, integers $K \ge 0$ and $L, L' \geq 1$ such that $L'$ divides $L/2$ and $L'\ge L^{1-c}$,
the following holds. Given any two functions $f_1,f_2:\Z^d \to \R_+$ 
such that $f=f_1+f_2$ satisfies $ u \ge f  \ge CL^{-c'}$ on $ B_{K+L}$, there exists a coupling $\mathbb Q$ of $(\mathcal I_1,\mathcal I_2)$ such that  
\begin{equation}
\label{eq:short_long}
\begin{split}
&\mathcal I_1 \stackrel{\textnormal{law}}{=}  \mathcal I^{\frac12(1 + P_{L/2})f_1, \frac L2} \cup \mathcal I^{f_2, L}, \\
&\mathcal I_2 
\stackrel{\textnormal{law}}{=} \mathcal I^{(1 - \Cl{C:sprinkle-easy}({L}/{L'})^{-1/2})P^{L'}_L(f1_{B_K}), L'}, 
\text{ and}\\
&	\mathbb Q \left[ {\mathcal I}_1  \supset  {\mathcal I}_2  \right] \geq 1 - C 
(u \vee 1)(K+L)^d \,e^{-c \,({L}/{L'})^{1 / 4}},
\end{split}
\end{equation}
where ${\mathcal I}^{\frac12(1 + P_{L/2})f_1, \frac{L}{2}}$ and ${\mathcal I}^{f_2, L}$ are sampled independently in the law defining $\mathcal I_1 $.
\end{thm}
\begin{remark}\label{R:short_long}
\begin{enumerate}[label*=\arabic*)]
\item \label{rmk:short_length} One immediately deduces Theorem~\ref{thm:short_long-intro}  by setting $f_1 = 0$ and $f_2 = f$, at least when $L' | \frac L2$. The only loss of generality that we incur here is 
the requirement $L' | \frac L2$ instead of just $L' | L$.  However, our proof makes  clear that the 
assumption $L' | \frac L2$ is unnecessary when $f_1 = 0$, i.e. when we are in the set-up of 
Theorem~\ref{thm:short_long-intro}. 

\item \label{R:profile-compl}Profiles like the one defining $\mathcal{I}_1$ in \eqref{eq:short_long} naturally arise e.g.~while transitioning from the configuration $\mathcal I^{u, 
2L}$ to $\mathcal I^{u, L}$ through sequential couplings, as we will see in the proof of 
Theorem~\ref{thm:main III} in Section~\ref{sec:VuVuLcouple}. It also appears for a similar reason in our 
companion paper \cite{RI-I}.
\end{enumerate}
\end{remark}

We now give a brief overview of the proof of Theorem~\ref{thm:short_long}, which occupies the remainder of this section. Throughout, we often abbreviate
\begin{equation}\label{eq:no-loop-l}
l=  L/{L'}
\end{equation}
the (integer) ratio of the two spatial scales of concern. Recalling the definition of $P_L^{L'}(f)$ from \eqref{eq:f'}, one immediately sees that the law of 	${\mathcal I}^{f',\, L'}$ with $f'=P_L^{L'}(f)$ is the same as that of the union over $k$ of $l$-many {\em independent} 
configurations ${\mathcal I}^{f'_k, L'}$ for $0 \le k < l$, where $f_k'=l^{-1}P_{kL'}(f)$. On the other hand, if one cuts each of the length-$L/2$ and length-$L$ trajectories underlying ${\mathcal I}^{(1 + 
P_{L/2})(f_1/2), L/2}$ and ${\mathcal I}^{f_2, L}$ at times $kL'$, where $0 \le k <\tfrac{l}{2}$ in the former and $\frac{l}{2} 
\le k < l$ in the latter case, and collects the resulting length-$L'$ trajectories, one can similarly view 
the law of $\mathcal I_1$ in \eqref{eq:short_long} as that of the union of $l$-many configurations whose \textit{marginal} laws are readily seen to coincide with  ${\mathcal I}^{f_k', L'}$, for $0 \le k < l$. 
However, their joint law is nowhere near independent.

We deal with this problem by introducing a gap time $t_g\ll L'$ between any two successive segments, designed to be just long enough so as to allow the different sets of endpoints to `mix'. As will be seen in \S\ref{ssec:tgap}, the statement of Theorem~\ref{thm:short_long} is actually amenable to the introduction of gaps between segments, essentially due to the sprinkling inherent to $\mathcal{I}_2$ in \eqref{eq:short_long}. This leads to Proposition~\ref{thm:short_long'} below. 

We then prove Proposition~\ref{thm:short_long'} in \S\ref{S:proofofeasycoupling} by coupling the collections of starting points from these new segments using the soft local time 
technique which was already at play in \S\ref{subsec:RI} (see the proof of 
Lemma~\ref{L:RI_basic_coupling}). The strength of the resulting coupling depends on two factors 
which are actually entwined in the present context. Firstly, we need the mixing rate of the walk segments to be 
good enough which is only true when the segments start `nearby'. To this effect, we subdivide the domain 
into boxes of intermediate scale $\widetilde{L}$ and couple the walk segments starting from each 
such box separately (see Lemma~\ref{lem:easy_coup_bnd} below). To help convey an adequate picture, we stress that mixing happens at much larger scales, i.e.~$\widetilde{L} \ll \sqrt{t_g}$ in the end; cf.~for instance \eqref{e:Ltilde-cond}.
Secondly, the coupling error also depends on the number and concentration of the {\em difference} in the number of starting points of the 
two configurations to be coupled, which is where the `ellipticity' lower bound on $f$ and the multiplicative sprinkling  term $(1-l^{-1/2})$ in the definition of $\mathcal I_2$ 
(see~\eqref{eq:short_long}) enter. These features will play a role when determining the mean $\lambda_D$ (defined below \eqref{eq:wsk}) and typical fluctuations for the number of relevant trajectories starting inside a box $D$ of radius $\widetilde{L}$. 

\subsection{Gaps and the timescale $t_g$}\label{ssec:tgap}
Theorem~\ref{thm:short_long} will be obtained from the following result.
\begin{prop}\label{thm:short_long'}
For all $u \in (0,\infty)$, $\kappa \in (0, 1)$, any (integer) $K \geq 0$, $f_1, f_2:\Z^d \to \R_+$ such that $u \ge f=f_1+f_2 \ge \kappa$ 
pointwise on $B_{K + L}$, any $L, L' \geq 1$ such that $L'' = L' + t_g \leq \frac L2$, where $t_g =  \lfloor L^{7/8} \rfloor$, and all $\varepsilon 
\in (0, 1)$, there exists a coupling $\widetilde{\mathbb Q}$ of $(\mathcal I_1, \mathcal I_2)$ such that
\begin{equation}
	\label{eq:easy_cover}
	\begin{split}
		&\mathcal I_1 \stackrel{\textnormal{law}}{=} \mathcal I^{\frac12(1 + P_{L/2})f_1, \frac{L}{2}} \cup \mathcal I^{f_2, L}, \quad \mathcal I_2 \stackrel{\textnormal{law}}{=} 
		\mathcal I^{(1-\varepsilon)f'', L'} \text{ and }\\
		&	\widetilde{\mathbb Q} \left[ {\mathcal I}_1  \supset  {\mathcal I}_2  \right] \geq 1 -  C (u \vee 1)(K + L)^d e^{- c\kappa\varepsilon^2 L^{c}},
	\end{split}
\end{equation}
where ${\mathcal I}^{\frac12(1 + P_{L/2})f_1, \frac{L}{2}}$ and ${\mathcal I}^{f_2, L}$ are independent in the law defining $\mathcal I_1 $, and 
\begin{equation}
	\label{eq:f''_easy}
	f'' \stackrel{\textnormal{def.}}{=} \textstyle l^{-1} \,\sum_{0\leq k < \lfloor{\frac {L}{2L''}}\rfloor} 
	\big[P_{kL''}(f1_{B_K}) + P_{L/2 + kL'' + t_g}(f1_{B_K})\big].
\end{equation}
\end{prop}

\bigskip

Assuming Proposition~\ref{thm:short_long'} to hold we now present the:
\begin{proof}[Proof of Theorem~\ref{thm:short_long} (assuming Proposition~\ref{thm:short_long'})] First observe that we can always assume $L' \le cL$ and $L \geq C$ (often implicit in the sequel); for, in all other cases choosing $\mathbb{Q}$ any coupling between $(\mathcal{I}_1, \mathcal{I}_2)$, the conclusion \eqref{eq:short_long} trivially holds. By choosing the constant $c$ small enough in the condition $L' \ge L^{1 - c}$ and $L \geq C$, we may assume that  
\begin{equation}\label{eq:gap_bnd1}
(L')^{2} > Lt_g /2.
\end{equation}
In particular, \eqref{eq:gap_bnd1} implies that $L' > t_g$ when $L \geq C$. Deducing 
Theorem~\ref{thm:short_long} involves replacing ${\mathcal{I}}_2$ in \eqref{eq:easy_cover} by the 
corresponding quantity in \eqref{eq:short_long}, with 
underlying intensities $f''$ and $f'=P_L^{L'}(f)$ given by \eqref{eq:f''_easy} and \eqref{eq:f'}, 
respectively. We will take care of the discrepancy in the definition of $f'$ and $f''$ in two steps: first, 
adjusting the times in \eqref{eq:f''_easy} at which the heat kernels are evaluated to be suitable 
multiples of $L'$, and second adjusting the summation over $k$.

In view of \eqref{eq:f'} and \eqref{eq:f''_easy}, we first compare $P_{kL''}$ to $P_{kL'}$ and $P_{L/2 + kL'' + t_g}$ to $P_{L/2 + kL'}$. Using \eqref{eq:LCLT}-\eqref{eq:density}, one obtains, for any integer $n \geq 1$,
$\Delta \in (n^{-2/3}, 1)$ and $|x| \le \sqrt{n} 
\Delta^{-1/4}$, whence $\frac{|x|^4}{n^3}= O(n^{1/3})= O(\Delta^{-1/2})$ as $n \to \infty$, that
\begin{align}
	\label{eq:regularity} \Big| \textstyle \frac{p_n(0, x)}{p_{\lfloor n(1 + \Delta) \rfloor}(0, x)} -1 \Big| \le C \Delta^{1/2}.
\end{align}
We apply \eqref{eq:regularity} with the choice $\Delta=\frac{kt_g }{ kL'} = \frac{t_g }{ L'}$ and $n=kL'$ 
for $1 \le k < \lfloor \frac{L}{2L''} \rfloor$ to compare $P_{kL''}$ to $P_{kL'}$. Notice that $\Delta \in (n^{-2/3}, 1)$ with these choices when $L \geq C$, as follows from the fact that $L' > t_g$ noted below \eqref{eq:gap_bnd1}. Similarly, we compare $P_{L/2 + kL'' + t_g}$ to $P_{L/2 + kL'}$ with the choice $\Delta = \frac{(k + 
	1)t_g}{L/2 + kL'} \le \frac{t_g}{L'}$ for $0 \le k < \lfloor \frac{L}{2L''} \rfloor$. Since $f \leq u$ 
on $B_{K}$, we then have for any $1 \le k < \lfloor L/2L'' \rfloor$, pointwise on $B_{K+L}$,
\begin{multline}\label{eq:regularity_bnd}
	\textstyle\big(1 - C\big(\frac{t_g }{ L'} \big)^{1/2}\big )P_{kL'}(f 1_{B_K}) \\ \le P_{kL''}(f 1_{B_K}) + \textstyle uP_0\big [\,|X_{kL'}| > c\sqrt{kL'} \big(\frac{L'}{ t_g}\big)^{1/4}\,\big]  \le P_{kL''}(f 1_{B_K}) + u\, e^{-c (L' / t_g)^{1/2}},
\end{multline}
and the same bound holds true with $P_{L/2 + kL'}$ and $P_{L/2 + kL'' + t_g}$ in place of $P_{kL'}$ and $P_{kL''}$. From \eqref{eq:regularity_bnd}, one 
immediately infers the existence for every $\varepsilon \in (0, \sqrt{t_g / L'})$ of a coupling between three random sets 
distributed as ${\mathcal I}^{(1-\varepsilon)f'', L'}$, $ {\mathcal I}^{u_11_{B_{K+L}}, L'}$ and $ {\mathcal I}^{(1- 
	C\sqrt{t_g / L'})\tilde{f}', L'} $ respectively, with 
$f''$ as in \eqref{eq:f''_easy}, $u_1\stackrel{\textnormal{def.}}{=} u l e^{-c (L' / t_g)^{1/2}}$ and
\begin{equation}
	\label{eq:tildef'-easy}
	\tilde f' \stackrel{\textnormal{def.}}{=}  
	l^{-1} \, \textstyle \sum_{0 \leq k < \lfloor{\frac {L}{2L''}}\rfloor} 
	\big(P_{kL'}(f1_{B_K}) + P_{L/2 + kL'}(f1_{B_K})\big),
\end{equation}
such that, in accordance with the resulting bound in \eqref{eq:regularity_bnd}, the former two are 
independent and their union contains the latter a.s. Applying Lemma~\ref{lem:concatenation}, one 
then concatenates this coupling with the one from Proposition~\ref{thm:short_long'} (their common 
marginal being $({\mathcal I}^{(1-\varepsilon)f'', L'},  {\mathcal I}^{u'1_{B_{K+L}}, L'})$, the latter being 
sampled independently by suitable extension of $\widetilde{\mathbb{Q}}$ in the context of 
Proposition~\ref{thm:short_long'}) with the choice $\varepsilon = L^{-c'}$, $\kappa = L^{-c''}$ for suitable $c', c'' >0$, to find a coupling $\mathbb Q_1$ of three random variables $\mathcal I_1$, $\tilde{\mathcal I}_2$ 
and ${\mathcal I}_3^1$ such that
\begin{equation}
	\label{eq:easy_cover2}
	\begin{split}
		&\mathcal I_1 \stackrel{\text{law}}{=} \mathcal I^{\frac12(1 + P_{L/2})f_1, \frac{L}{2}} \cup {\mathcal I}^{f_2, L} , \quad \tilde{\mathcal I}_2 \stackrel{\text{law}}{=} {\mathcal I}^{(1- C\sqrt{t_g / L'})\tilde f', L'} \text{ and } {\mathcal I}_3^1 \stackrel{\text{law}}{=} {\mathcal I}^{u_1 1_{B_{K+L}}, L'} \text{ with }\\
		&	\mathbb Q_1 \big[ {\mathcal I}_1  \cup {\mathcal I}_3^1 \supset  \tilde {\mathcal I}_2\big] \geq 1 -  C (u \vee 1) (K + L)^de^{- c u L^{c}}.
	\end{split}
\end{equation}
In words, \eqref{eq:easy_cover2} asserts that, at the cost of adding ${\mathcal I}_3^1$, which has 
intensity $u_1 (\ll u)$, one can replace the intensity $\tilde{f}$ by $\tilde{f}'$ in 
Proposition~\ref{thm:short_long'}. Now, comparing \eqref{eq:f'} and \eqref{eq:tildef'-easy}, observe that,
\begin{align}
\label{eq:last_pieces}
P_{L}^{L'}(f1_{B_K}) = \tilde f' + \textstyle l^{-1}\sum_{ \lfloor{\frac {L}{2L''}}\rfloor \le k < \frac {L}{2L'}} \big(P_{kL'} (f1_{B_K}) +  P_{L/2 + kL'} (f1_{B_K})\big)
\end{align}
But since $L''=L'+t_g$, and due to \eqref{eq:gap_bnd1}, one has $0 < \frac {L}{2L'} - \lfloor{\frac {L}{2L''}}\rfloor \le \frac {L}{2L'} - \frac {L}{2L''}   + 1 \le \frac{L t_g}{2(L')^2} + 1 < 2$.
Consequently \eqref{eq:last_pieces} can be recast as
\begin{align}
\label{eq:last_piece}
P_{L}^{L'}(f1_{B_K})  = \tilde f' +l^{-1}(P_{L/2 - L'} (f1_{B_K}) + P_{L - L'} (f1_{B_K})).
\end{align}
One then applies \eqref{eq:regularity} with the choice $\Delta = (k - 1)L' / (L - kL')$ and $n = L - kL'$, 
for $2 \le k \le \sqrt{L / L'}= \sqrt{l}$, which as above can be seen to satisfy the condition $\Delta \in 
(n^{-2/3}, 1)$ as soon as $ L^{1-c'}\leq L' \leq c' L$ for small enough $c'$. Arguing similarly as in 
\eqref{eq:regularity_bnd}, one obtains for all such $k$,
\begin{multline}\label{eq:extra_piece}
\big(1 - Cl^{-\frac12}\big )P_{L - L'}(f 1_{B_K}) \\
\le P_{L - kL'}(f 1_{B_K}) + uP_0\big [\,|X_{L - L'}| > c\sqrt{L} l^{1/8}\,\big]  \le P_{L-kL'}(f 1_{B_K}) + 
u\, e^{-c l^{1/4}}
\end{multline}
pointwise on $B_{K + L}$. The same holds with $P_{L/2 - L'}(f1_{B_K})$ on the left-hand side and $P_{L/2 - kL'}(f1_{B_K})$ on the right. One then averages over $k$ on the right-hand side of \eqref{eq:extra_piece} noting with a view towards \eqref{eq:tildef'-easy} that $\{ L-kL': 2 \le k \le \sqrt{l} \} \subset  \{ L/2 + kL'  : 0 \leq k < \lfloor{\frac {L}{2L''}}\rfloor \}$, and similarly that $\{ \frac L2-kL': 2 \le k \le \sqrt{l} \} \subset  \{ kL'  : 0 \leq k < \lfloor{\frac {L}{2L''}}\rfloor \}$.
Together with \eqref{eq:last_piece} and \eqref{eq:tildef'-easy}, this gives
\begin{equation}\label{eq:extra_piece2}
\begin{split}
P_{L}^{L'}(f1_{B_K})  &\le \tilde f' + \big(Cl^{-\frac12}\tilde f' + C u\, e^{-c l^{1/4}}\big) = (1 + 
Cl^{-\frac12})\tilde f' + C \, e^{-c l^{1/4}}.
\end{split}
\end{equation}
Since $l^{1/4}=(L / L')^{1/4} \leq (2L' / t_g)^{1/2}$ in view of \eqref{eq:gap_bnd1}, proceeding analogously as in the argument leading from \eqref{eq:regularity_bnd} to \eqref{eq:easy_cover2}, with 
\eqref{eq:extra_piece2} and \eqref{eq:easy_cover2} now playing similar roles as \eqref{eq:regularity_bnd} and the coupling \eqref{eq:easy_cover} from Proposition~\ref{thm:short_long'}, respectively, one finds a coupling $\mathbb Q_2$ of three random variables $\mathcal I_1$, ${\mathcal I}_2$ and ${\mathcal I}_3$ 
such that
\begin{equation*}
	\begin{split}
		&\mathcal I_1 \stackrel{\text{law}}{=} \mathcal 
		I^{\frac12(1 + P_{L/2})f_1, \frac{L}{2}} \cup {\mathcal I}^{f_2, L}, \quad {\mathcal I}_2 \stackrel{\text{law}}{=} {\mathcal I}^{(1- Cl^{-1/2})P_L^{L'}(f1_{B_K}), L'} \text{ and } \mathcal 
		I_3 \stackrel{\text{law}}{=} {\mathcal I}^{u_21_{B_{K+L}}, L'} \text{ with }\\
		&	\mathbb Q_2 \left[ {\mathcal I}_1  \cup {\mathcal I}_3 \supset  {\mathcal I}_2\right] \geq 1 -  C(u \vee 1)(K + L)^de^{- c\kappa L^{c}},
	\end{split}
\end{equation*}
where $u_2 \stackrel{\textnormal{def.}}{=} C u l e^{-c l^{1/4}}$. However, since $\mathbb Q[{\mathcal I}_3 = \emptyset] \ge 1 - C(K+L)^du_2$ by standard properties of Poisson variables, we immediately 
obtain \eqref{eq:short_long} with $\mathbb{Q}=\mathbb{Q}_2$, thus yielding 
Theorem~\ref{thm:short_long}. 
\end{proof}

\subsection{Proof of Proposition~\ref{thm:short_long'}}\label{S:proofofeasycoupling}
We will in fact prove a slightly stronger statement, namely that the bound in the second line of \eqref{eq:easy_cover} holds with $\mathcal I_1$ replaced by a (smaller) 
configuration $\hat{\mathcal I}_1 \leq_{\textnormal{st.}} {\mathcal{I}}_1$ to be defined in the course of the proof, see \eqref{eq:stoch_dom} and \eqref{def:hatJ_1k}. By chaining (cf.~\S\ref{A:chaining}) the original statement in~\eqref{eq:easy_cover} then quickly follows. Roughly speaking, the set $\hat{\mathcal I}_1$ comprises fragments of length $L'$ from the length-$\frac{L}{2}$ and length-$L$ trajectories making up ${\mathcal{I}}_1$, cf.~\eqref{eq:easy_cover}, which are merged with corresponding pieces of ${\mathcal{I}}_2$ indexed by $k$ in \eqref{eq:f''_easy}. The merging happens recursively in $k$ in terms of a sequence of couplings supplied by Lemma~\ref{L:main-coupling-easier} below. Taking advantage of Proposition~\ref{thm:short_long'} over Theorem~\ref{thm:short_long}, the fragments of longer trajectories alluded to above can now be separated by a time roughly of order $t_g$. This allows for good mixing at an intermediate scale $\widetilde{L}$, chosen suitably in Lemma~\ref{lem:easy_coup_bnd}, which is ultimately responsible for the good control on the error term, as outlined in the discussion following \eqref{eq:no-loop-l}. We forewarn the perceptive reader that the somewhat intricate definition of ${\mathcal{I}}_1$ in \eqref{eq:easy_cover} or \eqref{eq:short_long} (required for later purposes, see Remark~\ref{R:short_long},\ref{R:profile-compl}) makes the proof slightly more involved. An option is to set $f_1 \equiv 0$, which already yields an interesting special case of Proposition~\ref{thm:short_long'} and effectively makes all matters relating to $\overline{\omega}_1$ and $\overline{\omega}_3$ disappear from the proof below, thus leading to a streamlined argument.

\begin{proof}[Proof of Proposition~\ref{thm:short_long'}]
	The measure $\widetilde{\mathbb{Q}}$ under which the desired coupling will be constructed is assumed to carry independent Poisson processes $\overline{\omega}_j$, $1\leq j \leq 3$ and $\omega^k$, $1 \leq k < 2 k_L$, where $k_L \stackrel{\text{def.}}{=} \lfloor L/2L'' \rfloor$. All processes have state space $[0,1] \times \Z^d \times W^+$, and the processes $\overline{\omega}_j$, $1 \leq j \leq 3$, have respective intensity measure 
	\begin{equation}
		\label{eq:easy_cover0}
		\begin{split}
			&\overline{\nu}_j([0,u] \times S \times A) \stackrel{\textnormal{def.}}{=} u \cdot \frac{4d}{L}\sum_{x \in S} f_j(x) P_x[X \in A],
		\end{split}
	\end{equation}
	for $u \in (0,1]$, $ S\subset \Z^d$ and $A \in \mathcal{W}^+$, with $f_1$, $f_2$ as appearing in the statement of Proposition~\ref{thm:short_long'} and $f_3 = P_{L/2} f_1$; compare with the definition of ${\mathcal{I}}_1$ in \eqref{eq:easy_cover}. The process $\omega^k$ has intensity
	\begin{equation}
		\label{eq:easy_cover1}
		\nu^k([0,u] \times S \times A) \stackrel{\textnormal{def.}}{=}u \cdot \frac{4d }{L}\sum_{x \in S} (P_{t_k} 
		f)(x) P_x[ Y \in A],
	\end{equation}
	with $Y= X \circ \theta_{t_g+1} = (X_{t_g+1+n})_{n \geq 0}$ and (recall that $t_g = L'' - L'$) 
	\begin{equation}\label{def:tk}
		t_k \stackrel{\textnormal{def.}}{=} \begin{cases}
			kL'' - (t_g+1); & \mbox{for } 1\leq k <k_L, \\
			\frac L2 +  (k -k_L)L'' - 1; & \mbox{for $k_L \leq k < 2k_L$}.
		\end{cases}
	\end{equation}
	In the sequel, with $\pi: [0,1] \times \Z^d \times W^+ \to [0,1] \times W^+$ 
	denoting the canonical projection map onto the first and third coordinates, we write $\overline{\eta}_j= \pi \circ \overline{\omega}_j$,  $\eta^k = \pi \circ \omega^k$ for the corresponding push-forward processes on $[0,1] 
	\times W^+$. To understand the relevance of  $\omega^k$, notice that, substituting $S=\Z^d$ in \eqref{eq:easy_cover1} and using reversibility and the semigroup property, by which, for $k< k_L$, 
	$$
	\sum_{x} (P_{t_k} 
	f)(x) P_x[ Y \in A] = \sum_y f(y) P_y[ X \circ \theta_{t_g+1+ t_k} \in A] \stackrel{\eqref{def:tk}}{=} \sum_x (P_{kL''}f)(x)P_x[X \in A]
	$$
	(all sums ranging over $\Z^d$), along with a similar computation when $k \geq k_L$, observing in the latter case that $t_g+1+ t_k= t_g + \frac{L}2 + (k-k_L)L''$, it follows that $\eta^k$ intensity
	\begin{equation}
		\label{eq:projection0}
		\begin{split}
			&\hat{\nu}^k([0,u] \times A)=  u\cdot  \frac{4d }{L}\sum_{x } (\Pi_k f)(x) P_x[X \in A], \text{ where }\\
			&\Pi_k f=  \begin{cases} P_{kL''} f, & \mbox{if }0< k < k_L, \mbox{}\\
				P_{\frac L2 + k'L'' + t_g} f;& \mbox{if } k = k' + k_L \mbox{ for } 0 \leq k' < k_L
			\end{cases}
		\end{split}
	\end{equation}
	(compare with \eqref{eq:f''_easy}). 
	The sets $\hat{\mathcal I}_1$ (smaller than ${\mathcal{I}}_1$) and ${\mathcal I}_2$ that we 
	aim to couple under $\widetilde{\mathbb{Q}}$ will be defined as unions over several parts $\hat{\mathcal I}_1^k$ and ${\mathcal I}_2^k$ indexed by $k$, where $0 \leq k 
	< 2k_L$. Using $\overline{\omega}_j$, $j = 1, 2$, the 
	first of these contributions (corresponding to $k=0$) is simply defined as follows: 
	\begin{align}
		\label{eq:def_first_part}
		\begin{split}
			&\hat{\mathcal I}_1^0= \hat{\mathcal I}_1^0(\overline{\omega}_1 \cup \overline{\omega}_2) \stackrel{\textnormal{def.}}{=} \bigcup_{(v, w) \in\overline{\eta}_1 \cup\overline{\eta}_2} w[0,L'-1], \mbox{ and } \\
			&\mathcal I_2^0= \mathcal I_2^0(\overline{\omega}_1 \cup \overline{\omega}_2) \stackrel{\textnormal{def.}}{=} \bigcup_{\substack{(v, x, w) \in \overline{\omega}_1 \, \cup \, \overline{\omega}_2 :\\ v \leq 1 - \varepsilon, \,  x \in  B_K}} w[0,L'-1].
		\end{split}
	\end{align}
	Recalling \eqref{eq:J}, the fact that $f=f_1+f_2$ and using \eqref{eq:easy_cover0}, it follows that $\hat{\mathcal I}_1^0$ and $\mathcal I_2^0$ are distributed as
	${\mathcal I}^{\frac{L'}{L}f , L'}$ and ${\mathcal I}^{(1 - \varepsilon)f''_0 , L'}$ 
	respectively, with $f_0'' \stackrel{\textnormal{def.}}{=} \frac{L'}{L}f 1_{B_K}$. 
	To understand what the former has to do with ${\mathcal{I}}_1$ in \eqref{eq:easy_cover}, observe that both ${\mathcal I}^{\frac{L'}{L}f , L'}$ and ${\mathcal I}^{\frac{1}{2}f_1 , \frac L2} \cup {\mathcal{I}}^{f_2,L}$ (sampled independently) can be obtained by retaining a certain number of steps from each trajectory in the support of a Poisson process on $W^+$ of intensity $\frac{4d}{L} \sum_x f(x)P_x$. In particular, the laws of their starting points coincide.
	
	In the sequel it will be convenient to let $\zeta^0_{v, w}= w  (\in W^+)$ for $(v, w) \in \overline{\eta}_1 \cup\overline{\eta}_2$ and define $\mathcal{W}^0 =\{\zeta^0_{v, w}: (v, w) \in \overline{\eta}_1 \cup\overline{\eta}_2\}$, by which $\hat{\mathcal I}_1^0$ can be viewed as a function of $\mathcal{W}^0$. 
	
	The next lemma constructs recursively a sequence $(\mathcal W^{k}, \mathcal I_2^{k})_{0\leq k <2k_L}$ under $\widetilde{\mathbb{Q}}$, where $\mathcal W^{k}$ is a certain family of random paths $\zeta^k_{v, w}$ (i.e.~having values in $W^+$) indexed by points $(v,w) \in \overline{\eta}_{j(k)} \cup \overline{\eta}_{2}$, where $j(k)=1$ if $0 \leq k<k_L$ and $j(k)=3 $ if $k_L \leq k < 2k_L$. In terms of $\mathcal W^{k}$, the set $\hat{\mathcal I}_1^k$ is obtained by setting
	\begin{equation}\label{def:hatJ_1k}
		\hat{\mathcal I}_1^k = \bigcup_{(v, w) \in\overline{\eta}_{j(k)} \cup\overline{\eta}_{2}} \zeta^k_{v, 
			w}[0,L'-1],\end{equation} 
which is consistent \eqref{eq:def_first_part} when $k=0$. The requirements on $(\mathcal W^{k}, \mathcal I_2^{k})$ vary depending on the value of $k$ below, which reflects \eqref{def:tk}-\eqref{eq:projection0} and is owed to the specific form of ${\mathcal{I}}_1$ in \eqref{eq:easy_cover}.

\begin{lemma} \label{L:main-coupling-easier} There exists a sequence $(\mathcal W^{k}, \mathcal I_2^{k})_{0\leq k <2k_L}$ under $\widetilde{\mathbb{Q}}$ with the following properties:
	\begin{enumerate}[label = {(\roman*)}]
		\item For all $0< k <2k_L$, conditionally on $(\mathcal W^0, \mathcal 
		J_2^0), \ldots, (\mathcal W^{k-1}, \mathcal I_2^{k-1})$, $\mathcal{W}^k \stackrel{\textnormal{def.}}{=} \{\zeta^k_{v, w} \in 
		W^+: (v, w) \in \overline{\eta}_{j(k)} \cup\overline{\eta}_{2}\}$ is a family of independent random variables and
		\begin{equation}\label{eq:zeta_cond-law}
			\begin{split}
				&\zeta_{v, 
					w}^k \stackrel{\textnormal{law}}{=} P_{\zeta_{v, w}^{k-1}(L'-1)} [ Y 
				\in \cdot], \text{ if $k \neq k_L$; }   \\
				&\zeta_{v, 
					w}^{k_L}  \stackrel{\textnormal{law}}{=}
				\begin{cases} P_{w(0)} [ X\circ{\theta}_{t_g} \in \cdot], & \text{if $(v,w) \in\overline{\eta}_3$,}\\
					P_{\zeta_{v, w}^{k-1}(
						\Delta)}[ Y \in \cdot], & \text{if $(v,w) \in\overline{\eta}_2$}, \text{ where $\Delta= t_{k_L} -t_{k_L-1} -(t_g+1) $.}
				\end{cases}
			\end{split}
		\end{equation} 
		\item For all $0< k <2k_L$, conditionally on $(\mathcal W^0, \mathcal I_2^0), \ldots, (\mathcal W^{k-1}, \mathcal I_2^{k-1})$, the 
		set $\mathcal I_2^{k}$ is distributed as ${\mathcal I}^{(1 - \varepsilon)f''_k, L'}$, with $f''_k = \frac{L'}{L}\Pi_k (f1_{B_K})$ (cf.~\eqref{eq:projection0}).
		
		\item 
		The bound  $\widetilde{\mathbb Q}[ \hat{\mathcal I}_1^k \not \supset  {\mathcal I}_2^k] \le C|B_{K+L}|e^{- 
			c\kappa\varepsilon^2 L^{c}}$ holds for all $0 < k < 2k_L$.
	\end{enumerate}
\end{lemma}

We defer the proof of Lemma~\ref{L:main-coupling-easier} for a few lines and first complete the proof of Proposition~\ref{thm:short_long'} assuming it to hold. With $\mathcal{W}^k =\{\zeta^k_{v, w} \in 
W^+: (v, w) \in \overline{\eta}_{j(k)} \cup\overline{\eta}_{2}\}$ defined by Lemma~\ref{L:main-coupling-easier} for all $0 < k < 2k_L$, the set $\hat{\mathcal I}_1^k $ in \eqref{def:hatJ_1k} is declared and one sets
\begin{align}
	\label{eq:stoch_dom}
	\hat{\mathcal {I}}_1= \bigcup_{0 \le k < 2k_L}\hat{\mathcal {I}}_1^k
\end{align}
(recall that $\hat{\mathcal {I}}_1^0$ is supplied by \eqref{eq:def_first_part}). Similarly, using the sets $\mathcal I_2^{k}$, $0< k < 2k_L$, furnished by Lemma~\ref{L:main-coupling-easier} and ${\mathcal{I}}^2_0$ from \eqref{eq:def_first_part}, one defines
\begin{align}
	\label{eq:I2}
	\mathcal {I}_2 =\bigcup_{0 \le k < 2k_L}\mathcal {I}_2^k 
\end{align}
It then immediately follows by combining the fact that $\hat{\mathcal {I}}_1^0 \supset {\mathcal {I}}_2^0$, which is plain from \eqref{eq:def_first_part}, with item \textit{(iii)} above and a union bound over $k$, noting that $2k_L \leq 2 \frac{L}{L'} = 2l$ (see above \eqref{eq:easy_cover0} regarding $k_L$), that $\widetilde{\mathbb Q}[ \hat{\mathcal I}_1  \supset  {\mathcal I}_2  ]$ is bounded from below by the right-hand side of \eqref{eq:easy_cover}. To conclude the proof, it is thus enough to argue that ${\mathcal{I}}_2$ defined by \eqref{eq:I2} is equal to ${\mathcal I}^{(1 - \varepsilon)f'', L'}$ in law, as prescribed by \eqref{eq:easy_cover}, and that $\hat{\mathcal {I}}_1 \leq_{\text{st.}} {\mathcal {I}}_1$. From this, the original statement in~\eqref{eq:easy_cover} is readily obtained by chaining the coupling $\widetilde{\mathbb{Q}}$ and the one inherent to the stochastic domination using Lemma~\ref{lem:concatenation}; see also Remark~\ref{R:chain},\ref{R:chain-appli}.

The fact that ${\mathcal{I}}_2$ in \eqref{eq:I2} satisfies ${\mathcal{I}}_2 \stackrel{\text{law}}{=}{\mathcal I}^{(1 - \varepsilon)f'', L'}$ is an immediate consequence of \textit{(ii)}, recalling 
$\Pi_k$ from~\eqref{eq:projection0}, by which $f''_k = \frac{L'}{L}P_{kL''}(f1_{B_K})$ if $k <k_L$ (see also below \eqref{eq:def_first_part} regarding $f_0''$) and $f_k'' =  \frac{L'}{L}P_{L + k'L'' + t_g}(f1_{B_K})$ if $k = k' + k_L$ for $k' \geq0$, which implies in turn that $\sum_{0 \leq k< 2k_L} f_k''= f''$ as defined by \eqref{eq:f''_easy}.

We now argue that $\hat{\mathcal {I}}_1 \leq_{\text{st.}} {\mathcal {I}}_1$. To see this, it is simplest to think of $\mathcal I_1$ in \eqref{eq:easy_cover} as the union of the three independent sets ${\mathcal I}^{\frac12 f_1, \frac{L}{2}}$, ${\mathcal I}^{\frac12 P_{L/2}f_1, \frac{L}{2}}$ and ${\mathcal I}^{f_2, L}$. We first consider the contributions to each $\hat{\mathcal I}_1^k$ in \eqref{def:hatJ_1k} stemming from points $(v,w) \in \overline{\eta}_3$. As we now explain, these generate an independent set which is dominated by ${\mathcal I}^{\frac12 P_{L/2}f_1, \frac{L}{2}}$. Since $j(k)=3$ only when $k \geq k_L$, contributions of this kind only arise for these values of $k$. For a given $(v,w) \in \overline{\eta}_3$ then, concatenating the pieces $\zeta^k_{v, 
	w}[0,L'-1]$ in \eqref{def:hatJ_1k} for $k_L \leq k < 2k_L$ while making repeated use of \textit{(i)} yields a set having the same law as 
$$
\bigcup_{0 \leq k  < k_L} X[kL'' + t_g , (k+1)L''-1] \ \big(\subset X\big[0, \textstyle\frac{L}{2}-1\big] \big)
$$
under $P_{w(0)}$. As these fragments of random walk trajectories are independent as $w$ varies, the claim readily follows upon recalling that $\overline{\eta}_3$ has intensity $\pi \circ \overline{\nu}_3$ with $\overline{\nu}_3$ given by \eqref{eq:easy_cover0}. In the same vein, one verifies that the contributions to all $\hat{\mathcal I}_1^k$ stemming from points $(v,w) \in \overline{\eta}_1$ yield an independent set dominated by ${\mathcal I}^{\frac12 f_1, \frac{L}{2}}$, and similarly those in $\overline{\eta}_2$ account for ${\mathcal I}^{f_2, L}$. 
This completes the proof of Proposition~\ref{thm:short_long'}, assuming Lemma~\ref{L:main-coupling-easier}.\end{proof}

\begin{proof}[Proof of Lemma~\ref{L:main-coupling-easier}]
We define the sets $(\mathcal{W}^k, \mathcal I_2^k)$, $0\leq k <  2 k_L$ inductively such that \textit{(i)}-\textit{(iii)} above hold. In addition, the induction will also carry the property that
\begin{equation}\label{e:easy-ind_meas}
	\text{the pair $(\mathcal{W}^k, \mathcal I_2^k)$ is measurable relative to 
	$ \mathcal{F}_k = \sigma\big(\overline{\omega}_{1},\overline{\omega}_{2},\overline{\omega}_{3} , \omega^1, \ldots, \omega^k\big)$}
\end{equation}
(understood as 
$\sigma ((\overline{\omega}_{j})_{1 \leq j \leq 3})$ when $k=0$). The pair $(\mathcal{W}^0, \mathcal I_2^0)$ defined around \eqref{eq:def_first_part} plainly satisfies \eqref{e:easy-ind_meas} and properties \textit{(i)}-\textit{(iii)} are trivially satisfied for $k = 0$.

For arbitrary $k \geq 1$, we now describe how to sample $(\mathcal W^{k}, \mathcal I_2^{k})$ conditionally on $(\mathcal W^0, \mathcal I_2^0)$,\dots,
$(\mathcal W^{k-1}, \mathcal{I}_2^{k-1})$. In doing so, Properties~$(i)$ and $(ii)$, which are the relevant requirements on the (conditional) laws of $\mathcal W^{k}$ and $\mathcal I_2^{k}$ 
, respectively, along with \eqref{e:easy-ind_meas}, will immediately follow. The  construction is such that the (key) property \textit{(iii)} holds, which is verified separately.

The construction of $(\mathcal W^{k}, \mathcal I_2^{k})$ involves $\omega^k$, which we now recall from \eqref{eq:easy_cover1}-\eqref{def:tk}. For any $S \subset \Z^d$, let $\omega_S^k$ denote the restriction of $\omega^k$ to points lying in the slab $\{(v, x, w): x \in 
S\}$. It follows by \eqref{eq:easy_cover1} that the projection $\eta_S^k =\pi \circ  \omega_S^k$ is a Poisson process on $[0,1 ]\times W_+$ with intensity 
\begin{equation}
\label{eq:projection}
\hat{\nu}_S^k([0,u] \times A) = u \mu_S^k(A), \text{ with }  
\mu_S^k(A)\stackrel{\textnormal{def.}}{=}\frac{4d}{L}\sum_{x \in S} (P_{t_k} f)(x) P_x[Y \in A],
\end{equation}
for $A \in \mathcal{W}^+$. Similarly let $\tilde \mu_S^k$ denote the measure on $(W^+,\mathcal{W}^+)$ defined by
\begin{equation}
\label{eq:mu2k}
\tilde \mu_S^k(\cdot)=\frac {4d}{L}\sum_{x \in S}P_{t_k}(f1_{B_K})(x)P_x[Y \in \cdot].
\end{equation}
We now first describe how to sample ${\mathcal{I}}^2_k$, which only involves $\omega^k$.
For any $S \subset\subset \Z^d$, let
\begin{align}
\label{eq:def_second_part_short0}
{\mathcal I}_{2, S}^k \stackrel{\textnormal{def.}}{=}  \bigcup_{(v, w) \in\eta_S^k
	:\, v \leq (1 - \varepsilon) \frac{d\tilde \mu_{S}^{k}}{d\mu_{S}^k}(w)} w[0,L'-1].
\end{align}
For an intermediate scale $\widetilde{L}$ with $1 \leq \widetilde{L} \leq \frac{L}{2}$ to be determined, we consider the partition of $\Z^d$ into disjoint boxes $D= B(x,\widetilde{L})$ as $x$ ranges over $(2\widetilde{L}+1)\Z^d$. Let $\mathcal{D}$ denote the set of all such boxes and $\mathcal{D}'$ the subset of all $D$ with $D \subset B_{K+L}$. This leaves a (possibly empty) boundary region $V = B_{K+L} \setminus \bigcup_{D \in \mathcal{D}'} D$, which satisfies $d(V, B_K) \geq \frac{L}2$ whenever $L \geq C$ irrespectively of the value of $K$. Now set
\begin{equation}
\label{eq:def_second_part_short}
\mathcal I_2^k= \mathcal I_2^k(\omega^k) \stackrel{\textnormal{def.}}{=}  \bigcup_{ D \in \mathcal{D}} {\mathcal I}_{2, D}^k = \Big( \bigcup_{D \in \mathcal{D}'} {\mathcal I}_{2, D}^k \Big) \cup {\mathcal{I}}_{2, V}^k.
\end{equation}
The last equality follows due to the indicator function present in \eqref{eq:mu2k} and the fact that $t_k \leq L$ for any $k$, see \eqref{def:tk}, which together imply that $\tilde{\mu}_S^k=0$ whenever $S \cap B_{K+L} = \emptyset$. It immediately follows on account of \eqref{eq:def_second_part_short0} and the computation leading from \eqref{eq:easy_cover1} to \eqref{eq:projection0} that $(ii)$ holds for $k$, i.e.~$\mathcal I_2^k$ has the required law.

In order to sample the walks in $\mathcal{W}^k$, on the other hand, 
we use a similar method as in the proof of 
Lemma~\ref{L:RI_basic_coupling}.  For any $D \in \mathcal{D}$, let (cf.~\eqref{eq:zeta_cond-law})
\begin{equation}
\label{eq:wsk}
\begin{split}
	&\Lambda_D^k = \{(v, w) \in \overline{\eta}_{j(k)} \cup \overline{\eta}_2: \zeta_{v, 
		w}^{k-1}(L'-1) \in D\}, \, k \neq k_L,\\
	&\Lambda_D^{k_L} = \{(v, w) \in \overline{\eta}_{2}: \zeta_{v, 
		w}^{k_L-1}(\Delta) \in D\}\cup \{(v, w) \in \overline{\eta}_{3}: w(0) \in D\}.
\end{split}
\end{equation}
By \eqref{e:easy-ind_meas} and induction hypothesis, $\Lambda_D^k$ is clearly $\mathcal{F}_{k-1}$-measurable for all $k \geq 1$. Moreover, due to items~\textit{(i)} and~\textit{(iii)} (valid up to $k-1$ by assumption) and the Markov property, the quantity $|\Lambda_D^k|$ is seen to be a Poisson variable with 
mean $\lambda_D$, where
\begin{multline*}
\lambda_D= \big(1_{\{k < k_L\}}\overline{\nu}_1 + \overline{\nu}_2 \big) \big([0,1] \times \Z^d \times \{ w(t_k) \in D \}\big)+\\  1_{\{k \geq k_L\}}\overline{\nu}_3 \big([0,1] \times \Z^d \times \big\{ w\big((k-k_L)L''-1\big) \in D \big\}\big)\stackrel{\eqref{eq:easy_cover0}}{=}
\textstyle\frac{4d}{L} \displaystyle \sum_{x \in \Z^d} \Big[ 1_{\{k < k_L\}}  f_1(x)P_x[X_{t_k} \in D] \\[-0,5em]
+f_2(x)P_x[X_{t_k} \in D] +  1_{\{k \geq k_L\}} (P_{\frac L2} f_1)(x)
\times P_x[X_{(k-k_L)L''-1} \in D] \Big]=
\textstyle \frac{4d}{L} \displaystyle \sum_{x \in D}(P_{t_k}f)(x),
\end{multline*}
where in the last step one uses reversibility and the semigroup property, noting for the third term that $\frac{L}{2}+(k-k_L)L''-1 =t_k$ when $k\geq k_L$ by \eqref{def:tk}. Now, conditionally on $\mathcal{F}_{k-1}$ and for any $(v, w) \in \Lambda_D^k$, let $g_{v, w}: W^+ \to 
\R_+$ denote the Radon-Nikodym density of the law of $\zeta_{v, w}^k$ prescribed by \eqref{eq:zeta_cond-law} with respect to $\mu_D^k$ in \eqref{eq:projection}. That is, for any $(v, w) \in \Lambda_D^k$ and $w' \in W^+$, 
\begin{equation}
\label{eq:RND}
g_{v, w} (w') \stackrel{\textnormal{def.}}{=} \frac{p_{ t_g +1 }(\zeta_{v, w}^{k-1}(L'-1), w'(0))}{\frac {4d}L\sum_{x \in D}(P_{t_k} f)(x)p_{t_g +1}(x, w'(0))},  \text{ for }k \neq k_L,
\end{equation}
where  we adopt the convention $0 / 0 = 1$, and in view of \eqref{eq:zeta_cond-law} and  \eqref{eq:wsk}, the numerator in \eqref{eq:RND} is replaced when $k=k_L$ by $p_{ t_g+1  }(\zeta_{v, w}^{k-1}(\Delta), w'(0))$ if $(v,w) \in \overline{\eta}_2$ and by $p_{ t_g  }(w(0), w'(0))$ if $(v,w) \in \overline{\eta}_3$. Still conditionally on $\mathcal{F}_{k-1}$ and on the event $\{|\Lambda_D^k| = m \}$, which is $\mathcal{F}_{k-1}$-measurable as argued below \eqref{eq:wsk}, fix an enumeration $(v_1, w_1), \ldots, (v_m, w_m)$ of the points in $\Lambda_D^k$ with corresponding densities $g_{i}=g_{v_i, w_i}$, and define a sequence of 
(random) functions 
$G_0 = 0$, $G_1,
\dots, G_{m} \equiv G_D$ on $W^+$ 
inductively as follows. For all $1 \leq n \leq m$,
\begin{equation}
\label{eq:sofloctim}
\begin{split}
	&\xi_n \stackrel{\textnormal{def.}}{=} \inf \big\{t \ge 0: \exists \, (v, w) \in {\eta}_D^k\setminus \{(u_j, z_j): 1 \le j < n\} \mbox{ s.t. } G_{n-1}(w) + tg_{n}(w) \ge v\big\},\\
	& G_n(\cdot) \stackrel{\textnormal{def.}}{=} G_{n-1}(\cdot) + \xi_n g_n(\cdot).
\end{split}
\end{equation}
By \cite[Proposition~4.3]{PopTeix}, one obtains (conditionally on $\mathcal{F}_{k-1}$ and on the event $\{|\Lambda_D^k| = m \}$) that $ \xi_1, \ldots, \xi_{m}$ are distributed as independent exponential random variables with mean $1$. Moreover, denoting by $(u_n^D, z_n^D)$ the unique pair $(v,w) \in {\eta}_D^k \setminus  \{(u_j, w_j): 1 \le j < n\} $ satisfying $G_n(z_n^D) = u_n^D$, it follows that $(z_1^D, \ldots, z_{m}^D)$ are independent and $z_i$ has the law of $\xi_{v_i, w_i}^k$ prescribed by \eqref{eq:zeta_cond-law}. Hence, letting
\begin{equation}
\label{eq:Wk}
\mathcal W^k= \mathcal W^k\big((\overline{\omega}_j)_{1\leq j \leq 3}, (\omega^j)_{1\leq j \leq k}\big) = \bigcup_{D\in \mathcal{D}} \{z_{i}^D : 1 \le i \le |\Lambda_{D}^k|\},
\end{equation}
item \textit{(i)} readily follows. Moreover, by \eqref{eq:def_second_part_short} and \eqref{eq:Wk}, \eqref{e:easy-ind_meas} is immediately verified.

It remains to argue that $\textit{(iii)}$ holds. In view of \eqref{def:hatJ_1k} and by definition of $\mathcal{W}^k$, one can recast $\hat{\mathcal I}_1^k $ as $ \bigcup_{ \zeta \in \mathcal W^k } \zeta[0,L'-1].$ Using \eqref{eq:def_second_part_short}, it then follows that
\begin{align}
\label{eq:inclusion_bnd}
\widetilde{\mathbb Q}\big[ \hat{\mathcal I}_1^k \not \supset  {\mathcal I}_2^k \, \big] \le \widetilde{\mathbb Q}\big[   {{\mathcal I}}_{2,V}^k \neq \emptyset \, \big] +  \sum_{D \in \mathcal{D}'} \widetilde{\mathbb Q}\big[{\eta}_D^k(O_D) > 0\big], 
\end{align}
where (cf.~\eqref{eq:def_second_part_short0})
\begin{align}
\label{eq:profile}
O_D \stackrel{\textnormal{def.}}{=} \Big\{(v, w) \in [0,1]\times W^+: \, G_{D}(w) < v \le  (1 - \varepsilon)\textstyle\frac{d \tilde \mu_{D}^k}{d \mu^k_{D}}(w) \Big\}.
\end{align}
We bound each term on the right-hand side of \eqref{eq:inclusion_bnd} separately. By definition, see \eqref{eq:def_second_part_short0}, the set $ {{\mathcal I}}_{2,V}^k $ is obtained by retaining the first $L'-1$ steps of independent random walk trajectories attached to a Poisson number of starting points with mean
\begin{multline*}
(1-\varepsilon)\tilde{\mu}_{V}^k(W^+) \stackrel{\eqref{eq:mu2k}}{\leq} \frac {4d}{L}\sum_{x \in V} E_x\big[(f1_{B_K})(X_{t_k})\big] \,  \leq  \, C\frac{u}{L}|V| \sup_{t<L} P_0\big[|X_t| > \textstyle \frac L2 \big] \leq C u |B_{K+L}| e^{-cL},
\end{multline*}
where the second bound is obtained by combining the facts that $t_k < L$ for all $k$ and that $d(V,B_K) \geq \frac L2$, see \eqref{def:tk} and above \eqref{eq:def_second_part_short}, and using the assumption $f \leq u$, while the last bound is readily implied 
by a standard Gaussian upper bound on the heat kernel (also recalling that $V \subset B_{K+L}$). Since for $X$ a Poisson variable with mean $\lambda$, one has $P[X > 0]=1-e^{-\lambda} \leq \lambda$, it follows that $\widetilde{\mathbb Q}\big[   {{\mathcal I}}_{2,V}^k \neq \emptyset \, \big] $ is 
also bounded from above by $Cu|B_{K+L}| e^{-cL}$. To estimate the remaining terms in 
\eqref{eq:inclusion_bnd}, one uses the following:

\begin{lem}
\label{lem:easy_coup_bnd}	
For $\widetilde{L}= L^{\frac{1}{d}+ \alpha}$ with $\alpha = \frac1{200}$, all $D = B(x, \widetilde{L}) \subset B_{K + L}$ and $1\leq k < 2k_L$,
\begin{equation}\label{eq:easy_coup_bnd}
\widetilde{\mathbb Q}\big[{\eta}_D^k(O_D) > 0\big] \le C (u \vee 1) e^{-c \kappa \varepsilon^2L^c}.
\end{equation}
\end{lem}
Choosing $\widetilde{L}$ as in Lemma~\ref{lem:easy_coup_bnd} (see also below \eqref{eq:def_second_part_short0}, where the parameter $\widetilde{L}$ is introduced), item \textit{(iii)} follows upon combining the estimate for $\widetilde{\mathbb Q}\big[   {{\mathcal I}}_{2,V}^k \neq \emptyset \, \big] $, the bound \eqref{eq:easy_coup_bnd} while noting that $|\mathcal{D}'| \leq |B_{K+L}|$, thus completing the proof of Lemma~\ref{L:main-coupling-easier}.
\end{proof}

\begin{proof}[Proof of Lemma~\ref{lem:easy_coup_bnd}]
Let $\beta = \frac{1}5$ and recall that $t_g= \lfloor L^{\frac 78} \rfloor$. By choice of $\widetilde{L}$, this implies that
\begin{equation}\label{e:Ltilde-cond}
(t_g)^{\frac12 - \beta} \ge L^{\frac 1d + 2\alpha} \ \big( = \widetilde{L} \cdot L^{\alpha}\big) 
\end{equation}
for all $L \geq C$; indeed, $(\frac 1d + 2\alpha)/(\frac12 - \beta)= \frac{5}{2d}+ 5\alpha \leq \frac{5}{6} + \frac{1}{40}< \frac78 $ for $d \geq 3$. The inequality \eqref{e:Ltilde-cond} should be read as the condition that a slightly sub-diffusive length scale corresponding to 
the gap time $t_g$ is still polynomially larger (in $L$) than the side length $\widetilde{L}$ of the box $D$. 
We first estimate the density $g_n$ entering $G_D$, cf.~\eqref{eq:profile} and \eqref{eq:sofloctim}. In view of \eqref{eq:RND} and the display preceding it, one has \begin{align}
\label{eq:density_bnd}
g_n(w) \ge \lambda_D^{-1}\rho(w),
\end{align}
where $\rho: W^+ \to \R_{+}$ is defined as
$$\rho(w) = \min_{z, z' \in D} \min_{t\in\{t_g, \, t_{g+1}\}}\frac{p_{t}(z, w(0))}{p_{t_g+1}(z', w(0))};$$
in obtaining \eqref{eq:density_bnd}, we also used that $g_n=g_{v_n,u_n}$ corresponds to points $(v_n,u_n) \in \Lambda_D^k$, which belong to $D$ by definition, see \eqref{eq:wsk}, hence the minimum over $z$ and $z'$ in the definition of $\rho$.
An estimate of the ratio $\rho$ can be obtained via \eqref{eq:LCLT} and
\eqref{eq:density}, using \eqref{e:Ltilde-cond}. Doing so, one deduces for any $w \in W^+$ such that $|w(0) - x| < (t_g)^{\frac12 + \beta}$, where $x$ denotes the center of $D$, that
\begin{align}
\label{eq:ratio_lower_bnd_easy}
\rho(w) \ge  1 - C \big( \textstyle \frac{t_g^{4(1/2 + \beta)}}{t_g^3}   - \frac{ \widetilde{L} t_g^{1/2 + \beta}}{t_g} \big) \stackrel{\eqref{e:Ltilde-cond}}{\ge}  1 - C L^{-c};
\end{align}
here, the first term in the parenthesis corresponds to the dominating error in \eqref{eq:LCLT}, whereas the second term accounts for the comparison of (continuous) heat kernels in \eqref{eq:density}, and we also used that for any $z,z' \in D= B(x, \widetilde{L})$,
\begin{equation*}
|w(0)-z|^2 - |w(0)-z'|^2 
\leq \big( |w(0)-x| + \widetilde{L} \big)^2 - \big( |w(0)-x| - \widetilde{L} \big)^2 \leq C |w(0)-x| \widetilde{L}.
\end{equation*} 
Denoting the set of walks $w\in W^+$ with $|w(0) - x| < t_g^{1/2 + \beta}$ by $ F$, observe that the cardinality of the set $S = \big\{(v, w) \in {\eta}_D^k: 
\, w \notin   F\big\}$
follows, under $\widetilde{\mathbb Q}$, a Poisson distribution with mean $
\mu_D^k(F^c)$; see~\eqref{eq:projection}. Recalling that $Y= X \circ{\theta_{t_g+1}}$ and using \eqref{e:Ltilde-cond}, it follows that $P_x[Y \in  F^c] \le C e^{-L^c}$ uniformly in $x \in D$, which, together with the bound $f \le u$ valid on $B_{K + L}$ and in view of~\eqref{eq:projection}, readily implies that
\begin{equation}\label{e:S-easy}
\widetilde{\mathbb Q}[S \ne \emptyset] \le C u e^{-L^c}.
\end{equation}
On the other hand, since the lower bound \eqref{eq:ratio_lower_bnd_easy} on $\rho(\omega)$ is in force for all $w$ such that $(v,w)\in \eta_D^k$ whenever $S=\emptyset$, it follows, recalling $O_D$ from \eqref{eq:profile} and $G_D$ from above \eqref{eq:sofloctim}, that
\begin{multline}
\label{eq:indep_sum}
\widetilde{\mathbb Q}\big[(\eta_{D}^k(O_D) > 0, \, S=\emptyset \big] \\\stackrel{\eqref{eq:density_bnd}}{\leq} \sup_{w \in  F} 
\widetilde{\mathbb Q}\big[\sum_{1 \le n \le |\Lambda_D^k|} \lambda_D^{-1}\rho(w) \xi_n \le  (1 
- \varepsilon) \textstyle\frac{d \tilde{\mu}^k}{d \mu^k}(w)\big ] 
\stackrel{\eqref{eq:ratio_lower_bnd_easy}}{\le} \displaystyle\widetilde{\mathbb Q}\big[\sum_{1 \le n \le |\Lambda_D^k|}\xi_n \le  (1 + CL^{-c} - 
\varepsilon)\lambda_D \big ],
\end{multline}
where we used that $ \frac{d \tilde{\mu}^k}{d \mu^k} \leq 1$, 
as readily implied by
the definitions of $\tilde{\mu}^k$ and $\mu^k$ in \eqref{eq:mu2k} and \eqref{eq:projection}. 
The upper bound on the sum of $\xi_n$'s in the last bound of \eqref{eq:indep_sum} can further be replaced by $(1-\frac{\varepsilon}{2}) \lambda_D$ whenever $\varepsilon \ge 2 
CL^{-c}$, which is no loss of generality in view of the desired estimate in \eqref{lem:easy_coup_bnd}. Now recalling that $ |\Lambda_D^k|$ is a Poisson variable with parameter $\lambda_D$ (see below \eqref{eq:wsk}) and that the $\xi_n$'s are i.i.d.~exponential variables with mean $1$ (see below \eqref{eq:sofloctim}), applying standard large deviation estimates yields that
\begin{multline}
\label{eq:indep_sum_decomp}
\widetilde{\mathbb Q}\big[(\eta_{D}^k(O_D) > 0, \, S=\emptyset \big]\\ \le \widetilde{\mathbb Q}\big[\sum_{1 \le n 
\le (1 - \frac{\varepsilon}{4})\lambda_D}\xi_n \le  (1 - \textstyle\frac{\varepsilon}{2})\lambda_D\big ] + \widetilde{\mathbb Q}\big[ |\Lambda_D^k| \le (1 - \textstyle\frac{\varepsilon}{4})\lambda_D \big] \leq  Ce^{-c \varepsilon^2 \lambda_D}.
\end{multline}
To conclude, one recalls that $f \ge \kappa$ pointwise on $B_{K + L}$ and that $t_k < L$ which together readily yield that $P_{t_k}f(x)= E_x[f(X_{t_k})] \geq 2^{-d} \kappa$ for any $x \in B_{K+L}$, whence
\begin{align*}
\lambda_D = \frac{1}{L}\sum_{x \in D} (P_{t_k}f)(x) \ge  L^{-1}|D| 2^{-d} \kappa \geq  c \kappa L^{d \alpha},
\end{align*}
for any $D \subset B_{K+L}$. This last bound also accounts for the necessity that $\widetilde{L} \gg L^{\frac1d}$ implicit in the choice of $\widetilde{L}$. Substituting the lower bound on $\lambda_D$ into \eqref{eq:indep_sum_decomp} and combining with \eqref{e:S-easy} completes the proof.
\end{proof}

\section{{\large Covering length-$L$ by length-$L'$ interlacements outside a good obstacle set}}\label{sec:hard_couple_obstacle}

In this section we prove Theorem~\ref{thm:long_short_obstacle}. The coarse architecture is similar to that of Section~\ref{sec:easyCOUPLINGS}. In \S\ref{subsec:overlaps}, we first deduce Theorem~\ref{thm:long_short_obstacle} from a more malleable version, Proposition~\ref{thm:long_short'}, which allows for overlaps between successive length-$L'$ trajectories; here, overlaps rather than gaps (cf.~\S\ref{ssec:tgap}) are relevant due to the opposition direction of inclusion in which the coupling operates, cf.~\eqref{eq:short_long-intro} and~\eqref{eq:long_short}. 
The coupling postulated by Proposition~\ref{thm:long_short'}, which barring technical simplifications due to the presence of overlaps, is similar to that of Theorem~\ref{thm:long_short_obstacle}, is built inductively over $k$ (much like in \S\ref{S:proofofeasycoupling}) but now by gluing shorter length-$L'$ trajectories progressively into longer ones, until reaching the desired length $L > L'$.

The proof of Proposition~\ref{thm:long_short'} is presented in full in \S\ref{sec:coup-outside-obs} assuming its one-step version, Lemma~\ref{L:obs-ind-step}, which encapsulates the properties needed to pass from $k$ to $(k+1)$, which roughly amounts to `attaching' a piece of length $L'$. Although reminiscent of Lemma~\ref{L:main-coupling-easier}, which plays a similar role in the context of Theorem~\ref{thm:short_long}, Lemma~\ref{L:obs-ind-step} and its proof, which appears in \S\ref{subsec:gluing}, are much more involved. Indeed, the underlying arguments entail the precise gluing mechanism for trajectories inside the (enlarged) good obstacle set $\widetilde{\mathcal O}$. In particular, the conditions appearing as part of Definition~\ref{def:O} crucially enter in order to witness this gluing occur with high enough probability, by exhibiting a high `surface density' of trajectories on each of the obstacles comprising $\mathcal{O}$. Informally speaking, the proof exploits a property of exchangeability between the trajectories near a given obstacle. We now flesh out these ideas.

\subsection{Overlaps and the timescale $t_o$} \label{subsec:overlaps}
As with 
Proposition~\ref{thm:short_long'}, following is a version of Theorem~\ref{thm:long_short_obstacle} 
with overlaps. Throughout this section, we are always (often implicitly) working under the assumptions of Theorem~\ref{thm:long_short_obstacle}; in particular, \eqref{e:couplings-params} is in force. We omit in the sequel any reference to the parameters $f,K,L, L', \gamma, \ell_{\mathcal{O}}, \delta_{\mathcal{O}}, M_{{\mathcal{O}}}$  as well as $u$ and $\varepsilon$, whose ranges are all specified by Theorem~\ref{thm:long_short_obstacle}.

\begin{prop}\label{thm:long_short'}
Under the assumptions of Theorem~\ref{thm:long_short_obstacle}, with 
\begin{equation}\label{def:gap}
L'' \stackrel{\textnormal{def.}}{=} L' - t_o , \quad \text{ where } \quad t_{o} \stackrel{\textnormal{def.}}{=} 3 \textstyle \lfloor  L {(\log L)^{-10 \gamma}}  \rfloor\, \,(\leq L' \text{ when } L \geq C; \,  see \, \eqref{e:couplings-params})
\end{equation}
there exists a coupling $\widetilde{\mathbb Q}$ of two $\{0, 1\}^{\Z^d}$-valued random variables 
$\mathcal I_1$ and $\mathcal I_2$ such that, 
\begin{equation}
\label{eq:long_short'}
\begin{split}
&\mathcal I_1 \stackrel{\textnormal{law}}{=}  \mathcal{I}^{f1_{B_K}, L},\: \mathcal I_2 
\stackrel{\textnormal{law}}{=} 
\mathcal{I}^{(1 + \varepsilon)f'', L'}; \text{ and, for $L \ge L_0(d, \gamma,u)$,}\\
&\widetilde{\mathbb Q} \big[ {\mathcal{I}}_1\setminus \widetilde{\mathcal O}  \subset  {\mathcal{I}}_2\setminus  \widetilde{\mathcal O}  \big] \geq  1 - 
C(u \vee 1) (K + L)^d\big( e^{-c(\varepsilon M_{\mathcal{O}} \wedge (\log L)^{\gamma})}  \vee  \delta_{\mathcal{O}} \big),
\end{split}
\end{equation}
where (recall that $l = \frac{L}{L'}$, see \eqref{eq:no-loop-l})
\begin{equation}\label{eq:f''_hard}
f'' = l^{-1}\sum_{0 \le k < \lceil \frac{L}{L''}\rceil} P_{kL''}
(f).
\end{equation}
\end{prop}
In analogy with \eqref{eq:f''_easy}, 
\eqref{eq:f''_hard} presents the advantage over $P_L^{L'}(f)$ 
of inducing overlaps (of time $t_o$) in the intensity profile of the (bigger) process $\mathcal{I}_2$ 
comprising length-$L'$ trajectories, which will facilitate the proof of \eqref{eq:long_short'}.\\

Assuming Proposition~\ref{thm:long_short'} to hold, we first give the proof of Theorem~\ref{thm:long_short_obstacle}. 
\begin{proof}[Proof of Theorem~\ref{thm:long_short_obstacle} (assuming Proposition~\ref{thm:long_short'})]
The proof is similar to that of Theorem~\ref{thm:short_long} using Proposition~\ref{thm:short_long'}. 
We highlight the main changes. While stating the inequalities, we will often imply that $L$ is 
large enough in a manner depending only on $\gamma$ and $d$, as allowed in the statement 
of Theorem~\ref{thm:long_short_obstacle}. By choice of $t_o$ in \eqref{def:gap} and assumption on 
$L'$ in \eqref{e:couplings-params}, one has
\begin{equation}\label{eq:gap_bnd_hard}
(L')^2 > 2 L t_o
\end{equation}
(cf.~\eqref{eq:gap_bnd1}). In view of $f''$ in \eqref{eq:f''_hard} and $P_L^{L'}(f)$ in \eqref{eq:f'}, 
one compares $P_{kL''}$ to $P_{kL'}$ using \eqref{eq:regularity} with the choice $\Delta=kt_o / kL'' = 
t_o / L''$ and $n = kL''$ for $1 \le k < \lceil \frac L{L''} \rceil$, for which the 
condition $\Delta \in (n^{-2/3}, 1)$  is met. Now, as in \eqref{eq:regularity_bnd} this yields, pointwise 
on $\Z^d$ and for any $1 \le k < \lceil \frac L{ L''} \rceil$,  
\begin{equation}\label{eq:regularity_bnd_hard}
\begin{split}
P_{kL''}(f)&\le \textstyle \big(1 + C\big({ \frac{t_o}{ L'}}\big)^{1/2}\big) P_{kL'}(f) + u_1'1_{B_{K+L}},
\end{split}
\end{equation}
where $u_1'=  u\, e^{-c (L' / t_o)^{1/2}}$ (recall that $f$ is supported on $B_{K+L}$). From \eqref{eq:regularity_bnd_hard}, one immediately obtains 
	a coupling of three random sets distributed as $\mathcal{I}^{(1 + Cl^{-1/2})\tilde f', L'}$, 
$\mathcal{I}^{u_1'1_{B_{K+ L}}, L'}$ and $\mathcal{I}^{(1 + Cl^{-1/2})f'', L'} $, with  
	$f''$ as in \eqref{eq:f''_hard} 
	and
	\begin{equation*}
\tilde f' \textstyle \stackrel{\textnormal{def.}}{=} l^{-1}\,\sum_{0 \leq k < \lceil{\frac {L}{L''}}\rceil}P_{kL'} (f),
	\end{equation*}
such that the former two are independent and their union contains the latter a.s. Concatenating this 
coupling with the one from \eqref{eq:long_short'} using Lemma~\ref{lem:concatenation} yields a 
coupling $\mathbb Q_1'$ of three random variables $\mathcal I_1 \stackrel{\text{law}}{=} \mathcal{I}^{f 1_{B_{K}}, L}$, $\tilde{\mathcal{I}}_2 \stackrel{\text{law}}{=} \mathcal{I}^{(1 +  Cl^{-1/2})\tilde f', 
	L'}$ and 
$\tilde{\mathcal{I}}_3 \stackrel{\text{law}}{=} \mathcal{I}^{u_1'1_{B_{K+L}}, L'} $ such that
\begin{equation}
\label{eq:easy_cover2_hard}
\mathbb Q_1' \big[ {\mathcal{I}}_1\setminus \widetilde{\mathcal O}  \subset  (\tilde {\mathcal I_2} \cup \tilde{\mathcal{I}}_3)\setminus  \widetilde{\mathcal O} \big] \geq  1 - C(u \vee 1) (K + L)^d\big( e^{-c(\varepsilon 
	M_{\mathcal{O}} \wedge (\log L)^{\gamma})}  \vee  \delta_{\mathcal{O}} \big)
\end{equation}
where $\tilde {\mathcal I_2}$ and $\tilde{\mathcal{I}}_3)$ are sampled independently. 
Now since  $L'' = L' - t_o$ and $L' > 2 t_o$ on account of \eqref{eq:gap_bnd_hard}, one finds that 
$\lceil \frac {L}{L''} \rceil - l \in (0,2)$ and hence that
\begin{align}
\label{eq:last_piece_hard}\textstyle
P_L^{L'}(f) = \tilde f' -l^{-1}\sum_{l \leq k < \lceil{\frac {L}{L''}}\rceil } P_{kL'} (f)= \tilde f' - l^{-1} P_{L} (f).
\end{align}
Then, proceeding as in \eqref{eq:extra_piece}--\eqref{eq:extra_piece2}, one applies \eqref{eq:regularity} again with the choice $\Delta = k L' / (L - kL')$ and $n = L - kL'$, for $1 \le k \le \sqrt{l}$ to replace $P_{L}$ by $P_{L-kL'}$ on the right-hand side of \eqref{eq:last_piece_hard} and  
	to deduce after averaging over $k$, 
\begin{align*}
P_L^{L'}(f)	
		\le (1 + Cl^{-\frac12})\tilde f' +  C u\, e^{-c l^{1/4}}1_{B_{K+L}}
\end{align*}
	Combining the natural coupling induced by the previous estimate with that of \eqref{eq:easy_cover2_hard} via chaining, 
	one obtains a coupling $\mathbb Q_2'$ of three random variables $\mathcal I_1 
	\stackrel{\text{law}}{=} \mathcal{I}^{f1_{B_K}, L}$, 
${\mathcal{I}}_2 \stackrel{\text{law}}{=} \mathcal{I}^{(1+ Cl^{-1/2})P_L^{L'}(f), L'}$ and 
$\mathcal{I}_3 \stackrel{\text{law}}{=} \mathcal{I}^{u_2' 1_{B_{K + L}}, L'}$ with $u_2' = C u e^{-c l^{1/4}}$, such that
\begin{equation*}
\mathbb Q_2' \big[ \mathcal I_1 \setminus \widetilde{\mathcal O}  \subset 
(\mathcal I_2 \cup \mathcal{I}_3) \setminus  \widetilde{\mathcal O}\big] \geq  
1 - C(u \vee 1) (K + L)^d\big( e^{-c(\varepsilon M_{\mathcal{O}} \wedge (\log L)^{\gamma})}  \vee  \delta_{\mathcal{O}} \big),
\end{equation*}
where 
$\mathcal I_2$ and $\mathcal{I}_3$ are sampled independently. Theorem~\ref{thm:long_short_obstacle} now 
readily follows with $\mathbb{Q}'= \mathbb Q_2'$ since $\mathbb Q[\mathcal{I}_3 = \emptyset] \ge 1 
- C(K + L)^d u_2'$ and $l = \tfrac{L}{L'} \le (\log L)^{\gamma}$ by 
\eqref{e:couplings-params}
.\end{proof}

\subsection{Coupling outside enlarged obstacle set $\tilde{\mathcal O}$}\label{sec:coup-outside-obs}
This subsection contains the proof of Proposition~\ref{thm:long_short'}, save for one result, Lemma~\ref{L:obs-ind-step} below, which represents one inductive step in a sequence of couplings to be constructed. A loose analogue of 
Lemma~\ref{L:obs-ind-step} in the context of the (simpler) Proposition~\ref{thm:short_long'} is 
Lemma~\ref{L:main-coupling-easier}, see~in particular 
item $(iii)$ therein. As opposed to the setup of Section~\ref{sec:easyCOUPLINGS}, the reverse inclusion, 
cf.~\eqref{eq:obs-goal} below, effectively requires recombining shorter trajectories into longer ones, 
for which the region $\widetilde{\mathcal{O}}$ acts as a buffer zone, in which the gluing occurs but 
control over the coupled interlacement sets is lost.

\begin{proof}[Proof of Proposition~\ref{thm:long_short'}]
Recall from Definition~\ref{def:O} (see below \eqref{eq:obs-dense}) that $U_{\mathcal O}=B_{K+\sqrt{L} (\log L)^{\gamma}}$. 
The measure $\widetilde {\mathbb Q}$ that will be constructed in the course of this proof will couple 
two random sets denoted by $\mathcal{I}^{\pm}$ with the properties that 
	\begin{equation}
		\label{eq:obs-goal}
		\begin{split}
			&\mathcal{I}^- \stackrel{\text{law}}{=}\mathcal{I}^{f1_{B_K}, L} \cap ( \Z^d \setminus \widetilde{\mathcal{O}} ), \quad \mathcal{I}^+ \stackrel{\text{law}}{=} \mathcal 
			J^{(1+\varepsilon) f'' 1_{U_{\mathcal{O}}}, L'}, \\
			&Q[\mathcal{I}^- \subset \mathcal{I}^+] \geq 1- C(u \vee 1)(K + L)^d  \big(e^{-c(\varepsilon M_{\mathcal{O}} \wedge (\log L)^{\gamma})}  \vee  \delta_{\mathcal{O}} \big),
		\end{split}
	\end{equation}
for $L \geq C(d,\gamma,u)$, which will tacitly be understood to hold from here onwards.
From \eqref{eq:obs-goal}, the claim of Proposition~\ref{thm:long_short'} immediately follows since 
$\mathcal{I}^{f'', L'}$ is readily obtained from $\mathcal{I}^+$ upon adding by extending 
$\widetilde {\mathbb Q}$, an independent process with law $\mathcal{I}^{f'' (1-1_{U_{\mathcal{O}}}), 
L'}$ so $\mathcal{I}^+ \subset \mathcal{I}^{f'',L'}$ $\widetilde{\mathbb Q}$-a.s.

The restriction to $U_{\mathcal{O}}$ 
in the definition of $\mathcal{I}^+$ in \eqref{eq:obs-goal} is convenient and motivated by the fact 
that any trajectory in the support of $\mathcal{I}^-$ (with starting point in $B_K$) will hit 
$\mathcal{O}$ before long, as implied by \eqref{eq:obs-visible}, and that the resulting sequence of 
`$k$-th return points' (in a sense defined precisely below, see~\eqref{eq:eq:coup101-new}) to the 
obstacle set $\mathcal{O}$ belonging to any box $B \in \mathcal{O} \cap U_{\mathcal{O}}$ will be 
both large and vastly outnumbered by the `first' hitting points produced by trajectories in the $k$-th 
`layer' of $\mathcal{I}^+$ (where $f''=\sum_k f_k''$, cf.~\eqref{eq:obs-f_k''}) with starting point in 
$U_{\mathcal{O}}$. Boxes lying outside $U_{\mathcal{O}}$ will typically be insignificant for 
$\mathcal{I}^-$, i.e.~the obstacle set will be hit long before exiting $U_{\mathcal{O}}$.

We now prepare the ground to produce the desired coupling. Recall $f''$ from \eqref{eq:f''_hard} and 
introduce, for $k \in \{ 0,\dots,k_L\}$ with $k_L = \lceil \frac L{L''} \rceil -1$ the function $f_k'':\Z^d\to 
\R_+$ given by
\begin{equation}\label{eq:obs-f_k''}
f_k''(x)=\textstyle\frac{L'}{L} (P_{kL''} f)(x),
\end{equation}
so that $f''=\sum_k f_k''$. Here and throughout the remainder of this proof, $k$ will always be tacitly 
assumed to range over all integers with $0\leq k \leq k_L$. Under an auxiliary probability measure 
$\widetilde{\mathbb Q}^+$, we introduce independent  Poisson processes $\omega_k^+$ (one for 
each $k$) on $\R_+ \times \Z^d  \times W^+$ each having intensity $\nu^+$ given by 
\begin{equation}
\label{eq:obs-nu-k-+}
\nu^+([0,u] \times \{x\} \times A)=u
P_x[X \in A].
\end{equation}
We now define certain quantities induced by the processes $\omega_k^+$. Given a stopping time 
$\tau$ for the random walk under $P_{\cdot}$ and a cemetery state $\Delta$, define the random 
variable
\begin{equation}
	\label{eq:obs-Y}
	Y: W^+ \to \Z^d \cup \{ \Delta\}, \quad w \mapsto Y(w)\stackrel{\text{def.}}{=} \begin{cases} w(\tau), & \text{if } \tau < \infty\\
		\Delta, & \text{else.}
	\end{cases}
\end{equation}
For suitable $X$, denote by $\Omega_X$ the space of $\sigma$-finite point measures on $X$ in the sequel, with its associated canonical $\sigma$-algebra. Recall $\mathcal{O} \subset \Z^d$, a finite set by Definition~\ref{def:O}, and let $g:\Z^d \to [0,\infty)$ have finite support. Attached to $Y$ in \eqref{eq:obs-Y} is the map $\Phi^{\tau}_ g:\Omega_{\R_+ \times \Z^d  \times W^+} \to \Omega_{\mathcal{O}}$ with
\begin{equation}
	\label{eq:obs-Phi}
	\Phi^{\tau}_ g(\omega) = \sum_i \delta_{Y(w_i)}1\big\{ \tau(w_i)< \infty, Y(w_i) \in \mathcal{O},  u_i \leq \textstyle\frac1{L'} g(x_i) \big\}, \quad \text{if }\omega =\sum_i \delta_{(u_i,x_i, w_i)}.
\end{equation}
The point measure $\Phi^{\tau}_ g(\omega)$ is increasing in both $\mathcal{O}$ and $g$ and under 
$\widetilde{\mathbb Q}^+$ and for any $k$,
\begin{equation}
	\label{eq:obs-Phi'}
	\text{$\Phi^{\tau}_ g(\omega_k^+)$ is a Poisson process on $\mathcal{O}$ of intensity $\textstyle\frac1{L'} P_g [  X_{\tau}  = \cdot,  X_{\tau} \in \mathcal{O}, \tau < \infty]$,}
\end{equation}
where, following our usual notation, $P_g =  \sum_x g(x)P_x$. Now, with $f_k''$ as in \eqref{eq:obs-f_k''}, define 
\begin{align}\label{eq:obs-h_k}
g_k^{+}: \Z^d \to \R_{+}, \quad  g_k^+= 
(1 + \varepsilon)f_k'' 1_{U_{\mathcal{O}}}, \ k \geq 0, 
\end{align}
let
\begin{equation}
	\label{eq:obs-tau}
	\tau_k = \begin{cases} L''+ H_{\mathcal{O}} \circ \theta_{L''}  & \text{if } k=0\\
		H_{\mathcal{O}} & \text{if } k\geq 1
	\end{cases}
\end{equation}
and 
\begin{equation}\label{eq:obs-phi_k} 
	\Phi_{k}^+= \Phi^{\tau_k}_{ g_k^+}(\omega_k^+) = \sum_{1 \leq i \leq N_k^+} \delta_{Y_{k,i}^+}, \quad k \geq 0,
\end{equation}
where, with a slight abuse of notation $Y_{k,i}^+= Y_{k,i}^+(\omega_k^+)$, $1\leq k \leq N_k^+$, refer to the random variables $Y(w_i)$ appearing in \eqref{eq:obs-Phi} (with the choices $\tau=\tau_k$ and $g=g_k^+$), ordered for sake of definiteness according to increasing label $u_i$. In words, $\Phi_{k}^+$ is the Poisson process obtained by taking each point $(w_i, x_i, u_i)$ in the support of $\omega_k^+$ with appropriate label $u_i$, and retaining for every trajectory $w_i$ for which $X_{\tau_k (w_i)}$ belongs to $\mathcal{O}$ the corresponding hitting point at time $\tau_k$. The discrepancy in \eqref{eq:obs-tau} between $k=0$ and $k \geq 1$ is due to the special role played by $k=0$ below. For later reference, note that due to \eqref{eq:obs-Phi'},
\begin{equation}
	\label{eq:obs_N_k^+}
	\text{$N_k^+$ is a Poisson variable with parameter $\textstyle\frac1{L'} P_{g_k^+} [\tau_k < \infty]$.}
\end{equation}

Let us briefly pause to give an overview over the remainder of the proof. We next define a sequence 
of couplings $Q_k$ inductively in $k$. For each value of $k$ with $0 \leq k \leq k_L$, the coupling 
$Q_k$ extends $Q_{k-1}$ (in case $k \geq 1$) and gives rise to two random sets 
$\mathcal{I}^{\pm}_k$ sharing similar properties as $\mathcal{I}^{\pm}$ in \eqref{eq:obs-goal}. The 
sets $\mathcal{I}^{\pm}$ will be generated at the terminal value $k=k_L$, i.e.~after $k_L+1 (=\lceil 
\frac L{L''} \rceil )$ iterations. Roughly speaking, $\mathcal{I}^{+}_k$ will correspond to the 
contribution to $\mathcal{I}^+$ stemming from $g_n^+$, $0\leq n \leq k$, cf.~\eqref{eq:obs-h_k} 
and \eqref{eq:obs-f_k''}, whereas the trajectories comprising $\mathcal{I}^{-}_k$ will correspond to 
the first $k+1$ pieces of length $L''$ destined to form the length-$L$ walks constituting 
$\mathcal{I}^-$, which will be reconstructed in the process. The reconstruction is facilitated by the 
presence of the enlarged obstacle set $\widetilde{\mathcal{O}}$, see \eqref{eq:obs-tilde}, inside of 
which the walks need not being tracked and which will serve as a `gluing zone'. The time loss needed 
to perform the gluing inside $\widetilde{\mathcal{O}}$ will turn out to be negligible thanks to the 
conditions listed in Definition~\ref{def:O}. 

\medskip
We now proceed to make the above strategy precise. 

\begin{lemma}[$0\leq k \leq k_L$] \label{L:obs-ind-step} There exists a probability measure $Q_k$ governing $ \mathcal{I}_k^{\pm}$ with
	\begin{align}
		&\mathcal{I}_k^{-}  \stackrel{\textnormal{law}}{=}  \mathcal{I}^{f1_{B_K}, \, (k+1) L''} , \label{eq:obs-indstep0}\\
		& \mathcal{I}_k^{+}  \stackrel{\textnormal{law}}{=}  \mathcal{I}^{ \sum_{0\leq \ell \leq k } g_k^+, L'}, \label{eq:obs-indstep1}\\
		&Q_k \big[ \big(\mathcal{I}_k^{-} \cap ( \Z^d \setminus \widetilde{\mathcal{O}} ) \big) \subset  
		\mathcal{I}_k^{+} \big] \geq  1 - C (u \vee 1) (K + L)^d \big( e^{-c(\varepsilon M_{\mathcal{O}} \wedge (\log L)^{\gamma})}  \vee   \delta_{\mathcal{O}}  \big) \label{eq:obs-indstep2}.
	\end{align}
\end{lemma}

Assuming Lemma~\ref{L:obs-ind-step} to hold for the time being, we first conclude the proof of 
Proposition~\ref{thm:long_short'}, which entails verifying that \eqref{eq:obs-goal} holds. Choosing $Q 
= Q_{k_L}$, it follows on account of \eqref{eq:obs-indstep1} and \eqref{eq:obs-f_k''} that 
$\mathcal{I}^+ \stackrel{\text{def.}}{=}\mathcal{I}_{k_L}^+$ has the law prescribed in 
\eqref{eq:obs-goal}. Moreover, in view of \eqref{eq:obs-indstep0}, as the length of the trajectories in 
the support of $\mathcal{I}_{k_L}^-$ is $(k_L+1)L'' \geq L$, one straightforwardly extends $Q $ by 
suitable thinning 
to a coupling carrying a random set $\mathcal{I}^-$ with the same law as $\mathcal{I}^{f1_{B_K}, L}$ 
in such a way that the inclusion $\mathcal{I}^- \subset \mathcal{I}_{k_L}^-$ holds almost surely. 
Then, the inclusion with high probability asserted in \eqref{eq:obs-goal} is a direct consequence of 
\eqref{eq:obs-indstep2}. Thus, \eqref{eq:obs-goal} holds, which completes the proof of 
Proposition~\ref{thm:long_short'} (conditionally on Lemma~\ref{L:obs-ind-step}). \end{proof}

\subsection{Gluing trajectories}\label{subsec:gluing}

We now give the proof of Lemma~\ref{L:obs-ind-step}.

\begin{proof}[Proof of Lemma~\ref{L:obs-ind-step}]
For each $0 \leq k \leq k_L$ define 
\begin{equation} \label{e:U_k-hard}U_k = B_{K+ \frac12 \sqrt{(k+1) L''} (\log L)^{\gamma}} \ (\subset U_{\mathcal{O}}),
\end{equation}
where the latter inclusion follows on account of the definitions $U_{{\mathcal O}} = B_{K + \sqrt{L} (\log L)^{\gamma}}$, $k_L = \lceil \frac L{L''} \rceil -1 $ 
as well as the displays \eqref{e:couplings-params} and \eqref{def:gap}. We proceed inductively in $k$ 
and show the following stronger statement, which is convenient to 
carry out the inductive step: for every $k$, there exists a coupling $Q_k$, extension of $Q_{k-1}$ for 
all $k \geq 1$, satisfying \eqref{eq:obs-indstep0}-\eqref{eq:obs-indstep1} and carrying one further 
random set $\mathcal{K}_k$ and its associated endpoints $Y_{k}^-= \{ Y_{k,i}^{-} : 1 \leq i \leq 
N_0^-\}$ with the property that, for some event $G_k$ decreasing in $k$ with
\begin{equation}
	\label{eq:obs-G_k-bound}
	\begin{split}
		&Q_0[G_0] \geq 1- u |U_{\mathcal{O}}|( \delta_{\mathcal{O}} \vee e^{-c(\log L)^{\gamma}} ), \\
		&Q_k[G_{k-1} \setminus G_{k}] \leq (u \vee 1) |U_{\mathcal{O}}| \big(  \delta_{\mathcal{O}} \vee  e^{-c(\varepsilon M_{\mathcal{O}} \wedge (\log L)^{\gamma})}  \big),\,  1\leq k \leq k_L,
	\end{split}
\end{equation}
one has the inclusions
\begin{equation}
	\label{eq:coup100}
	\big(\mathcal{I}_k^{-} \cap ( \Z^d \setminus \widetilde{\mathcal{O}} ) \big) \subset \big(\mathcal{K}_k \cap ( \Z^d \setminus \widetilde{\mathcal{O}} ) \big) \subset \mathcal{I}_k^{+} \text{ and } Y_{k}^- \subset U_{k}: \quad \text{on } G_k.
\end{equation}
The set $\mathcal{K}_k$ represents the precise proportion of the set $\mathcal{I}^-$ in \eqref{eq:obs-goal} that is covered by the coupling until step $k$, that is, by using the first $k+1$ pieces of length-$L'$ trajectories only. Intuitively, it comprises the length-$L$ trajectories that will eventually form $\mathcal{I}^-$, run until a time which is slightly larger than the target time $(k+1) L''$ needed for $\mathcal{I}_k^{-}$. 
Specifically, the set $\mathcal{K}_k$ has the following prescribed law. Consider the sequence of `return' times (see \eqref{eq:obs-tau} for notation)
\begin{equation}\label{eq:eq:coup101-new}
		R_0= \tau_0, \quad R_k = (k+1)L'' + \tau_k \circ \theta_{(k+1)L''}, \, k \geq 1.
	\end{equation}
	Note that the sequence of times $R_k$ is a-priori unordered, but, as will turn out, the event $G_k$ (cf.~\eqref{eq:obs-G_k-bound}) to be constructed in the course of the proof will guarantee that the map $\ell \in \{0,\dots k\} \mapsto R_{\ell}$, is in fact increasing when $G_k$ occurs. With $R_k$ given by \eqref{eq:eq:coup101-new}, the distribution of $\mathcal{K}_k=\mathcal{K}_k(\omega_0^+)$ is specified as
	\begin{equation}
		\label{eq:coup102}
		\mathcal{K}_k \stackrel{\text{law}}{=}   \bigcup_{1\leq i \leq N_0^{-}} w_i[0,R_k(w_i)] \text{ (under $Q^+$),}
	\end{equation}
with associated endpoints $Y_{k}^-$ defined as $\{w_i(R_k(w_i)): 1\leq i \leq N_0^- \}$ (note that 
these tacitly require $\{R_k(w_i)< \infty\}$ to occur) and where, recalling that $\omega_0^+=  \sum_{i} \delta_{(u_i,x_i,w_i)}$ has intensity $\nu^+$ given by \eqref{eq:obs-nu-k-+}, we have set
\begin{align}\label{eq:obs-N_0^-}
&N_0^- = \omega_0^+ \big(\{ u \leq \textstyle \frac{1}{L'}g_0^-(X_0) , \tau_0 < \infty \}\big),
\end{align}
	a Poisson variable with intensity $\textstyle\frac1{L'} P_{g_0^-} [\tau_0 < \infty]$ and 
	\begin{equation}
		\label{eq:obs-g_0^-} 
		g_0^- =  f_0'' 1_{B_{K}}, \text{ so that } \textstyle \frac{1}{L'}g_0^-= \frac{1}{L}f1_{B_K}.
	\end{equation}
	In \eqref{eq:coup102}, $w_i$ refers to the (ordered) trajectory in the support of $\omega_0^+$ containing the point $Y_{0,i}^+$ for every $i \leq N_0^+$ and \eqref{eq:coup102} is well-defined because $N_0^- \leq N_0^+$ holds $Q^+$-a.s.~on account of  \eqref{eq:obs_N_k^+} and \eqref{eq:obs-h_k}, \eqref{eq:obs-g_0^-}, which together imply that $g_0^- \leq g_0^+$. 
	
	We now proceed to show the existence of $Q_k$ with the above properties, i.e., satisfying \eqref{eq:obs-G_k-bound} and \eqref{eq:coup100} with $\mathcal{K}_k, Y_k^-$ as in \eqref{eq:coup102}.
	
	\medskip
	\noindent\textbf{The case $k=0$.} We simply define $Q_0 = Q^+$ (see above \eqref{eq:obs-nu-k-+}) and set, with $\omega_0^+$ having intensity $\nu^+$ given by \eqref{eq:obs-nu-k-+},
	\begin{equation}\label{eq:obs-coup0}
		{\mathcal{K}}_0= {\mathcal{K}}_0 (\omega_0^+)= \bigcup_{1\leq i \leq N_0^{-}} w_i[0,\tau_0(w_i)]
	\end{equation}
	with $N_0^-$ as in \eqref{eq:obs-N_0^-} and where $w_i$ refers to the trajectory in the support of $\omega_0^+$ containing the point $Y_{0,i}^+$ for every $i (\leq N_0^+)$, as above.  We now introduce the `good' event (under $Q_0$)
	\begin{equation}
		\begin{split}
			\label{eq:obs-coup2}
			G_0 = &\bigg\{ \begin{array}{c} \omega_0^+\big(\{  u \leq \textstyle \frac{1}{L'}g_0^+(X_0) , \tau_0 - L'' > \frac{t_o}{3} \}\big) =0,\\[0.3em]
				\omega_0^+\big(\{  u \leq \textstyle \frac{1}{L'}g_0^-(X_0) ,\, X_{\tau_0} \notin U_0, \, \tau_0 < \infty \}\big) =0\end{array}\bigg\}
		\end{split}
	\end{equation}
	(observe that the first line includes the possibility that $\tau_0 = \infty$). In view of \eqref{eq:obs-tau}, when $G_0$ occurs, any trajectory $w_i$ involved in the construction of $\mathcal{K}_0$ in \eqref{eq:obs-coup0} satisfies $(L'' \leq) \tau_0(w_i) \leq L''+\frac{t_o}{3}  < L' (=L'' + t_o)$. Hence, as we now explain, on the event $G_0$, one obtains the chain of inclusions (recall \eqref{eq:J} for notation)
	\begin{equation}
		\label{eq:obs-coup3}
		\mathcal{I}^{f1_{B_K}, L''}  \stackrel{\text{law}}{=} \mathcal{I}^{g_0^-, L''}(\omega_0^+) \subset {\mathcal{K}}_0 (\omega_0^+) \subset 
		\mathcal{I}^{g_0^+, L'} (\omega_0^+)\stackrel{\text{law}}{=} \mathcal 
			I^{(1 + \varepsilon) f_0'' 1_{U_{\mathcal{O}}}, L'};
	\end{equation}
	here, the first inclusion follows from the fact that all trajectories $w$ in the support of $\mathcal{I}^{g_0^-, L''} $ have $\tau(w)< \infty$ on $G_0$, hence by definition of ${\mathcal{K}}_0$ in \eqref{eq:obs-coup0}, they also appear in its support and run for a longer time. The second inclusion \eqref{eq:obs-coup3} is due to the combined facts that $N_0^- \leq N_0^+$ and the upper bound $\tau_0 < L'$ noted above, which is valid on the event $G_0$ for all the relevant trajectories. All in all, setting $ \mathcal{I}_0^{\pm} = \mathcal{I}^{g_0^{\pm}, L''}(\omega_0^+) $, it follows on account of \eqref{eq:obs-coup3} that the inclusion \eqref{eq:coup100} holds (in fact the intersection with $(\Z^d \setminus \widetilde{\mathcal{O}})$ can be omitted here) and that the sets $ \mathcal{I}_0^{\pm}$ and $\mathcal{K}_0$ have the required laws prescribed by \eqref{eq:obs-indstep0}, \eqref{eq:obs-indstep1} and \eqref{eq:coup102}, respectively (in the latter case, this follows plainly from \eqref{eq:obs-coup0} and the definition of $R_0$ in \eqref{eq:eq:coup101-new}).
	
	Finally, it remains to argue that the event $G_0$ given by \eqref{eq:obs-coup2} satisfies the required bound in \eqref{eq:obs-G_k-bound}. This is due to the fact that the random variable entering the first line in the definition of $G_0$ is Poisson with parameter
	$$
	\nu^+( u \leq \textstyle \frac{g_0^+(X_0)}{L'} , \tau_0 - L'' > \frac{t_o}{3}) \stackrel{\eqref{eq:obs-nu-k-+}}{=} \frac1{L'}\sum_{x} g_0^+(x)P_{x}[P_{X_{L''}}[H_{\mathcal{O}} > \frac{t_o}{3}]] \leq  (\delta_{\mathcal{O}} \vee e^{-cL^c}) \sum_{x} g_0^+(x),
	$$
	where the last step follows by means of the visibility condition \eqref{eq:obs-visible} upon observing that $\text{supp}(g_0^+) \subset U_{\mathcal{O}}$, whence $X_{L''} \in B_{K+L}$ holds with $P_x$-probability $1-e^{-cL^c}$~for $x \in \text{supp}(g_0^+)$. A similar bound can be derived for relevant intensity of the event appearing in the second line of \eqref{eq:obs-coup2}. From this \eqref{eq:obs-G_k-bound} follows using the elementary inequality $1-e^{-x} \leq x$ valid for $x\geq 0$. Upon setting
	\begin{align*}
		&Y_{0,i}^{-}=Y_{0,i}^{+} (=w_i(\tau_0)), \ 1 \leq i \leq N_0^-,
	\end{align*}
	by which $Y_{0,i}^{-}$, $1 \leq i \leq N_0^-$, correspond on the event $G_0$ to the endpoints of all the trajectories involved in the construction of $\mathcal{K}^0$, cf.~\eqref{eq:obs-coup0}, this concludes the proof in the case $k=0$. \\

	\noindent\textbf{The case $(k-1) \to k$.} We now inductively assume $Q_{k-1}$, extension of $Q_0=Q^+$ for some $1 \leq k  \leq k_L$, to be a coupling of $\mathcal{I}_{k-1}^{\pm}$, $\mathcal{K}_{k-1}$ with the desired laws, such that \eqref{eq:coup100} holds on a suitable event $G_{k-1}$. We further suppose the random variables $Y_{\ell}^- = (Y_{\ell,i}^{-}: \, 1\leq i \leq N_0^-)$, for $1 \leq \ell \leq k-1$ to be defined under $Q_{k-1}$, corresponding to the endpoints of all trajectories in the support of $\mathcal{K}_{\ell}$ on the event $G_{\ell} (\supset G_{k-1})$. In law, these are precisely the points $R_\ell(w_i)$ appearing in \eqref{eq:coup102}. By construction, see \eqref{eq:eq:coup101-new} and \eqref{eq:obs-tau}, these random variables have values in $\mathcal{O}$.

	Recall the random variables $Y_{\ell}^+= (Y_{\ell,i}^+: 1\leq i \leq N_{\ell}^+)$ from \eqref{eq:obs-phi_k}, which are measurable functions of $\omega_{\ell}^+$ for each $\ell$. Our aim is to merge the trajectories stemming from $\mathcal{K}_{k-1}$ (their endpoints are $Y_{k-1,i}^{-}$, $1\leq i \leq N_0^+$) with trajectories from the support of $\omega_k^+$ (entering $\mathcal{O}$ through the points $Y_{k,i}^{+}$, $1 \leq i \leq N_k^+$) after letting them mix, which will take (homogenization) time 
	\begin{equation}\label{eq:obs-t-h-new}
		t_{h} \stackrel{\text{def.}}{=} H_{\partial{\widetilde{\mathcal O}}} 
	\end{equation}
	and take place inside $\widetilde{\mathcal{O}}$ ($ \supset \mathcal{O}$, cf.~\eqref{eq:obs-tilde}). The following lemma is key to this.

	The variables $Z^{\pm}$ below represent the locations of the walk after this homogenisation has taken place. In writing $\{Z^{-}  \subset  Z^{+}\}$, see e.g.~\eqref{eq:obs-Z-inclusion}, we include multiplicities, i.e.~we view $Z^{\pm}$ as multi-sets.
	Let $\mathcal{F}_{k}$ denote the $\sigma$-algebra generated by $\omega_{0}^+$, $\mathcal{I}_{\ell}^{\pm}$, $\mathcal{K}_{\ell}$ and $Y_{\ell}^-$ for $0 \leq \ell < k$ and 
	$\omega_\ell^+$ for $0 \leq \ell \leq k$ (under $Q_{k-1}$). 
	
\begin{lemma}[$1 \leq k \leq k_L$]\label{L:obs-condcoup} 
Conditionally on $\mathcal{F}_{k}$, there exists $q=q_{\mathcal{F}_{k}}$ coupling of
\begin{align*}
Z^- = \{ Z_{i}^{-} : 1 \leq i \leq N_0^+  \}, \quad  Z^+= \{ Z_{i}^{+} : 1 \leq i \leq N_{k}^+  \}
\end{align*}
		with the property that, abbreviating $\mu_z[\, \cdot \,] = P_{z}[X_{t_{h}} =\cdot \, ]$, $z \in \mathcal{O}$, 
		\begin{align}
			& Z_{i}^{-} \stackrel{\textnormal{law}}{=} \mu_{Y_{k-1,i}^{-}} , \
			Z_{i}^{+}   \stackrel{\textnormal{law}}{=}  \mu_{Y_{k,i}^{+}} \label{eq:obs-Z-law} \\ 
			&Q_{k-1}\big[q_{\mathcal{F}_k}[ Z^{-}  \subset  Z^{+}] 1_{G_{k-1}}\big] \geq 1- e^{-c\varepsilon^2 M_{\mathcal{O}}} , 
				\label{eq:obs-Z-inclusion}
			\end{align}
			whenever 
			$C \ell_{\mathcal O}^{-1/100}\leq \varepsilon$ (recall from the discussion before Definition~\ref{def:O} that $\ell_{{\mathcal O}}$ is radius of any box comprising the obstacle 
			set $\mathcal O$) and~\eqref{eq:obs-dense} holds.
		\end{lemma}
		
		We defer the proof of Lemma~\ref{L:obs-condcoup} and first complete the induction step.
We start by defining $Q_k$. By suitable extension, we assume that $Q_{k-1}$ further carries two 
independent families $ \{ \beta_{i}^{x,y}: x, y \in \Z^d, i \geq 1\}$ and $\{ X_{ i}^x: x \in  \Z^d, i \geq 1\}$ 
of i.i.d.~random variables, independent of the all the remaining randomness governed by $Q_{k-1}$, 
such that $ \beta_{i}^{x,y}$ has the law of the bridge $X_{\cdot \wedge t_h}$ under $P_x[\, \cdot \,| X_{t_{h}}=y]$ and $X_{i}^x$ has the same law as $(X_t)_{t \geq 0}$ under $P_x$ for every $i \geq 1$. 
For a bridge $\beta$ we write $t_h(\beta)=H_{\partial \widetilde{\mathcal{O}}}(\beta)$ for its 
(time-)length. Let
\begin{equation*}
Q_k[\, \cdot\,] \stackrel{\text{def.}}{=} Q_{k-1}[q_{\mathcal{F}_{k}}[\, \cdot\,]].
\end{equation*}
In order to construct $(\mathcal{I}_{k}^{-},\mathcal{K}_{k}, \mathcal{I}_{k}^{+})$ out of $(\mathcal{I}_{k-1}^{-},\mathcal{K}_{k-1}, \mathcal{I}_{k-1}^{+})$, we will use 
\begin{equation}\label{eq:obs-cond300}
\beta_i^{-} \stackrel{\text{def.}}{=}\beta^{Y_{k-1, i}^{-}, \,Z_i^{-}}_i, \quad \beta_i^{+} \stackrel{\text{def.}}{=}\beta^{Y_{k,i}^{+}, \,Z_i^{+}}_i, \quad X_i^{\pm} \stackrel{\text{def.}}{=} X^{Z_i^{\pm}}_i \text{ for }  1 \leq i \leq N_k^{\pm}
\end{equation}
		where for convenience we have set $N_k^-=N_0^-$. Now write $\underline{\omega}_k^+ $ for the restriction of the projection of $\omega_k^+ $ onto its first and third marginal to points $(u, w)$ satisfying 
		$u_i \leq \textstyle\frac1{L'} g_k^+(x)$, $w_i(0)=x$ and decompose, for $k \geq 1$,
		\begin{equation}
			\label{eq:obs_omega-k-decomp}
			\underline{\omega}_k^+ = \omega_1 + \omega_2 , \quad \omega_1= \underline{\omega}_k^+ 1\{ (u,w) :  \tau_k(w)< \infty 
			\} 
		\end{equation}
		(observe that the number of points in the support of $\omega_1$ is exactly $N_k^+$ given by \eqref{eq:obs_N_k^+})
		and consider for $1\leq i \leq N_k^+$ the concatenated process 
		\begin{equation}\label{eq:obs_W^+}
			W_{k,i}^+(t) = \begin{cases}
				w_i(t), &  0 \leq t \leq \tau_k, \\
				\beta_i^+(t- \tau_k), &\tau_k < t < \tau_k +t_h,  \\
				Z_i^+, &  t= \tau_k +t_h,\\
				X_i^+(t - \tau_k -t_h),&  t > \tau_k +t_h, 
			\end{cases}
		\end{equation}
		where $t_h= t_h(\beta_i^+)$, the point measure $\sum_{1\leq i \leq N_k^+}\delta_{(w_i, u_i)}$ refers to an (ordered) realization of the projection of $\omega_1$ and $\tau_k=\tau_k(w_i)=H_{\mathcal{O}}(w_i)$, cf.~\eqref{eq:obs-tau} (recall that $k \geq 1$). As we now explain, conditionally on $\mathcal{F}_k$ as defined above Lemma~\ref{L:obs-condcoup}, $(W_{k,i}^+(t))_{t \geq 0}$
		has the law of $P_{w_i(0)}$ for all $1\leq i \leq N_k^+$; 
		indeed, the intensity measure \eqref{eq:obs-nu-k-+} of $\omega_k^+$ and \eqref{eq:obs_omega-k-decomp} imply that (conditionally on $N_k^+$), $w_i(t)$, $0 \leq t \leq \tau_k$, follows the law of a random walk path until time $\tau_k= H_{\mathcal{O}}$, at which time the walk is at position $w_i(\tau_k)=Y_{k,i}^{+} = \beta^+_i(0) $, cf.~\eqref{eq:obs-phi_k} and \eqref{eq:obs-cond300}. Using \eqref{eq:obs-Z-law} and applying the strong Markov property at times $\tau_k$ and $\tau_k+ t_h$, the claim then readily follows.
		
		In a similar vein, for $1\leq i \leq N_k^-(=N_0^-)$, define 
		\begin{equation}\label{eq:obs_W^-}
			W_{k,i}(t) = \begin{cases}
				\beta_i^-(t), &0 \leq t < t_h,  \\
				Z_i^-, &  t= t_h,\\
				X_i^-(t - t_h),&  t > t_h
			\end{cases}
		\end{equation}
		with $t_h=t_h(\beta_i^-)$. In view of \eqref{eq:obs-Z-law} and \eqref{eq:obs-cond300}, it follows that under $Q_k$ and conditionally on $\mathcal{F}_k$, the process $(W_{k,i}(t) )_{t \geq 0}$ has law $P_{Y_{k-1,i}^{-}}$ for all $1\leq i \leq N_k^-$.
		
		With $\{ W_{k,i}^+ : 1\leq i \leq N_k^+\}$ and $\{ W_{k,i} : 1\leq i \leq N_k^-\}$ given by \eqref{eq:obs_W^+} and \eqref{eq:obs_W^-}, respectively, we now specify the sets $(\mathcal{I}_{k}^{-},\mathcal{K}_{k}, \mathcal{I}_{k}^{+})$ and proceed to verify that they have the desired marginal laws.
		Recalling $N_k^- =N_0^-$ from \eqref{eq:obs-N_0^-} and abbreviating $W[0,t]= \{ W(s) : 0 \leq s \leq t\}$, let
		\begin{align}
			& \mathcal{K}_k = \mathcal{K}_{k-1} \cup \widetilde{\mathcal{K}}, \quad \ \widetilde{\mathcal{K}}= \bigcup_{1\leq i \leq N_0^-} W_{k,i}[0, t_h + L''+ H_{\mathcal{O}} \circ \theta_{t_h + L''}]
			\label{eq:coup301} \\
			& \mathcal{I}_k^+ = \mathcal{I}_{k-1}^+ \cup \widetilde{\mathcal{I}}, \quad    \widetilde{\mathcal{I}}= \Big(\bigcup_{1\leq i \leq N_k^+} 
			W_{k,i}^+[0, L'] \Big)\cup \mathcal{I}^{g_k^+, L'}(\omega_2)\label{eq:coup303} 
		\end{align}
		(see \eqref{eq:obs_omega-k-decomp} regarding $\omega_2$). \\[0.5cm]
		Then define
		\begin{equation}\label{eq:coup304}
			Y_{k,i}^- =  W_{k,i} (t_h + L''+ H_{\mathcal{O}} \circ \theta_{t_h + L''}), \quad 1 \leq i \leq N_0^-,
		\end{equation}
		which, in view of \eqref{eq:coup301} correspond to the endpoints of the trajectories in $\mathcal{K}_k$.
		
		Before specifying $\mathcal{I}_k^-$, we first argue that the sets $\mathcal{K}_{k}, \mathcal{I}_{k}^{+}$ obtained in this way have the required marginal laws prescribed by \eqref{eq:obs-indstep1} and \eqref{eq:coup102}. In the case of $\mathcal{I}_k^+$, this is a direct consequence of the fact that, conditionally on $\mathcal{F}_k$ as defined above Lemma~\ref{L:obs-condcoup}, and writing $\omega_1= \sum_{1\leq i \leq N_k^+}\delta_{(w_i,x_i,u_i)}$ for a realization of $\omega_k^+$, the process 
		$
		W_i^+(t) $
		has law $P_{w_i(0)}$ for all $1\leq i \leq N_k^+$.  Indeed, since $\mathcal{I}_{k-1}^{+}$ is $\mathcal{F}_k$-measurable, this implies that, conditionally on $\mathcal{I}_{k-1}^+$, the set $\bigcup_{1\leq i \leq N_k^+} W_i^+[0, L']$ has the same law as $\mathcal{I}^{g_k^+, L'}(\omega_1)$ and is independent of $\mathcal{I}^{g_k^+, L'}(\omega_2)$, and thus that $ \widetilde{\mathcal{I}}$ in \eqref{eq:coup303} has the same law as $\mathcal{I}^{g_k^+, L'}(\omega_k^+)$ in view of~\eqref{eq:obs_omega-k-decomp}. Thus,~ \eqref{eq:coup102} holds for $\mathcal{I}_k^+$ defined by \eqref{eq:obs-indstep1}.
		
		The case of $\mathcal{K}_k $ is simpler. As $(W_i(t) )_{t \geq 0}$ has law $P_{Y_{k-1,i}^{-}}$ conditionally on $\mathcal{F}_k$, it readily follows that $\mathcal{K}_k$ given by \eqref{eq:coup301} has the same law as \eqref{eq:coup102}: the set $W_i[0, H_{\mathcal{O}} \circ \theta_{t_h + L''}]$ contributing to $\widetilde{\mathcal{K}}$ plays the same role (in law) as the increment  $ w_i[R_{k-1}(w_i),R_k(w_i)])(\omega_0^+)$, part of \eqref{eq:coup102}, as can be seen immediately from 
		\eqref{eq:coup301} upon recalling $R_k$ from \eqref{eq:eq:coup101-new} and observing that $W_i(0)= Y_{k-1,i}^{-}$ and the points $Y_{k-1,i}^{-}$, $1\leq i \leq N_0^-$ represent the endpoints of trajectories comprising $\mathcal{K}_{k-1}$ by induction assumption. All in all, this shows that $\mathcal{K}_k$ has the same law under $Q_k$ as the set in \eqref{eq:coup102}.
		
		It remains to define $\mathcal{I}_k^-$ and an event $G_k$ of suitably large probability on which \eqref{eq:coup100} holds. We first define $\mathcal{I}_k^-$. For $1\leq i \leq N_0^-$, recall  $W_{k,i}$ from \eqref{eq:obs_W^-}. By induction assumption we have analogues $W_{\ell,i}$, $1 \leq \ell < k$ entering the definition of $\mathcal{K}_{k-1}$. Similarly as in \eqref{eq:obs_omega-k-decomp} let $$ \underline{\omega}_{0}^+= (\pi\circ \omega_0^+)1\big\{ (w,u):  \, u \leq \textstyle \frac1{L'} g_0^-(w(0)) \big\} $$
		where $\pi$ denotes the projection onto the first and third marginal, and decompose
		\begin{equation}
			\label{eq:obs_omega-0-decomp}
			\underline{\omega}_{0}^+ = \mu_1 + \mu_2 , \qquad \mu_1= \underline{\omega}_{0}^+ 1{\{ (w,u):\, \tau_0(w)< \infty \}}
		\end{equation}
		Thus, cf.~\eqref{eq:obs-N_0^-}, the point measure $\mu_1$ has exactly $N_0^-$ elements in its support. Writing $\mu_1= \sum_{1\leq i \leq N_0^-} \delta_{(w_i, u_i)}$ for a generic realization, and with
		$$
		t_0= \tau_0(w_i), \quad t_{\ell} = t_h + L'' + (H_{\mathcal{O}} \circ \theta_{t_h + L''} )(W_{\ell,i}) , \quad 1 \leq \ell < k
		$$
		it follows with $s_{\ell}= \sum_{0 \leq n < \ell} t_n$ that for all $1\leq i \leq N_0^-$, the process $(W_{k,i}^-(t))_{t \geq 0}$ defined as
		\begin{equation}
			\label{eq:obs_W^--}
			W_{k,i}^-(t) = \begin{cases}
				w_i(t), &0 \leq t < s_1,  \\
				W_{\ell,i}(t- s_{\ell}), & s_{\ell} \leq t < s_{\ell+1}, \, 1 \leq \ell < k\\
				W_{k,i}(t- s_{k})&  s_k \leq t 
			\end{cases}
		\end{equation}
		has the same law conditionally on $N_0^-$ as $(X_t)_{t \geq 0}$ under $P_{\bar g_0^-} [ \, \cdot \, , \tau_0 < \infty]$, where $\bar g_0^-(\cdot)=  g_0^-(\cdot)/ P_{ g_0^-} [ \tau_0 < \infty]$, with $\tau_0$ as in \eqref{eq:obs-tau}. Hence, the set
		\begin{equation}
			\label{eq:coup321}
			\mathcal{I}_k^- \stackrel{\text{def.}}{=} \Big(\bigcup_{1 \leq i \leq N_0^+} W_{k,i}^-[0, (k+1)L'']\Big) \cup \mathcal{I}^{g_0^-, (k+1)L''}(\mu_2)
		\end{equation}
		has the law prescribed by \eqref{eq:obs-indstep0} in view of \eqref{eq:obs-g_0^-} and \eqref{eq:obs_omega-0-decomp}. Finally, let
		\begin{equation}\label{eq:obs-coupG_k}
			\begin{split}
				&G_k= G_{k-1} \cap G_k^- \cap G_k^+ \cap \{  Y_{k}^- \subset U_{k} \}, 
			\end{split}
		\end{equation}
		where
		\begin{align*}
			&G_{k}^-= \{ \mu_2(\Omega_{\R_+ \times W^+}) = 0 \},\\[0.5em]
			& G_{k}^+=
			\left\{ \begin{array}{c} \text{ $Z^{-}  \subset  Z^{+}$,  for all $1 \leq i \leq N_k^{-}$: } 
				(H_{\mathcal{O}} \circ \theta_{ L''} )(X_i^-) \leq \textstyle \frac{t_o}{3}, \\[0.5em]
				\text{for all $ 1 \leq i \leq N_k^{+}$: }\tau_k(w_i)  \vee t_h(\beta_i^+) \leq \textstyle \frac{t_o}{3} 
			\end{array}
			\right\},
		\end{align*}
		and $ \tau_k(w_i)$ refer to the $\tau_k$ stopping times attached to trajectories in the support of $\omega_1$ as below \eqref{eq:obs_W^+}.
		Each of $G_{k}^{\pm}$ will separately account for one of the two inclusions in~\eqref{eq:coup100}. We first argue that $G_{k-1} \cap G_k^-$ implies $\mathcal{I}_k^- \subset \mathcal{K}_k$ (the removal of $\widetilde{\mathcal{O}}$ is unnecessary for this inclusion). The event $G_{k}^-$ ensures that $\mathcal{I}^{g_0^-, (k+1)L''}(\mu_2)
		= \emptyset$ in \eqref{eq:coup321}. The induction assumption and the occurrence of $G_{k-1}$ imply that $\mathcal{I}_{k-1}^- \subset \mathcal{K}_{k-1}$. In view of \eqref{eq:obs_W^-}, we may thus assume that 
		\begin{equation}
			\label{eq:coup322}
			W_{k-1, i}^-[0, kL''] \subset \mathcal{K}_{k-1} \,\text{for each $1 \leq i \leq N_0^+$}
		\end{equation}
		as part of the induction hypothesis. But by construction, $s_{\ell+1}- s_{\ell}= t_{\ell} > L''$ for all $1 \leq \ell \leq k$ hence \eqref{eq:coup321} implies that
		$W_{k-1, i}^-[0, kL''] = W_{k, i}^-[0, kL''] $. Inserting this into \eqref{eq:coup322}, and then going back to \eqref{eq:coup321} we deduce that the inclusion $\mathcal{I}_k^- \subset \mathcal{K}_k$ follows at once if we argue that 
		\begin{equation}
			\label{eq:coup323}
			W_{k,i}^-[kL''+1, (k+1)L''] \subset \mathcal{K}_{k}.
		\end{equation} Any contribution to $W_{k,i}^-[kL''+1, (k+1)L'']$ falling onto $W_{\ell,i}$ for some $\ell<k$ in \eqref{eq:obs_W^--} is in fact included in $\mathcal{K}_{k-1}$, as $\{W_{\ell,i}(t- s_{\ell}) :  s_{\ell} \leq t < s_{\ell+1}\}$ appears as part of $\widetilde{\mathcal{K}}$ on the right-hand side of \eqref{eq:coup301} for $k=\ell$. Thus, to deduce \eqref{eq:coup323}, it suffices to show that $W_{k,i}[0, 0 \vee ((k+1)L''-s_k)]\subset \widetilde{\mathcal{K}}$, which is automatic as $kL'' \leq s_k$ by induction assumption and $L'' < t_k$, whence $W_{k,i}[0, 0 \vee ((k+1)L''-s_k)] \subset W_{k,i}[0, t_k]$.
		
		We now show that $G_{k-1} \cap G_k^+$ implies the inclusion $  (\mathcal{K}_k \cap ( \Z^d \setminus \widetilde{\mathcal{O}} ) ) \subset \mathcal{I}_k^{+}$, thereby completing the verification of \eqref{eq:coup100}. Combining the induction assumption, \eqref{eq:coup301} and \eqref{eq:coup303}, we see that it is sufficient to argue that 
		\begin{equation}
			\label{eq:coup324}
			(\widetilde{\mathcal{K}} \cap ( \Z^d \setminus \widetilde{\mathcal{O}} ) ) \ \subset \bigcup_{1\leq i \leq N_k^+} 
			W_{k,i}^+[0, L']  \quad ( \stackrel{\eqref{eq:coup303}}{\subset}\widetilde{\mathcal{I}} ).
		\end{equation}
		By condition on the range of $\beta_i^-$, $1\le i \le N_k^-$ inherent to 
		the definition of $t_h$ and due to \eqref{eq:obs_W^-}-\eqref{eq:coup301}, the set $\widetilde{\mathcal{K}} \cap ( \Z^d \setminus \widetilde{\mathcal{O}} )$ is contained in $$
		\bigcup_{1\leq i \leq N_0^-} W_{k,i}[t_h, t_h + L''+ H_{\mathcal{O}} \circ \theta_{t_h + L''}]=  \bigcup_{1\leq i \leq N_0^-}  X_i^-[0,  L''+ H_{\mathcal{O}} \circ \theta_{ L''}].
		$$
		On the other hand, the set on the right-hand side of \eqref{eq:coup324} contains $$\bigcup_{1\leq i \leq N_k^+}  X_i^+[ 0, (L' - \tau_k -t_h)_+]$$ as a subset on account of \eqref{eq:obs_W^+}. Moreover, on the event $\{ Z^{-}  \subset  Z^{+}\} $ implied by $G_k^+$, by definition of $X_i^{\pm}$ in \eqref{eq:obs-cond300}, each of the trajectories $X_i^-$, $1 \leq i \leq N_0^-$ is equal to one of the trajectories among $X_i^+$, $1 \leq i \leq N_k^+$. From this, \eqref{eq:coup324} follows upon observing that, on the event $G_k^+$, $L''+ (H_{\mathcal{O}} \circ \theta_{ L''})(X_i^-) \leq L'' + \textstyle \frac{t_o}{3}$ for any $1\leq i \leq N_{k}^-=N_0^-$ whereas  
		$L' - \tau_k -t_h \geq L' - \textstyle \frac{\mathrm{2g}}{3}= L'' + \textstyle \frac{t_o}{3}$ (recall that $t_h =t_h(\beta_i^+)\leq \textstyle \frac{t_o}{3}$ by definition of $G_k^+$) for any $1\leq i \leq N_{k}^+$.
		
		To complete the proof of the induction step, it remains to argue that the event $G_{k-1} \setminus G_k$ satisfies the estimate set forth in \eqref{eq:obs-G_k-bound}. The event $\{  
		Y_{k}^- \subset U_{k} \}$ appearing in \eqref{eq:obs-coupG_k} is readily seen to have sufficiently 
		high probability upon observing that, on the event $G_{k-1}$ and for each $q \leq i \leq N_0^+$, 
		the starting point $W_{k,i}(0)=Y_{k-1,i}^-$ (cf.~\eqref{eq:obs_W^-} and recall that 
		$\beta_i^-(0)=Y_{k-1,i}^-$ by \eqref{eq:obs-cond300}) of $W_{k,i}$ lies in $U_{k-1}$. Since 
		conditionally on $\mathcal{F}_{k-1}$, $W_{k,i}(\cdot)$ performs a random walk, it then readily 
		follows that $Y_{k,i}$ defined by \eqref{eq:coup304} lies outside $U_k$ with probability at most $e^{-c (\log L)^{\gamma}}$, cf.~\eqref{e:U_k-hard}. Overall this yields
		$$
		Q_k[G_{k-1} \setminus  \{Y_{k}^- \subset U_{k} \}] \leq \textstyle \frac1{L'} P_{g_0^-}\big[ R_{k}< \infty, X_{R_{k-1}} \in U_{k-1}, X_{R_k} \notin U_k, R_k<\infty \big]\leq u |U_{\mathcal{O}}| e^{-c (\log L)^{\gamma}}.
		$$
		The remaining events $G_{k}^{\pm}$ are handled much like in the $k=0$ case, using \eqref{eq:obs-visible}, \eqref{eq:obs-Z-inclusion} and an estimate akin to \eqref{eq:final-good-obs54} to deal with the (unlikely) event that $t_h(\beta_i^+) > \textstyle \frac{t_o}{3} $.
	\end{proof}

	It remains to give the proof of Lemma~\ref{L:obs-condcoup}.
		
	\begin{proof}[Proof of Lemma~\ref{L:obs-condcoup}] Set 
\begin{equation}\label{e:obs-mu-nu}
\mu_{k} = P_{kL''}(f)1_{U_{\mathcal O}}
\end{equation}
i.e., $\mu_{k}(x)= E_x[f(X_{kL''})]1_{U_{\mathcal O}}(x)$ for $x \in \Z^d$, and similarly		
\begin{equation}
			\label{e:mu_k^pm}
			\mu_{k}^- =  P_{kL''}(f1_{B_{K}})1_{U_{k-1}} \stackrel{\eqref{e:obs-mu-nu}}{\leq} \mu_k, \quad \mu_{k}^+ = 
			\textstyle (1+ \varepsilon) \mu_k.
		\end{equation}
		Conditionally on $\mathcal{F}_k$, if $G_{k-1}( \in \mathcal{F}_k)$ does not occur, we simply sample $Z^-$ and $Z^+$ independently with the right conditional law under $q$. Henceforth, we assume that $G_{k-1}$ occurs. Consider a box $B \in \mathcal{O}$. Referring to $\mathcal{Y}^- = ( Y_{k-1,i}^{-} : 1\leq i \leq N_0^+)$  and $\mathcal{Y}^+ = (Y_{k,i}^{+} : 1 \leq i \leq N_k^+)$ as the relevant collections of starting points (both $\mathcal{F}_{k}$-measurable), let $N^{\pm}(B)=| \{ Y \in \mathcal{Y}^\pm : Y \in B\}|$ and let $Y^\pm_i(B)$, $1\leq i \leq N^{\pm}(B)$ be an arbitrary deterministic (when conditioning on $\mathcal{F}_{k}$) enumeration of the points of $ \mathcal{Y}^\pm$ lying in $B$. 
				By construction, as further detailed below, for $B \cap U_{k-1} \neq \emptyset$,
		\begin{equation}\label{e:hard-slt-coup1}
			\begin{split}
				&\text{under $Q_{k-1}$, }N^{+}(B) \text{ is a Poisson variable with parameter $\lambda^+(B)$, where}\\
				&\text{$\lambda^+(B)= \textstyle \frac1{L}P_{\mu_k^+}[X_{H_{\mathcal{O}}} \in B , \, H_{\mathcal{O}} < \infty]$, and on $G_{k-1}$, $N^{-}(B)$ is a.s.~equal}\\
				&\text{to the Poisson variable $| \{ Y \in \mathcal{Y}^- : Y \in B \cap U_{k-1}\}|$}.
			\end{split}
		\end{equation}
Let $\textstyle\lambda^-(B)$ denote the mean of the latter random variable. In view of 
\eqref{eq:coup102}, \eqref{eq:obs-g_0^-} and recalling that the points in $\mathcal{Y}^-$ represent 
the endpoints of $\mathcal{K}_{k-1}$, one finds, due to \eqref{eq:eq:coup101-new} and 
\eqref{eq:coup100}, applying the simple Markov property at time $kL''$ and reversibility, that
\begin{multline*}
\lambda^-(B)= \frac1{L'} P_{g_0^-} [R_{k-1} < \infty, X_{R_{k-1}} \in (B \cap U_{k-1})] \stackrel{\eqref{eq:obs-g_0^-}}{=}  \frac{1}{L}P_{f1_{B_K}}[R_{k-1} < \infty,  \, X_{R_{k-1}} \in (B \cap U_{k-1})]\\
=\frac{1}{L} \sum_{x,y} (f1_{B_K})(x) p_{kL''}(x,y)1_{U_{k-1}}(y)  P_y [X_{H_{\mathcal{O}}} \in B , \, H_{\mathcal{O}} < \infty] \stackrel{\eqref{e:mu_k^pm}}{=}  \frac1{L}  P_{\mu_k^-}[X_{H_{\mathcal{O}}} \in B , \, H_{\mathcal{O}} < \infty].
\end{multline*}
		Let $\eta_B$, for $B \in \mathcal{O}$ with $B \cap U_{k-1} \neq \emptyset$, denote a family of independent Poisson processes on $[0,\infty) \times \Z^d$, independent of $\mathcal{F}_k$, with respective intensity measure
		\begin{equation*}
			\nu_B([0,u] \times K) = u P_{x_B}[X_{t_h} \in K], \quad u \geq 0, \, K \subset \Z^d,
		\end{equation*}
		where $x_B$ denotes the center of the box $B$ and $t_h$ is the homogenization time \eqref{eq:obs-t-h-new}. The processes $\{ \eta_B : B \in \mathcal{O}, B \cap U_{k-1} \neq \emptyset\}$ will be used to construct the desired coupling. For a given realization $\eta_B = \sum_{\lambda} \delta_{(u_\lambda, z_{\lambda})}$, let
		\begin{equation*}
			\xi_{B,1}^{\pm}= \inf \big\{ t \geq 0 : \exists \lambda \text{ s.t. } t P_{Y^\pm_1(B)}[X_{t_h} = z_{\lambda}] \geq u_{\lambda} \big\},
		\end{equation*}
		and, referring to $(u_1^{\pm},z_1^{\pm})$ as the unique (see \cite{PopTeix}, Prop.~4.1(i)) point achieving the infimum above, define recursively for $2 \leq n \leq N^{\pm}(B)$,
		\begin{equation*}
			\xi_{B,n}^{\pm}= \inf \big\{ t \geq 0 : \exists (u_{\lambda}, z_{\lambda}) \notin (u_k^{\pm}, z_k^{\pm})_{1\leq k < n} \text{ s.t. } G_{B, n-1}^{\pm}(z_{\lambda})+ t P_{Y^\pm_n(B)}[X_{t_h} = z_{\lambda}] \geq u_{\lambda} \big\},
		\end{equation*}
		where 
		\begin{equation}\label{e:hard-slt-coup3}
			G_{B, n}^{\pm}(z)= \sum_{1\leq k \leq n} \xi_{B,k}^{\pm} P_{Y^\pm_k(B)}[X_{t_h} = z], \quad 1 \leq n \leq N^{\pm}(B).
		\end{equation}
		Then, on account of \cite{PopTeix}, Prop.~4.3 (see in particular (4.12)), conditionally on $\mathcal{F}_k$, for each $B$ and $1 \le n \le N^{\pm}(B)$ the random variable $z_n^{\pm} \equiv z_n^{\pm}(B) $ above has law $\mu_{Y^\pm_n(B)}$. Hence the collection $Z^{\pm} \stackrel{\text{def.}}{=} \{ z_n^{\pm}(B) : 1 \le n \le N^{\pm}(B), \, B \in \mathcal{O}, \, B \cap U_{k-1} \neq \emptyset  \}$ satisfies \eqref{eq:obs-Z-law}. Moreover, due to \cite{PopTeix}, Cor.~4.4, on $G_{k-1}$, the inclusion in \eqref{eq:obs-Z-inclusion} holds unless the event
		\begin{equation}\label{e:hard-slt-coup4}
			\big\{ G_{B}^{-}(z) > G_{B}^{+}(z) \text{ for some $z \in \Z^d$ and $B \cap U_{k-1} \neq \emptyset$} \big\}
		\end{equation}
		occurs, where $G_B^{\pm}(z)$ is short for $G_{B, N^{\pm}(B)}^{\pm}(z)$; the restriction to $B$ intersecting $U_{k-1}$ rather than $B \in \mathcal{O}$ in 
		\eqref{e:hard-slt-coup4} is owed to the fact that, on the event $G_{k-1}$, one has $Y^{-}_i(B) = Y_{k-1,i}^-(B)\in U_{k-1}$ for all $i$ due to \eqref{eq:coup100}, which implies that $B\cap U_{k-1} \neq \emptyset$.
		Thus, the proof is complete once we argue that the intersection of $G_{k-1}$ and the event in \eqref{e:hard-slt-coup4} have probability bounded by $e^{-c\varepsilon^2 M_{\mathcal{O}}}$, yielding \eqref{eq:obs-Z-inclusion}. In view of \eqref{e:hard-slt-coup3} and the definition of $\widetilde{\mathcal{O}}$, and owing to \eqref{e:exit-dist1},
		for any $B$, $n$ and $\sigma \in \{ \pm\}$,
		\begin{equation}\label{e:hard-slt-coup5}
			\bigg| \frac{G_{B, n}^{\sigma}(z)}{ P_{x_B}[X_{t_h} = z] \sum_{k=1}^n \xi_{B,k}^{\sigma}} -1 \bigg| \leq C 
			\ell_{\mathcal O}^{-1/100} \ (\leq \textstyle \frac{\varepsilon}{100}),
		\end{equation}
where the bound in parenthesis follows by imposing a suitable assumption on $\ell_{\mathcal O}$, as 
appearing below 
\eqref{eq:obs-Z-inclusion}. As we now explain, from 
\eqref{e:hard-slt-coup3}-\eqref{e:hard-slt-coup5}, one sees that the event in \eqref{e:hard-slt-coup4} is included in the union $E= \bigcup_{B}E_1 \cup E_2 \cup E_3 \cup E_4$, 
where $B$ ranges over all boxes in $\mathcal{O}$ satisfying $B \cap U_{k-1} \neq \emptyset$ and 
with $E_i=E_i(B)$ given by
\begin{align*}
&\textstyle E_1= \{ N^+(B) < (1- \frac{\varepsilon}{100}) \lambda^+(B)\}, \\
&\textstyle E_2= \{ N^-(B) > (1+ \frac{\varepsilon}{100})(1+\varepsilon)^{-1}\lambda^+(B)\},\\
&E_3= \big\{ \textstyle \sum_{1\leq k \leq \lfloor(1- \frac{\varepsilon}{100}) \lambda^+(B)\rfloor} \xi_{B,k}^{+} < \textstyle (1- \frac{\varepsilon}{50}) \lambda^+(B) \big\}, \\
&\textstyle E_4= \big\{ \sum_{1\leq k \leq \lceil(1+ \frac{\varepsilon}{100}) (1+\varepsilon)^{-1} \lambda^+(B)\rceil} \xi_{B,k}^{-} >  (1+ \frac{\varepsilon}{50})(1+\varepsilon)^{-1} \lambda^+(B) \big\};
\end{align*} 
indeed, whenever the complements of $E_1$ and $E_3$ jointly occur, one obtains for any $z \in B$ that 
$$\frac{G_{B}^{+}(z)}{P_{x_B}[X_{t_h} = z]} \geq (1- C \ell_{\mathcal O}^{-1/100}) \sum_{1\leq k \leq \lfloor (1- \frac{\varepsilon}{100})\lambda^+(B)\rfloor} \xi_{B,k}^{+} > \textstyle  (1- \frac{\varepsilon}{50}) \lambda^+(B).$$ 
		A corresponding upper bound for $ \frac{G_{B}^{-}(z)}{P_{x_B}[X_{t_h} = z]}$ holds uniformly in $z \in B$ on the complement of $E_2 \cup E_4$ yielding overall that the bound
		$$\sup_{z,B} \frac{G_{B}^{-}(z)}{G_{B}^{+}(z)} \leq \frac{1+ \frac{\varepsilon}{50}}{1- \frac{\varepsilon}{50}} \cdot (1+ \varepsilon)^{-1}
		\leq 1 : \text{ on the event $E^c$.}$$ 
		Thus, \eqref{eq:obs-Z-inclusion} follows at once if $Q_k[E]$ is bounded by the exponential term on the right-hand side of \eqref{eq:obs-Z-inclusion}, which follows readily using standard large deviation bounds for Poisson and exponential random variables in combination with the following two facts. First, one observes that $\frac{\lambda^-(B)}{\lambda^+(B)} \leq (1+ \varepsilon)^{-1}$, as follows plainly from \eqref{e:mu_k^pm} and \eqref{e:hard-slt-coup1}. In particular this implies that $E_2$ and $E_4$ are indeed deviant events. Second, one uses that 
\begin{equation}\label{eq:lambda_LB}
\lambda^{+}(B) \geq M_{\mathcal{O}},
\end{equation}
 which we explain momentarily. The ensuing large-deviation estimates imply that $\P[E_i] \leq e^{-c \varepsilon^2 \lambda^+(B)}$ for all $1\leq i \leq 4$. In particular, in the case of $E_2$ for instance, one observes that $\P[E_2] \le \P[N^-(B) > \lambda^-(B) + \frac{\varepsilon}{100}(1 + \varepsilon)^{-1}\lambda^+(B)]$ and applies the standard estimate $\P[X \geq \lambda + x] \le \exp\{-\frac{x^2}{2(\lambda + x)}\}$ valid for a Poisson variable $X$ with mean $\lambda$ and all $x>0$.
  
 To see \eqref{eq:lambda_LB}, one first recalls $\lambda^+(B)$ from \eqref{e:hard-slt-coup1} with $\mu_k^+$ as in \eqref{e:mu_k^pm} and applies the density condition \eqref{eq:obs-dense}, observing that $\mu_k^+ \geq \mu$ with $\mu$ as given by \eqref{eq:mu-cond}. The latter follows because $\mu_k^+ \geq \mu_k$, see \eqref{e:obs-mu-nu}, and for every $x \in U_{\mathcal{O}}$, one has that $E_x[f(X_{kL''})] \geq \alpha^{-1}$ for $L \geq C(\gamma)$ and $0 \leq k \leq k_L$ using that $P_x[X_t \in B_{K+L}] \geq c$ for all $x \in U_{\mathcal{O}}$ and $0 \leq t \leq L$ and the assumption that $f\ge (\log L)^{-\gamma}$ on $B_{K+L}$.
		\end{proof}

\section{Good obstacle sets}\label{sec:obstacle_set}

We now construct examples of good obstacle sets $\mathcal{O}$. Recall from Definition~\ref{def:O} that $\mathcal{O}  \subset B_{K+2L}$ for given positive integers $K,L$ is $(\delta_{\mathcal{O}},M_{\mathcal{O}})$-\textit{good} if the random walk satisfies the visibility and density conditions \eqref{eq:obs-visible}-\eqref{eq:obs-dense}. Proposition~\ref{L:final-good-obs} below, which formalizes and extends \eqref{eq:ex-good-obs}, yields a large class of such sets that will cover all applications we have in mind. In particular, as will be seen in the proof of Theorem~\ref{thm:long_short}, presented at the end of this section, this includes relevant situations where the obstacle set is disordered.

Throughout this section, let $B \in \{ B(x,N): x \in \Z^d, \, N \geq 0\} \cup \{ \emptyset\}$ (in the latter case we set $\partial B= \emptyset$) and $K \geq 0$, $L \geq 1$ be integers. For a length scale $\widetilde{L} \geq 1$ recall that $\widetilde{ \mathbb{L}} =3\widetilde{L} \Z^d \cap B_{K + 3L/2}$ (see above \eqref{e:obs-tilde-C}) and from~\eqref{e:obs-tilde-C}
that $\widetilde{\mathcal C}$, the set of cells, consists of all sets of the form $\widetilde{C}= B(z,\widetilde{L})$ such that $z \in \widetilde{ \mathbb{L}}$ and $B(z, 2\widetilde{L}) \cap \partial B= \emptyset $.  

\begin{prop}[$L \geq 1, K \geq 0, \gamma > 10$]
\label{L:final-good-obs}
If 
$(\log L)^{100\gamma} \le \ell=  \ell_{\mathcal{O}} \le L^{\frac1{3d}}$ (as in \eqref{eq:obs-cond-scales}) and $\widetilde{L} =  \lfloor ( \alpha L  \ell^{({d-2})/{2}})^\frac1{d}  \rfloor$ (as in \eqref{e:obs-Ltilde-choice}), for each choice $\{y_{\widetilde{C}} \in\widetilde{C} : \widetilde{C} \in \widetilde{\mathcal C}\}$, 
 the set
\begin{equation}
\label{e:obs-final3}
\mathcal{O} =  \bigcup_{\widetilde{C} \in \widetilde{\mathcal{C}}}  
B(y_{\widetilde{C}}, \ell)\ (\subset B_{K+2L})
\end{equation}
is a $(e^{-c(\log 
L)^{\gamma}}, c\ell^{({d-2})/{2}})$-good obstacle set whenever $L \ge L_0( \gamma,d)$.
\end{prop}

The bulk of this section is devoted to proving Proposition~\ref{L:final-good-obs}. To ease notations, we will routinely
 assume that $L$ is large enough in a manner possibly depending on $\gamma$ (and $d$) 
as might be required for various bounds to hold. We first isolate the following estimate, 
which is key. It exhibits a `mean free path' for the random walk among the (obstacle) set $\mathcal{O}$ in \eqref{e:obs-final3}. 
In the sequel we set $\mathcal{B}_{\mathcal{O}}= \{ B(y_{\widetilde{C}}, \ell) : \widetilde{C} \in \widetilde{\mathcal{C}}\}$ so that 
\eqref{e:obs-final3} is of the form \eqref{eq:obstacles}.

\begin{lemma}
\label{L:mfp} 
Under the assumptions of Proposition~\ref{L:final-good-obs}, there exists 
$\lambda$ such that, defining
	\begin{equation}\label{e:mfp1}
		t_{\mathcal{O}}=t_{\mathcal{O}}(\lambda) = \textstyle \lambda \ell^2 \big(\frac{\widetilde{L}}{ \ell}\big)^d, 
	\end{equation}
	the following holds. For all $B \in \mathcal{B}_\mathcal{O}$, all $x \in \Z^d$ such that both $d(x,\mathcal{O}) \geq \frac1{10} \widetilde{L}$ and $ \frac12  \sqrt{ t_{\mathcal{O}}} \leq d(x,B) \leq  \sqrt{ t_{\mathcal{O}}}$ hold, and all $y \in \Z^d$ such that $d(y, \mathcal{O}) \geq \frac1{10} \widetilde{L}$ and $ d(y,B ) \leq 10 \widetilde{L}$, one has
	\begin{equation}
		\label{e:mfp2}
		g_{\mathcal{O}}(x,y) \geq \Cl[c]{c:obs-green} g(x,y),
	\end{equation}
	where $g_{\mathcal{O}}$ denotes the Green's function killed upon entering $\mathcal{O}$ (cf.~\eqref{eq:Greenkilled}).
\end{lemma}

\begin{remark}\label{R:MFP}
It is worth highlighting that $\widetilde{L} \ll \sqrt{t_{\mathcal{O}}}$ by our choice of scales in Proposition~\ref{L:final-good-obs}, so $\sqrt{t_{\mathcal{O}}}$ gives a (large) scale for which, despite 
their presence, the effect of the obstacles comprising $\mathcal{O}$ is not too strongly felt, as indicated by \eqref{e:mfp2}. 
\end{remark}

\begin{proof}
For all $B \in \mathcal{B}_{\mathcal{O}}$ and a parameter $M \geq 1$ to be chosen momentarily, introduce the set 
$V= V_B =\{ z \in \Z^d :  d(z,B) \leq  M \sqrt{ t_{\mathcal{O}}} \}$. Using the strong Markov property at time $H_{\mathcal{O}}$, one sees that for all $x,y  \in \Z^d$,
	\begin{equation}
		\label{e:mfp3}
		g_{\mathcal{O}}(x,y) = g(x,y) - E_x\big[g(X_{H_{\mathcal{O}}},y) 1\{H_{\mathcal{O}} < \infty\}\big].
	\end{equation}
	Applying the bound \eqref{eq:Greenasympt} repeatedly and choosing $M \geq 1$ large enough (depending 
	only on $d$) yields that for all $y$ satisfying $ d(y,B ) \leq 10 \widetilde{L}$ and $x$ such that 
	$d(x,B) \leq \sqrt{t_{\mathcal{O}}}$,
	\begin{multline}
		\label{e:mfp4}
		E_x\big[g(X_{H_{\mathcal{O}}},y) 1\{H_{\mathcal{O}} < \infty, X_{H_{\mathcal{O}}} \notin V \}\big] \\
		\leq \sup_{z \in \Z^d \setminus V} g(z,y) \leq \frac{C}{M^{d-2}} \frac1 {(t_{\mathcal{O}}\vee 1)^{\frac{d-2}{2}}} \leq \frac{C'}{M^{d-2}} g(x,y) \leq  \frac{1}{3} g(x,y),
	\end{multline}
where in the second bound we use that $\widetilde L /\sqrt{t_{\mathcal O}} \to 0$ as $L \to \infty$ (see Remark~\ref{R:MFP} above) while the last bound requires $M$ to be sufficiently large, which is henceforth fixed. In what follows, let $y_{\mathcal{O}}$ range over all centers of balls 
$B(y_{\mathcal{O}}, \ell) \equiv B_{y_{\mathcal{O}}}$ in the obstacle set $\mathcal{O}$ having 
non-empty intersection with $V$. For a set $U$ we also introduce the handy notation $g(z, U)= \sup_ 
{w\in U} g(z,w) = g(U,z)$. It follows that for all $x,y \in \Z^d$,
	\begin{multline}
		\label{e:mfp5}
		E_x\big[g(X_{H_{\mathcal{O}}},y) 1\{H_{\mathcal{O}} < \infty, X_{H_{\mathcal{O}}} \in V \}\big] \\
		\leq \sum_{y_{\mathcal{O}}} \sum_{z } P_x\big[ H_{B_{y_\mathcal{O}}}< \infty, \, X_{H_{B_{y_\mathcal{O}}}}=z \big] g(z,y) \stackrel{\eqref{eq:lastexit}, \eqref{e:cap-box}}{\leq} C \ell^{d-2} \sum_{y_{\mathcal{O}}} g(x,B_{y_{\mathcal{O}}}) g(B_{y_{\mathcal{O}}}, y) .
	\end{multline}
	The next estimate is key in handling \eqref{e:mfp5}. For all $z \in \Z^d$ satisfying $d(z,\mathcal{O}) \geq \frac{\widetilde{L}}{10}$ and with the sum ranging over all centers  $y_{\mathcal{O}}$ of balls in $\mathcal{O}$ satisfying the given constraint below (recall that $M$ is fixed), one finds that for all $\lambda > 0$,
	\begin{equation}
		\label{e:mfp6}
		\sum_{| y_{\mathcal{O}}-z| \leq 3M\sqrt{t_{\mathcal{O}}}} g(z,B_{y_{\mathcal{O}}}) \stackrel{\eqref{e:obs-tilde-C}}{\leq} C \sum_{1 \leq k \leq C' \frac{\sqrt{t_{\mathcal{O}}}}{\widetilde{L}}} \frac{k^{d-1} }{(\widetilde{L}k)^{d-2}} \leq  \frac{Ct_{\mathcal{O}}}{\widetilde{L}^d} \stackrel{\eqref{e:mfp1}}{\leq} C\lambda \ell^{2-d};
	\end{equation}
	here, in obtaining the first estimate, we crucially used the fact that, by construction of $\mathcal{O}$, in each annulus $B(z, c(k+1)\widetilde{L}) \setminus B(z, ck\widetilde{L})$ with $c= 20^{-1}$, there are at most $Ck^{d-1}$ obstacles $B_{y_{\mathcal{O}}}$, each contributing an amount $g(z,B_{y_{\mathcal{O}}})  \leq C'(k\widetilde{L})^{2-d}$,
	and that $d(z,\mathcal{O}) \geq \frac1{10}\widetilde{L}$, whence the sum over $k$ starts at $k=1$. Returning to \eqref{e:mfp5}, one then considers separately the contributions to the sum according to whether i) $d(B_{y_{\mathcal{O}}}, y) > \frac{\sqrt{t_{\mathcal{O}}}}{10}$, or ii) $d(B_{y_{\mathcal{O}}}, y) \leq \frac{\sqrt{t_{\mathcal{O}}}}{10}$. Recall that the centers 
	$y_{\mathcal{O}}$ appearing in \eqref{e:mfp5} all have the property that $B_{y_{\mathcal{O}}}\cap 
	V \neq \emptyset$. From here on we tacitly assume that $x,y$  in \eqref{e:mfp5} satisfy all 
	assumptions above \eqref{e:mfp2}. In case i), using the fact that $g(B_{y_{\mathcal{O}}}, y) \leq C t_{\mathcal{O}}^{-(d-2)/2} \leq C' g(x,y)$, one applies \eqref{e:mfp6} with $z=x$ to control the 
	resulting sum over $g(x, B_{y_{\mathcal{O}}})$. In case ii), one uses instead that $d(x, B_{y_{\mathcal{O}}}) \geq c \sqrt{t_{\mathcal{O}}}$ (which follows because $|x-y| \geq \frac{1}5 
	\sqrt{t_{\mathcal{O}}}$ when $L \geq C$), whence $g(x, B_{y_{\mathcal{O}}}) \leq Cg(x,y)$ and one applies \eqref{e:mfp6} with $z=y$ instead, yielding overall that
	\begin{equation}
		\label{e:mfp7}
		E_x\big[g(X_{H_{\mathcal{O}}},y) 1\{H_{\mathcal{O}} < \infty, X_{H_{\mathcal{O}}} \in V \}\big] \leq C \lambda g(x,y) ,
	\end{equation}
	for all $\lambda > 0$ and  $x,y$ as above \eqref{e:mfp2}. Choosing $\lambda$ small enough so the last expression in \eqref{e:mfp7} is at most $\frac13g(x,y)$ and inserting the resulting estimate along with \eqref{e:mfp4} into \eqref{e:mfp3} yields \eqref{e:mfp2}.
\end{proof}

With Lemma~\ref{L:mfp} at our disposal, we supply the proof of Proposition~\ref{L:final-good-obs}.

\begin{proof}[Proof of Proposition~\ref{L:final-good-obs}]
We 
need to verify the two conditions \eqref{eq:obs-visible} and \eqref{eq:obs-dense} with 
\begin{equation}\label{e:obs-final4}
\delta_{\mathcal{O}} = e^{-c(\log L)^{\gamma}} \mbox{ and } M_{\mathcal{O}} = c \ell^{(d-2)/2}.
\end{equation}
We first show \eqref{eq:obs-visible}. Recalling $\widetilde{L}$ from \eqref{e:dense401}, let
	\begin{equation}
		\label{eq:final-good-obs51}
		\widetilde{T}= \widetilde{L}^2 (\log L)^{\gamma}.
	\end{equation}
	With $n_0 = \lfloor {t_o}/{3\widetilde{T}} \rfloor$ and $t_o$ as in \eqref{def:gap}, applying the Markov property successively at times $n \widetilde{T}$, $1\leq n \leq n_0$, one obtains for all $x \in B_{K+L}$, with $R= K+L+\sqrt{t_o (\log L)^{\gamma}}$, that
	\begin{equation}
		\label{eq:final-good-obs52}
		\begin{split}
			P_x[H_{\mathcal{O}} > \textstyle \frac{t_o}{3} ] &\leq P_x\big[ X_{[(n-1)\widetilde{T}, n \widetilde{T}]} \cap \mathcal{O} = \emptyset, \, 1 \leq n \leq n_0, \, \textstyle\frac{t_o}{3}  < T_{B_R} \big] + P_x\big[  T_{B_R} \leq\textstyle \frac{t_o}{3} \big]\\[0.3em]
			& \leq \sup_{y \in B_{R}} P_y [H_{\mathcal{O}} > \widetilde{T}]^{n_0} + \sum_{0 \leq n \leq \frac{t_o}3}P_0\big[|X_n| > \textstyle \sqrt{t_o (\log L)^{\gamma}} \, \big].
		\end{split}
	\end{equation}
	Recalling $t_o$ from \eqref{def:gap} and applying a standard upper bound on the heat kernel, one finds that the last term in the second line of \eqref{eq:final-good-obs52} is bounded by $e^{-c(\log L)^{\gamma}}$, which is affordable with a view towards showing \eqref{eq:obs-visible} with $\delta_{\mathcal{O}}$ as in \eqref{e:obs-final4}. To bound $P_y [H_{\mathcal{O}} > \widetilde{T}]$ uniformly in $y \in B_{R}$, one observes that, by construction of $\mathcal{O}$, see \eqref{e:obs-final3} and \eqref{e:obs-tilde-C}, for any such $y$, there exists a box $B=B(y_{\tilde{C}}, \ell) \in \mathcal{B}_\mathcal{O}$ with the property that 
	\begin{equation}
		\label{eq:final-good-obs53}
		d(y,B) \leq \Cl{c:dist-obst} \widetilde{L}
	\end{equation}
	in particular, one notes to this effect that the proximity of a box $B$ asserted in \eqref{eq:final-good-obs53} is not spoiled by the removal of the boxes near the boundary of $B(x,N)$ which is happening in \eqref{e:obs-tilde-C}. Let $U = B(y, 2 \Cr{c:dist-obst} \widetilde{L})$ and $T_U$ denote the exit time from $U$. Classically $E_y[T_U] \leq C \widetilde{L}^2$ and by a Khas'minskii-type argument, see e.g.~\cite[(2.22) and (2.24)]{zbMATH07049491} in this setup, one has that
	\begin{equation}
		\label{eq:final-good-obs54}
		P_y[T_U \geq \widetilde{T}] \leq \exp\big\{- \textstyle \frac{\widetilde{T}}{2 E_y[T_U]} \big\} \stackrel{\eqref{eq:final-good-obs51}}{\leq} e^{-c (\log L)^{\gamma}}. 
	\end{equation}
	On the other hand,\begin{equation}
		\label{eq:final-good-obs55}
		P_y[H_{\mathcal{O}}  > T_U] \stackrel{\eqref{eq:final-good-obs53}}{\leq}  P_y[H_{B}  > T_U] = 1 - \sum_{z \in B} g_U(y,z) e_{B,U}(z) \leq 1- c \big({\ell}/{\widetilde{L}}\big)^{d-2},
	\end{equation}
	where the last inequality follows by using that for all $y,z \in B(y,  \Cr{c:dist-obst} \widetilde{L}) \subset U$, one has $g_U(y,z) \geq c g(y,z)$ and similarly $\text{cap}_U(B)= \int d e_{B,U} \geq c \ell^{d-2}$. Since $P_y[H_{\mathcal{O}} > \widetilde{T}]$ is bounded by the sum of the two probabilities on the left of \eqref{eq:final-good-obs54} and \eqref{eq:final-good-obs55}, feeding the bounds into \eqref{eq:final-good-obs52} yields that $P_x[H_{\mathcal{O}} > \textstyle \frac{t_o}{3} ] $ is bounded up to a term of order $e^{-c (\log L)^{\gamma}}$ by
	$$
	\big( 1- c \big({\ell}/{\widetilde{L}}\big)^{d-2} \big)^{n_0} \leq e^{-c' n_0 (\frac{\ell}{\widetilde{L}})^{d-2}} \leq e^{-c' \ell^{d-2}\frac{L}{\widetilde{L}^d (\log L)^{21\gamma}}} \leq \exp\big\{-c' \textstyle \frac{\ell^{d-2}}{M_{\mathcal{O}} (\log L)^{11\gamma}}\big\},
	$$
	where we used that $n_0 \geq c {L}{[(\log L)^{11\gamma} \widetilde{L}^2]^{-1}}$ and substituted $\widetilde{L}$ as chosen in the statement of Proposition~\ref{L:final-good-obs}; see also \eqref{eq:mu-cond} regarding $\alpha$. But the last term in the previous display is bounded by $\exp\{-\ell^{c'}\}$ on account of 
	\eqref{e:obs-final4} and the fact that $\ell \ge (\log L)^{100\gamma}$, and 
	\eqref{eq:obs-visible} follows.

	It remains to argue that \eqref{eq:obs-dense} holds (with $M_{\mathcal{O}}$ as in 
	\eqref{e:obs-final4}) in order to complete the proof. 
	In the sequel it is always tacitly assumed that $B=B(y_B, \ell) \in \mathcal{B}_\mathcal{O}$ satisfies $B \cap U_{\mathcal{O}} \neq \emptyset$. All subsequent estimates are uniform in $B$ unless explicitly said otherwise.
	For such $B (=B(y_B, \ell))$, define the sets (depending on $B$, with $B(K,r)= \bigcup_{x \in K} B(x,r)$ for $K \subset \Z^d$)
	\begin{equation}
		\label{e:dense401}
		\begin{split}
			&U = \textstyle B\big(y_B, \frac{\widetilde{L}}{10}\big), \\ 
			&A =  \big(\Z^d \setminus B(\mathcal{O}, \textstyle \frac{\widetilde{L}}{10}) \big) \cap \big\{ z: \,  \sqrt{t_{\mathcal{O}}}  \leq d(z,B) \leq2 \sqrt{t_{\mathcal{O}}} \big\} \cap U_{\mathcal{O}} 
		\end{split}
	\end{equation}
	with $t_{\mathcal{O}}$ as in \eqref{e:mfp1} and $\lambda$ inherent to the definition of $t_{\mathcal{O}}$ fixed in such a way that the conclusions of Lemma~\ref{L:mfp} hold. Let
	\begin{equation}
		\label{e:obs-h_B}
		h_B(x) =  P_{x}[X_{H_{\mathcal{O}}} \in B , \, H_{\mathcal{O}} < \infty], \, x\in \Z^d.
	\end{equation}
	By construction, see \eqref{e:obs-final3} and the discussion leading up to it, $U \cap \mathcal{O}= B$ and $\Z^d\setminus ( \mathcal{O} \cup U)$ is a connected set. Hence, for $x \in \Z^d \setminus U$, by a last-exit type decomposition on $\partial U$, one gets that
	\begin{equation}
		\label{e:dense823}
		\begin{split}
			h_B(x)&=  \sum_{z \in \partial U} \sum_{n \geq 0} P_x\big[X_1 \notin \mathcal{O},\dots, X_{n-1} \notin \mathcal{O}, X_n = z, \, \{ H_B< \widetilde{H}_{\partial U}\} \circ \theta_n \big]\\
			& = \sum_{z \in \partial U} g_{\mathcal{O}}(x,z) P_z[H_B< \widetilde{H}_{\partial U}].
		\end{split}
	\end{equation}
	To produce a meaningful lower bound on the second line of \eqref{e:dense823}, first note that, applying \cite[Proposition 1.5.10]{Law91}, one has uniformly in $z \in \partial U$,
	\begin{equation}
		\label{e:dense824}
		P_z[H_B< \widetilde{H}_{\partial U}] \geq c \, \text{cap}(B)|\partial U|^{-1}.
	\end{equation}
	Feeding \eqref{e:dense824} into \eqref{e:dense823}, and noting that Lemma~\ref{L:mfp}, which yields a lower bound on the killed Green's function $g_{\mathcal{O}}(x,z)$ appearing in the second line of \eqref{e:dense823}, is in force whenever $x \in A$, cf.~\eqref{e:dense401}, it follows that
	\begin{multline} \label{e:dense825}
		\int h_B d\mu \stackrel{\eqref{eq:mu-cond}}{\geq} \alpha^{-1} \sum_{x\in A}  \sum_{z \in \partial U} g_{\mathcal{O}}(x,z) P_z[H_B< \widetilde{H}_{\partial U}]
		\stackrel{\eqref{e:mfp2}, \eqref{e:dense824}}{\geq} c \, \alpha^{-1}   \text{cap}(B) \times \\ \inf_{z}  \sum_{x\in A}  g(x,z)  \stackrel{(\ast)}{\geq} c' \alpha^{-1}  \ell^{d-2} \sum_{\sqrt{t_{\mathcal{O}}} \leq k \leq 2 \sqrt{t_{\mathcal{O}}} } k = c \alpha^{-1} \ell^{d-2} t_{\mathcal{O}} \stackrel{\eqref{e:mfp1}}=  c \alpha^{-1} \widetilde{L}^d \stackrel{\eqref{e:obs-Ltilde-choice}, \eqref{e:obs-final4} }{\geq}  M_{\mathcal{O}}L;
	\end{multline}
	in deducing the lower bound $(\ast)$ in the previous chain of estimates, we have also crucially used that when summing over $A$, the distribution of the obstacles, manifest through the condition that $x \notin B(\mathcal{O}, \textstyle \frac{\widetilde{L}}{10})$ present in \eqref{e:dense401}, is sufficiently sparse to make this sum comparable to one over the full annulus $\sqrt{t_{\mathcal{O}}} \leq k= d(x,B) \leq 2 \sqrt{t_{\mathcal{O}}}$. The bound \eqref{e:dense825} is precisely \eqref{eq:obs-dense}.
\end{proof}

We conclude this section by presenting the proof of Theorem~\ref{thm:long_short}, which will follow by combining Theorem~\ref{thm:long_short_obstacle} and Proposition~\ref{L:final-good-obs} for a well-chosen (random) obstacle set $\mathcal{O}= \mathcal{O}(\omega)$ governed by the probability $\mathbf{P}$, which allows to couple the boundary clusters appearing in \eqref{eq:long-j-i}.

\begin{proof}[Proof of Theorem~\ref{thm:long_short}]
Recall the event $\mathscr{D}$ from \eqref{e:obs-final2} defined under $\mathbf{P}$. For definiteness, we set $\mathbb Q_{\omega}$ to be an independent coupling of $\mathcal{I}_1$ and $\mathcal{I}_2$ (or any coupling for that matter) whenever $\omega \notin \mathscr{D}$. From here onwards, we always assume that $\mathscr{D}$ occurs.
Since this entails that $\mathrm{Disc}(\widetilde{C})$ is non-empty for any $\widetilde{C} \in \widetilde{\mathcal{C}}$, we can pick in view of \eqref{eq:disc-coup} a point $y_{\widetilde{C}} \in \widetilde{C}$ such that
\begin{equation}\label{eq:choice-y}
\mathrm{Disc}(y_{\widetilde{C}}) \text{ occurs for each $\widetilde{C} \in \widetilde{\mathcal{C}}$.} 
\end{equation}
For any choice of $\ell= \ell_{\mathcal{O}}$ satisfying \eqref{eq:obs-cond-scales}, the collection $\{y_{\widetilde{C}} : \widetilde{C} \in \widetilde{\mathcal{C}}\}$ thereby obtained defines an obstacle set $\mathcal{O}= \mathcal{O}(\omega)$ by \eqref{e:obs-final3}, which is $(\delta_{\mathcal{O}},M_{\mathcal{O}})$-good with $\delta_{\mathcal{O}}$,$M_{\mathcal{O}}$ as in \eqref{e:obs-final4} by virtue of Proposition~\ref{L:final-good-obs}. In particular, Theorem~\ref{thm:long_short_obstacle} applies with this choice of obstacle set and yields a coupling $\mathbb{Q}=\mathbb{Q}_{\omega}$ satisfying \eqref{eq:long_short}. In view of the values for $\delta_{\mathcal{O}}$,$M_{\mathcal{O}}$ given by \eqref{e:obs-final4} and using one of the assumptions on $\varepsilon$ inherent to Theorem~\ref{thm:long_short} (the other one is needed for Theorem~\ref{thm:long_short_obstacle} to apply), one readily sees that the error term appearing on the right-hand side of \eqref{eq:long_short} can be replaced by that of \eqref{eq:long-j-i}. 

To finish the proof, it thus remains to argue that the specific choice of obstacle set $\mathcal{O}= \mathcal{O}(\omega)$ above implies that the inclusion $({\mathcal I}_1\setminus \widetilde{\mathcal O} ) \subset  ({\mathcal I}_2\setminus  \widetilde{\mathcal O} )$ in \eqref{eq:long_short} can be lifted to boundary clusters. More precisely, since the previous inclusion is preserved by taking a union with $\mathcal{I}(\omega)$ on either side, it is sufficient to argue that $\mathbf{P}[\cdot, \mathscr{D}]$-a.s., under $\mathbb{Q}_{\omega}$,
\begin{equation}
\label{e:obs-final101}
\big\{ ({\mathcal J}_1\setminus \widetilde{\mathcal O}(\omega) ) \subset  ({\mathcal J}_2\setminus  \widetilde{\mathcal O}(\omega) ) \big\} \subset
 \big\{ \mathscr{C}^{\partial}_{B}\big(\mathcal V(\mathcal J_{1} )\big)  \supset  \mathscr{C}^{\partial}_{B}\big(\mathcal V(\mathcal J_{2} )  \big)  \big\},
\end{equation}
where $\mathcal J_{i}= \mathcal I_{i} \cup \mathcal{I} (\omega)$, $i=1,2$, with  $\mathcal{I}_1$, $\mathcal{I}_2$ having the laws specified in \eqref{eq:long_short} and (coupled through $\mathbb{Q}_{\omega}$). If it weren't for the removal of the set $ 
\widetilde{\mathcal{O}}$ from the region of inclusion on the left of \eqref{e:obs-final101}, the asserted 
inclusion would be trivial. Thus, to deduce \eqref{e:obs-final101}, it suffices to argue that
\begin{equation}
\label{e:obs-final102}
\text{ $\mathscr{C}^{\partial}_{B}\big(\mathcal V(\mathcal J_{i}) \big) \cap \widetilde{\mathcal{O}}(\omega) = \emptyset $, $i=1,2$;}
\end{equation}
for, recalling that $\mathscr{C}^{\partial}_{B}\big(\mathcal V(\mathcal J_{i})\big)$ refers to the connected component of $\partial B$ in $\mathcal V(\mathcal J_{i}) \cap B$, the inclusion 
$(\mathcal J_1 \cap B) \subset (\mathcal J_2 \cap B)$ implies in particular $\mathscr{C}^{\partial}_{B}\big(\mathcal V(\mathcal J_{2})   
\big)  \subset \mathscr{C}^{\partial}_{B}\big(\mathcal V(\mathcal J_{1})  \big)$, and \eqref{e:obs-final102} ensures that the latter inclusion remains unaffected when altering the configuration of $\mathcal J_1$ or $\mathcal 
J_2$ inside~$\widetilde{\mathcal{O}}= \widetilde{\mathcal{O}}(\omega)$.

We proceed to show \eqref{e:obs-final102}. By \eqref{e:obs-final3} and \eqref{eq:obs-tilde}, the enlarged obstacle set $\widetilde{\mathcal{O}}$ consists of boxes $\widetilde{B}'$ of the form $\widetilde{B}'= B(y_{\tilde{C}}, \tilde{\ell})$, where $y_{\tilde{C}}$ satisfies \eqref{eq:choice-y}. Fix such a box $\widetilde{B}'$. By construction of $\widetilde{\mathcal{O}}$ and since $\mathcal V(\mathcal J_{i}) \subset \mathcal{V}(\mathcal{I}(\omega)) \equiv  \mathcal{V} $, it suffices to show that 
\begin{equation}
\label{e:obs-final103}
\mathscr{C}^{\partial}_{B}( \mathcal{V}) \cap \widetilde{B}' = \emptyset, \ i=1,2;
\end{equation}
cf.~Fig.~\ref{F:discon}.
\bigskip
\begin{figure}[h!]
 \hspace{2cm} \includegraphics[scale=0.85]{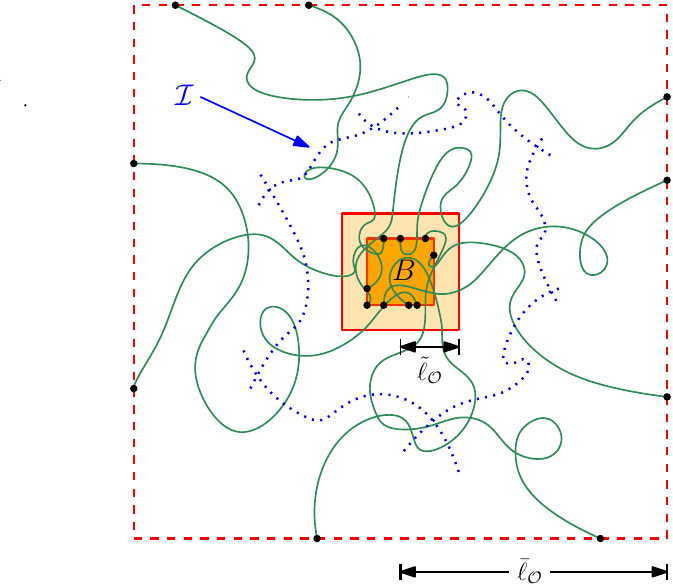}
  \caption{`Shielding' an obstacle $B$ (orange) by enforcing its disconnection in the environment configuration $\mathcal{I}$ (blue). Any vacant cluster in the configuration obtained as the union of $\mathcal{I}_i$ (green) with $\mathcal{I}$ that intersects the complement of the large box of radius $\bar{\ell}_{\mathcal{O}}$ does not `see' $\widetilde{B}$.}
  \label{F:discon}
\end{figure}
Let ${\mathscr{C}}_{B}( \mathcal{V})$ denote the cluster of $\partial B$ (in $ \Z^d$) for $\mathcal{V}$, of which $\mathscr{C}_{B}^{\partial}( \mathcal{V})$ is a part. That is, 
${\mathscr{C}}_{B}( \mathcal{V})$ is the union of all clusters (i.e.~maximal connected 
components) in $\mathcal{V}$ intersecting $\partial B$. Now, omitting subscripts `$\mathcal{O}$,' if $\widetilde{B}'= B(y_{\tilde{C}}, 
\tilde{\ell})$, for some $\tilde{C} = B(z, \widetilde{L}) \in \widetilde{\mathcal{C}}$, is the box in question in 
\eqref{e:obs-final103}, then $B(y_{\tilde{C}}, \bar\ell) \subset B(z, 2\widetilde{L})$ as $\bar\ell \leq \widetilde{L}$ by assumption. But by definition of $\widetilde{\mathcal{C}}$, the last inclusion implies that
\begin{equation}
\label{e:obs-final104}
B(y_{\tilde{C}}, \bar\ell) \cap \partial B \stackrel{\eqref{e:obs-tilde-C}}{=} \emptyset.\, 
\end{equation}
Moreover, owing to \eqref{eq:choice-y} one has that $\widetilde{B}'=B(y_{\tilde{C}}, \tilde{\ell})$ is not connected to $B(y_{\tilde{C}}, \bar\ell )$ in $\mathcal{V}$. Together with \eqref{e:obs-final104}, this implies that $\widetilde{B}$ is not part of ${\mathscr{C}}_{B}( \mathcal{V})$ and \eqref{e:obs-final103} follows since $\mathscr{C}^{\partial}_{B}( \mathcal{V}) \subset {\mathscr{C}}_{B}( \mathcal{V})$.
\end{proof}

\begin{remark}\label{rm:gen_S}
By inspection of the above proof, one finds that $ \mathscr{C}^{\partial}_{B}$ in \eqref{eq:long-j-i} can be replaced by $ \mathscr{C}^{\partial}_{S}$ for any $S \in \{ B' \text{ a box}: B' = B \text{ or } B' \supset B_{K + 2L}\}$. Indeed the proof (in particular, the coupling) remains unchanged, one simply observes that the arguments used to deduce \eqref{e:obs-final101} continue to hold with $ \mathscr{C}^{\partial}_{S}$ in place of $ \mathscr{C}^{\partial}_{B}$. In particular, \eqref{e:obs-final104} holds more generally in that $B(y_{\tilde{C}}, \bar\ell ) \cap \partial S =  \emptyset$ for any $S$ as above. For an example where this flexibility in the choice of $S$ is used, see \cite{RI-I}.
\end{remark}

\section{Coupling between $\mathcal V^{u, L}$ and $\mathcal V^u$}\label{sec:VuVuLcouple}
This section is devoted to proving Theorem~\ref{thm:main III}, which will follow readily by iterating its `one-step' version, Proposition~\ref{P:ind-step} below, over dyadic scales and concatenating the ensuing `one-step' couplings. The proof of Proposition~\ref{P:ind-step}, which relates $\mathcal V^{u, L}$ and $\mathcal V^{u, 2L}$ appears in \S\ref{subsec:VuL2L} and combines Theorems~\ref{thm:short_long-intro} (actually, its enhanced version, Theorem~\ref{thm:short_long}) and Theorem~\ref{thm:long_short} (but see next paragraph), applied multiple times and for a well-chosen environment configuration $\mathcal{I}=\mathcal{I}(\omega)$, as delineated below Theorem~\ref{thm:main III}. 

As it turns out, all relevant environment configurations are of the form $\mathcal{I}\stackrel{\text{law}}{=} \mathcal{I}^{\rho}$ with $\rho$ as in \eqref{e:this-is-rho}, i.e.~they are in fact $\rho$-interlacements as introduced in Section~\ref{e:this-is-rho} for suitable choices of $\rho$. In light of this,  we derive in Theorem~\ref{P:long_short} a refinement of Theorem~\ref{thm:long_short} specific to this type of environment (i.e.~with $\mathbf{P}= \P_{\rho}$, cf.~below \eqref{eq:prelim2}), which includes a very practical (i.e.~verifiable) condition $(\textnormal{C}_{\textnormal{obst}})$, see Definition~\ref{def:background} below, expressed solely in terms of $\rho$, which effectively allows to control the term $ \mathbf{P}[\mathscr D]$ appearing in \eqref{eq:long-j-i_annealed}. Together with Theorem~\ref{thm:short_long}, Theorem~\ref{P:long_short} drives the proof of Proposition~\ref{P:ind-step}. The proof of Theorem~\ref{P:long_short} is given at the end of \S\ref{subsec:VuL2L}. A key input is the control on $\P_{\rho}[\mathscr D]$ implied by the condition $(\textnormal{C}_{\textnormal{obst}})$, stated in Proposition~\ref{P:traps}, which is proved separately in \S\ref{sec:disconnection}.

The main task of this section is to derive the  following `one-step' version of Theorem~\ref{thm:main III}, which is quantitative, see Remark~\ref{R:quanti} below. 
\begin{proposition}\label{P:ind-step} For any $\gamma > 10$, $u \in (0,\infty)$ and dyadic integer $L \geq 1$, if 
\begin{equation}
\label{eq:def_M-quant}
\P[\nlr{}{ \mathcal{V}^{u}}{B_{(\log L)^{200\gamma}}}{\partial B_{L^{1/10d}}}] \geq L^{-\frac1{10}},
   \end{equation}
 then for all $ v\geq u(1+ (\log L)^{-3})$, there exists a coupling $\mathbb Q$ of $\mathcal I_1$, $\mathcal I_2^+$, $\mathcal I_2^-$ satisfying
\begin{align}\label{eq:upper_L2L}
&\mathbb Q\big[\mathscr C^\partial_{B_R}(\mathcal V(\mathcal I_2^+))\subset\mathscr 
C^\partial_{B_R}(\mathcal V(\mathcal I_1))\subset \mathscr C^\partial_{B_R}(\mathcal V(\mathcal I_2^-))\big] \ge 1 - \Cl{C:entropy-2scale} (R + L)^d e^{-c(\log L)^\gamma},
\end{align}
for all $R \geq 1$, where $\Cr{C:entropy-2scale} =\Cr{C:entropy-2scale} (\gamma, u)$, $\mathcal I_1 \stackrel{{\rm law}}{=} \mathcal I^{v, 2L}$, $\mathcal I_2^\pm \stackrel{{\rm law}}{=} \mathcal I^{v_\pm^L, \, L}$ and $v_{\pm}^L = v(1 \pm (\log L)^{-4})$.
\end{proposition}

\begin{remark}\label{R:quanti}
In comparison with Theorem~\ref{thm:main III}, the scale $\ell$ at which disconnection is used (cf.~\eqref{eq:def_M-quant} and \eqref{eq:def_M}) in Proposition~\ref{P:ind-step} is quantitative (in $L$). By means of Proposition~\ref{prop1:cube} one could further replace $\mathcal{V}^u$ by $\mathcal{V}^{u,L}$ in \eqref{eq:def_M-quant}.
\end{remark}

Assuming Proposition~\ref{P:ind-step} to hold, we first complete the proof of Theorem~\ref{thm:main III}.

\begin{proof}[Proof of Theorem~\ref{thm:main III}]
In view of \eqref{eq:def_M}, we see that \eqref{eq:def_M-quant} is in force for any $L \geq C( \gamma,u)$, upon choosing $c,c' \in (0,\frac1{200})$ in the definition of $D(\ell)$ and \eqref{eq:def_M} suitably small. 
We may further assume with $L_n=2^n L$ when $L \geq C(\gamma,u)$, that $\sum_n (\log L_n)^{-k-1} \leq (\log L)^{-k}$ for $k=2,3$ and that $\sum_n (R+L_n)^d \leq C (R+L)^d$, for all $R \geq0$. It then follows by applying Proposition~\ref{P:ind-step} with $L=L_n$ and combining with a straightforward induction argument involving Lemma~\ref{lem:concatenation} to concatenate the resulting couplings that for all $n \geq 0$, $L \geq C(\gamma,u)$ and  $v \geq u(1 + (\log L)^{-2})$, there exists a coupling 
$\mathbb Q_n$ of $\mathcal I_1 \stackrel{{\rm law}}{=} \mathcal I^{v, L_n}$, $\mathcal I_2^+ \stackrel{{\rm 
law}}{=} \mathcal I^{v_+, \, L}$ and $\mathcal I_2^- \stackrel{{\rm law}}{=} \mathcal I^{v_-, \, L}$ such that
\begin{align}\label{eq:upper_LL'}
&\mathbb Q[\mathscr C^\partial_{B_R}(\mathcal V(\mathcal I_2^+))\subset\mathscr C^\partial_{B_R}(\mathcal V(\mathcal I_1))\subset \mathscr C^\partial_{B_R}(\mathcal V(\mathcal I_2^-))] \ge 1 - C (R + L)^d e^{-c(\log L)^\gamma}.
\end{align}
From \eqref{eq:upper_LL'}, \eqref{eq:upper} follows immediately upon choosing $n$ large enough in a manner depending on $R, L,\gamma$ and $v$ and applying Propositions~\ref{prop1:cube} and \ref{prop2:cube} with $\rho=\rho_{v,L_n}$ (see below \eqref{eq:prelim3.1}).
\end{proof}

\subsection{Coupling between $\mathcal V^{u, L}$ and $\mathcal V^{u, 2L}$}\label{subsec:VuL2L}
We proceed to prove Proposition~\ref{P:ind-step}.
In order to show the inclusion \eqref{eq:upper_L2L}, we will soon apply 
Theorem~\ref{thm:long_short} to environment configurations $\mathcal{I}=\mathcal{I}^{\rho}$ attached to certain $\rho$-interlacement sets, as introduced in 
\eqref{eq:prelim2}, with $\rho$ guaranteeing a very high probability of the 
event $\mathscr D$ in \eqref{eq:long-j-i_annealed}. The next result, Theorem~\ref{P:long_short}, is key to this. 
It hinges on the following definition. To provide some intuition, a simple (but for our later purposes insufficient) 
example of admissible profile $\rho$ to keep in mind for the next definition is that of 
uniform trajectories of length $\hat{L}=L$ (i.e.~$\rho(\ell,x)=\rho_{u,L}(\ell,x)= \frac{u}L 1_{L}(\ell), \, x 
\in \Z^d$, cf.~\eqref{eq:prelim3.1}). Recall the average occupation time density $\bar\ell_x=\bar\ell_x^{\rho}$ from~\eqref{eq:occtime}.

\begin{defn}\label{def:background} The function $\rho: \mathbb{N}^* \times \Z^d  \to \R_+$ is said to 
	satisfy~$(\textnormal{C}_{\textnormal{obst}})$ 
(with parameters $(u', u, \gamma, L,K)$) if for some dyadic integer $\hat{L} \in [\frac L8, 8L]$, the following conditions 
are satisfied:
\begin{align}
		&\:\:\bar\ell_x^{\rho}\in [u', u] \text{ and } \rho(\mathbb{N}^*, x) \leq \textstyle \frac{4du}{\hat{L}} \text{ for all $x \in B_{K + 5(L\vee \hat{L})}$;} 
	\label{eq:disconnect_background0} 
	\\[0.6em]
	&\begin{array}{l}\text{$\rho(\ell, \cdot) = \rho(\ell, \cdot)1_{\ell \in \{\hat{L}, \hat{L}/2\}}$, $\frac{\hat{L}}{ 8d} \rho(\frac{\hat{L}}{2}, \cdot) =
			\big(\frac{1 + P_{\hat{L}/2}}{2} \big)f_1$ and $\frac{\hat{L}}{4d} \rho(\hat{L}, \cdot) = f_2$,  where} \\
		\text{$f_1,f_2: \Z^d \to [0,\infty)$ satisfy 
			$u \ge f_1 + f_2 \ge (\log L)^{-{\gamma}}$ for all $x \in B_{K + 5(L\vee \hat{L})}$.
	} \end{array}\label{eq:disconnect_background2} 
\end{align}
\end{defn}

With this we have the following result. Recall the canonical law $\P_{\rho}$ of the Poisson process of $\rho$-interlacements introduced above \eqref{eq:prelim3}, defined on the canonical space denoted $(\Omega_{\rho}, \mathcal{A}_{\rho})$ below.
\begin{theorem}[$\gamma > 10$, \eqref{e:couplings-params}, $l = \frac L{L'}$]\label{P:long_short}
For all $u \in (0,\infty)$, integer $K, N \ge 0$, $B=B(x,N)$ for $x \in \Z^d$, and $f:\Z^d \to [0, u]$ as in Theorem~\ref{thm:long_short_obstacle}, the following holds.
If $\rho$ satisfies $(\textnormal{C}_{\textnormal{obst}})$ with parameters 
$(u', u, \gamma, L, K)$, for some $u'< u$ and \eqref{eq:def_M-quant} holds with $u'(1- (\log L)^{-4})$ in place of $u$, 
there exists for each $\omega \in \Omega_{\rho}$ a coupling $\mathbb Q_{\omega}$ of $\mathcal I_1, \mathcal I_2$ such that
\begin{equation}\label{eq:long-j-i'}
	\mathcal I_1 
	\stackrel{\textnormal{law}}{=} \mathcal I^{f1_{B_{K}},L}, \quad \mathcal I_2 
	\stackrel{\textnormal{law}}{=} \mathcal I^{(1+\varepsilon)P_L^{L'}(f),L'}, \text{ with } \varepsilon 
	=
 l^{-\frac12},
\end{equation}
and for all $L \ge C(u', u, \gamma)$, with $\eta= 1 - C(u \vee 1)(K + 
L)^de^{-c l^{1 / 4}} $, one has
\begin{equation}\label{e:hard-coup-final3-quenched}
\mathbb{P}_{\rho}\big[ \omega \in \Omega_{\rho} : \mathbb Q_{\omega} \big[ \mathscr{C}^{\partial}_{B}\big(\mathcal V(\mathcal I_{1} \cup \mathcal{I}^{\rho}(\omega))  \big)  \supset  \mathscr{C}^{\partial}_{B}\big(\mathcal V(\mathcal I_{2} \cup \mathcal{I}^{\rho}(\omega))  \big)   \big] \geq  \eta \big] \geq  \eta.
\end{equation}
\end{theorem}

\begin{remark}\label{rmk:background}
\begin{enumerate}[label*=\arabic*)]
\item When combined, Theorems~\ref{thm:short_long} and~\ref{P:long_short} are very useful in practice. The proof of Proposition~\ref{P:ind-step} below is a testimony to this. For further applications, we refer the reader to the proofs of \cite[Proposition 4.3]{RI-I} and \cite[Lemma 7.2]{RI-I}. In particular, as will be seen shortly, cf.~the proof of Lemma~\ref{L:Cobst-ver}, the set of conditions forming $(\textnormal{C}_{\textnormal{obst}})$ are relatively straightforward to verify in practice, even in cases where $\rho$ itself is rather involved.

\item (Range of $\hat L$). \label{R:lhat} The reason we require the scale $\hat L$ in Definition~\ref{def:background} to be dyadic is technical, and has to do with iterated applications of Theorem~\ref{thm:short_long} (see the proof of  Proposition~\ref{P:traps} below), which are facilitated if the `starting' scale $\hat L$ is an integer power of $2$.
Moreover, by inspecting the proof of Theorem~\ref{P:long_short}, one sees that the condition on $\hat L$ can be relaxed to the requirement that $\hat{L} \in [\frac L8, 8L]$ be of the form $\hat L= \hat{L}_1 - \hat{L}_2$, where $\hat{L}_1, \hat{L}_2$ are dyadic integers such that $\hat{L}_2 \ge 8L(\log \frac{L}{8})^{-4\gamma}$. This slight extension of the range of $\hat L$ allowing for differences of dyadics is convenient at times, see~Section 7 in \cite{RI-I} for instance.
\end{enumerate}
\end{remark}

Assuming Theorem~\ref{P:long_short} to which we will return later, we can now give the proof of Proposition~\ref{P:ind-step}. The following consequence of Theorem~\ref{P:long_short} will be enough for this purpose.

\begin{corollary}[Annealed coupling] \label{P:long_short_ann}
Under the assumptions of Theorem~\ref{P:long_short}, there exists a coupling $\mathbb{Q}$ of $\mathcal I_1, \mathcal I_2, \mathcal{I}^{\rho}$ with marginals specified by \eqref{eq:long-j-i'} and \eqref{eq:prelim2}, such that  $\mathcal I^{\rho}, \mathcal{I}_1$ are independent, $\mathcal I^{\rho},\mathcal{I}_2$ are independent, and for  $L \ge C(u', u ,\gamma)$: 
\begin{equation}\label{e:hard-coup-final3}
\mathbb Q \big[ \mathscr{C}^{\partial}_{B}\big(\mathcal V(\mathcal I_{1} \cup \mathcal I^{\rho})  \big)  \supset  \mathscr{C}^{\partial}_{B}\big(\mathcal V(\mathcal I_{2}\cup \mathcal I^{\rho})  \big)   \big] \geq 1 - C'(u \vee 1)(K + 
L)^de^{-c l^{1 / 4}}.
\end{equation}
\end{corollary}

\begin{proof} This follows immediately upon defining
\begin{equation*}
{\mathbb{Q}}\big[ f_1( \mathcal{I}_1) f_2( \mathcal{I}_2) f_3( \mathcal{I}^\rho) \big] =\\
\int d\P_{\rho}(\omega)  f_3( \mathcal{I}^\rho(\omega)) 
E^{\mathbb Q_{\omega}}\big[f_1( \mathcal{I}_1) f_2( \mathcal{I}_2) \big],
\end{equation*}
for bounded measurable  functions $f_i: \{ 0,1\}^{\Z^d} \to \R$, $1 \leq i \leq 3$, with $\mathbb Q_{\omega}$ as provided by Theorem~\ref{P:long_short}, and interpreting the right-hand side of \eqref{e:hard-coup-final3} as $1-2(1-\eta)$ with $\eta$ as in \eqref{e:hard-coup-final3-quenched}.
\end{proof}

\begin{proof}[Proof of Proposition~\ref{P:ind-step}]
We will only show the existence of a coupling exhibiting the second inclusion in \eqref{eq:upper_L2L} as the proof for the other inclusion is similar and the final 
result follows via chaining (using Lemma~\ref{lem:concatenation}). We will tacitly assume 
that $L \geq C( u, \gamma)$, and that the measure $\mathbb P$ carries two independent 
Poisson point processes $\omega_1$ and $\omega_2$ on $ W_+ \times \R_+$ each having intensity $\nu$ as in \eqref{eq:mu_intensity}.
To avoid clutter of notations, we will abbreviate $v_-=v_-^L$ below.

In the course of the proof, we will define a succession of configurations $\{{\mathcal I}^a; \,0 \le 
a \le A\}$ interpolating (in law) between  $\mathcal{I}^{v_-, L}$ and $\mathcal I^{v, 2L}$, where 
$\mathcal I^{a+1}$ is obtained from ${\mathcal I}^a$ by replacing a small fraction of the 
$L$-trajectories with $2L$-trajectories; see \eqref{eq:intermed_decomp} below. For each $0 \leq a < 
A$, we will then provide a coupling $\mathbb Q_a$ between the laws of $\mathcal I^a$ and 
$\mathcal I^{a+1}$ by combined application of Theorems~\ref{thm:short_long} and \ref{P:long_short}. 
We will arrive at the desired coupling $\mathbb Q$ by concatenating the $\mathbb Q_a$'s.
	
We now introduce the relevant intermediate configurations (in law). Let $\delta_1 = (\log L)^{-4}$ and $A \stackrel{{\rm def}.}{=} \lceil 
\delta_1^{-1} \rceil$ so that $\frac{1}{A} \le {\delta_1}$. For any integer $a$ satisfying $0 \leq a 
\leq A$, let
\begin{equation}
\label{eq:defgh^ka}
g^a(\cdot) = \big(1 - \textstyle\frac aA\big) \, \big(\textstyle\frac{1 + P_L}{2}\big)1_{B_{R+2L}}(\cdot),\quad 
h^a(\cdot) = 1_{B_{R + 2L}^c}(\cdot)  + \frac aA\,  \big(1_{B_{R+2L}}(\cdot) + \delta_1 1_{B_{R + 6L}}(\cdot)\big),
\end{equation}
with $P_L$ as in \eqref{eq:P_n-def}. Clearly $g^0 = (\textstyle\frac{1 + P_L}{2})1_{B_{R+2L}}$, 
$g^A = 0$ and similarly for $h^a$. Moreover, since $P_n$ has range $n$, we get 
$g^0_{|B_{R+L}} = 1$. Therefore introducing under $\P$ the configurations 
\begin{align}\label{eq:intermediate_config}
\mathcal{I}^{a} (\omega_1, \omega_2) \stackrel{{\rm def}.}{=} {\mathcal{I}}^{\, v_-  g^a, 
L}(\omega_1)  \cup  {\mathcal{I}}^{\, v_-  h^a,2L} (\omega_2),
\end{align}
for $0 \le a \le A$, we note that ${\mathcal{I}}^0 \cap B_R = \mathcal{I}^{v_-, L}(\omega_1) \cap B_R$ and ${\mathcal{I}}^{A} \subset {\mathcal{I}}^{v, 2L}(\omega_2)$. For any $0 \le a < A$, one extracts 
from both $\mathcal{I}^a$ and $\mathcal{I}^{a+1}$  a joint `bulk' contribution $\mathcal{I}^{a,1}$ by 
decomposing
\begin{align}\label{eq:intermed_decomp}
{\mathcal{I}}^{a} = {\mathcal{I}}^{a,1} \cup {\mathcal{I}}^{a,2}, \:\:\: 
{\mathcal{I}}^{a+1} = {\mathcal{I}}^{a,1} \cup {\mathcal{I}}^{a,3},
\end{align}
	where (under $\P$) 
		\begin{equation}\label{def:I'ka}
		\begin{split}
			&{\mathcal{I}}^{a,1}  \stackrel{{\rm def}.}{=} {\mathcal{I}}^{\, v_- g^{a+1},L}(\omega_1) \cup  {\mathcal{I}}^{\, v_- h^a,2L} (\omega_2), \\
			&{\mathcal{I}}^{a,2} \stackrel{{\rm def}.}{=} \mathcal I^{[v_- g^{a+1},\, v_- g^a], L}(\omega_1),\\
			&{\mathcal{I}}^{a,3} \stackrel{{\rm def}.}{=} \mathcal I^{[v_- h^{a},\, v_- h^{a+1}], 2L}(\omega_2). \end{split}
	\end{equation}
	The bulk contribution ${\mathcal{I}}^{a,1}$ is further split as
	\begin{equation}\label{eq:I'ka_decomp}
		{\mathcal{I}}^{a,1} =  {\mathcal{I}}^{a, 1,1} \cup {\mathcal{I}}^{a, 1,2},
	\end{equation}
where, with $U = B_{R + 15L}$ and using that $g^{a+1}_{|U^c} = 0$ and $h^{a}_{|U^c} = 1_{U^c}$,
\begin{equation}\label{eq:I'kaj}
		\begin{split}
			&{\mathcal{I}}^{a, 1,1}  \stackrel{{\rm def}.}{=} {\mathcal{I}}^{\, v_- \,g^{a+1}_{| U}, \,L}(\omega_1) \cup  {\mathcal{I}}^{\, v_- h^a_{| U},\, 2L} (\omega_2), \text{ and} \\
			&{\mathcal{I}}^{a, 1,2} \stackrel{{\rm def}.}{=} {\mathcal{I}}^{\, v_- 
				1_{U^c}, \,2L} (\omega_2).
		\end{split}
	\end{equation}
By construction, the random sets ${\mathcal{I}}^{a, 1,1}, {\mathcal{I}}^{a, 1,2}$, ${\mathcal{I}}^{a,2}$ 
and ${\mathcal{I}}^{a,3}$ are independent under $\mathbb P$. 
	
The set ${\mathcal{I}}^{a, 1,1}$ is going to play the role of `environment' configuration $\mathcal I^\rho$ appearing in the context of Theorem~\ref{P:long_short}. Clearly, in view of \eqref{eq:prelim2}, the set ${\mathcal{I}}^{a, 1,1}$ has the same law (under $\mathbb P$) as $\mathcal I^\rho$ (under $\P_{\rho}$), where $\rho: 
\mathbb{N}^\ast \times \Z^d  \to \R_+$ is given by
\begin{equation}\label{def:v} \textstyle
\rho(\ell, x) = \frac{4d v_-}{L}\big( 1_L(\ell) g^{a+1}(x) + \frac12 1_{2L}(\ell) \, 
h^a(x)\big)1_{{U}}(x).\end{equation}
Recall the obstacle condition $(\textnormal{C}_{\textnormal{obst}})$ from 
Definition~\ref{def:background}. In what follows, $\rho$ is said to satisfy 
$(\textnormal{C}_{\textnormal{obst}})(x)$ for some $x \in \Z^d$ if $ \rho_0$ given by $\rho_0(\ell,y)= \rho (\ell,x+y)$ for all $\ell \in \mathbb{N}^\ast$ and $y \in \Z^d$ satisfies 
$(\textnormal{C}_{\textnormal{obst}})$. We first isolate the following result.
\begin{lemma} \label{L:Cobst-ver} For all $>0$, $\gamma >1$ and $L \geq C(\gamma)$, the density $\rho$ in \eqref{def:v} satisfies $(\textnormal{C}_{\textnormal{obst}})$ with parameters 
$(v(1-3 (\log L)^{-4}), 4v, \gamma , L, K = R + 3L)$.
\end{lemma}
	
\begin{proof}[Proof of Lemma~\ref{L:Cobst-ver}] Verifying $(\textnormal{C}_{\textnormal{obst}})$ amounts to checking conditions~\eqref{eq:disconnect_background0} and 
	\eqref{eq:disconnect_background2} inside the box $B_{K + 5(L\vee \hat{L})}$. We choose $\hat L = 2L$. Since $K = R + 3L$ we thus need 
	\eqref{eq:disconnect_background0}-\eqref{eq:disconnect_background2} to hold in $B_{R + 3 L + 5 \times 2L} = B_{R + 13L} \stackrel{\text{def.}}{=} \hat B$. By \eqref{eq:occtime}, $\bar \ell_x^\rho$ only depends on $\rho(k,y)$ if 
	$|y-x|< k$ and the relevant values of $k$ are $L$ and $2L$, so we can drop the indicator 
	function $1_{U}$ in \eqref{def:v} when dealing with the quantities $\bar \ell_x^\rho$ 
	and $\rho(\ell, x)$ for $ x\in \hat B$, as for the verification of 
	\eqref{eq:disconnect_background0}-\eqref{eq:disconnect_background2} with the above choices. 
		
We first verify condition \eqref{eq:disconnect_background2} for $\rho$ as in \eqref{def:v}. It follows 
directly from the definition of $g^a, h^a$, see~\eqref{eq:defgh^ka}, that $\rho$ 
satisfies~\eqref{eq:disconnect_background2} with $\hat L = 2L$: indeed, $2L \rho(2L, x) = 4d\,f_2$ and  $L \rho(L, x) = 4d\,(\frac{1 + P_{L}}{2})f_1$ with 
		$f_1 + f_2 \in \left[v_-(1 - \frac1A),  v_- (1 + \delta_1)\right]$ for all $x \in \hat{B}$.  
		
We now verify~\eqref{eq:disconnect_background0}. The previous paragraph further yields that 
$\rho(\N^\ast, x) \le 4d\frac{4v}{\hat L}$ with $\hat L = 2L$ for all $L \geq C$. We 
proceed to check the required lower bound on $\bar \ell_x^\rho$. For convenience, we use the 
notation $ \bar \ell_x(\rho) = \bar \ell_x^\rho$ in the sequel. Let $\tilde{\rho} \leq \rho$ be obtained by 
replacing $h^a$ in \eqref{def:v} by $f^a$, defined similarly as $h^a$ but with the function $f 
\stackrel{{\rm def}.}{=} 1_{B_{R + 2L}}$ playing the role of $1_{B_{R + 2L}} + \delta_1 1_{B_{R + 6L}}$. 
By monotonicity and linearity of $\rho \mapsto \bar \ell_{\cdot}(\rho)$ (see~\eqref{eq:occtime}) and by definition of $g^{a+1}$ and $f^a$ (see \eqref{eq:defgh^ka}), it follows that for all $x \in \hat{B}$,
\begin{multline*}
\bar \ell_x(\rho)  \geq \bar \ell_x(\tilde \rho) =  \bar \ell_x\big( \textstyle 4d\, \frac {v_-}{2L}1_{2L}(\cdot) \, f^a(\cdot)\big) + \bar \ell_x\big( \textstyle4d \, \frac{v_-} L1_L(\cdot)\, g^{a+1}(\cdot)\big)\\
\textstyle= \bar \ell_x(\tilde \rho_{{\rm out}}) +  \frac a A \, \bar \ell_x(\tilde \rho_{f}) +  \big(1 -  \frac{a+1}{A}\big)\bar \ell_x(\tilde \rho_{g}),
\end{multline*} 
where $\tilde \rho_{{\rm out}} = 4d \,\frac {v_-}{2L}1_{2L}(\ell)1_{B_{R + 2L}^c}$, $\tilde \rho_{f}(x, \ell) = 4d \,\frac {v_-}{2L}1_{2L}(\ell)1_{B_{R + 2L}}$ and $\tilde \rho_{g}(x, \ell) =  4d 
\frac {v_-}{L}1_{L}(\ell) g^0(x)$ (recall \eqref{eq:defgh^ka}). We will momentarily show that 
\begin{equation}\label{eq:equality-localtimes-new}
\bar \ell_x(\tilde \rho_{g}) = \bar \ell_x(\tilde \rho_{f}), \text{ for all $x \in \hat{B}$.}
\end{equation} 
Together with the previous display, and recalling from above \eqref{eq:defgh^ka} that $A^{-1} \leq (\log L)^{-4}$, this then yields that
 \begin{equation*}
\bar \ell_x(\rho) \ge (1 - A^{-1})\bar \ell_x(\tilde \rho_{{\rm out}} + \tilde \rho_{f}) = (1 - A^{-1}) \bar \ell_x\big( \textstyle4d\frac {v_-}{2L}1_{2L}(\cdot) \big) \stackrel{\eqref{eq:loc-time-uL}}{=} (1 - A^{-1}) v_- \ge  v(1-3 (\log L)^{-4}).
 \end{equation*}
 The upper bound condition on $\bar \ell_x(\rho)$  in \eqref{eq:disconnect_background0} is straightforward since 
$\bar \ell(4d\frac {v_-}{L}1_{L}(\ell)) = v_- < v$ by \eqref{eq:loc-time-uL}. This completes the verification of \eqref{eq:disconnect_background0}. It remains to argue that \eqref{eq:equality-localtimes-new} holds. Indeed,
\begin{equation*}
\begin{split}
\bar \ell_x(\tilde \rho_{g}) &= \bar \ell_x \big(4d \textstyle \frac {v_-}{L}1_{L}(\cdot)g^0(\cdot)\big) \displaystyle\stackrel{\eqref{eq:occtime}}{=} \frac {v_-}{L}\sum_{\ell \geq 0} E_x \Big[ \, \sum_{\ell' > 
	\ell}1_{L}(\ell') g^0(X_\ell)\,\Big]	= \frac {v_-}{L}\sum_{0 \le \ell < L} E_x [g^0(X_\ell)]
\\ &= \frac {v_-}{L}\sum_{0 \le \ell < L} (P_\ell g^0)(x) \stackrel{\eqref{eq:defgh^ka}}{=} \frac u{L}\sum_{0 \le \ell < L} P_\ell \Big(\frac{1+P_L}{2}\Big)\tilde f(x) = \frac {v_-}{2L}\sum_{0 \le \ell < 2L} P_\ell f(x) = \bar \ell_x(\tilde \rho_{f}),
\end{split}
\end{equation*}
where in the penultimate step we applied the semigroup property and we omitted the details leading to the final inequality, which are similar to the first four steps. 
	\end{proof}
	
With Lemma~\ref{L:Cobst-ver} at our disposal, we resume the proof of \eqref{eq:upper_L2L}. We proceed to define the desired coupling 
$\mathbb Q_a$ of $ {\mathcal I}^{a}$ and $\mathcal I^{a+1}$, for $0 \leq a < A$. In view of \eqref{eq:intermed_decomp}, this 
amounts to replacing ${\mathcal{I}}^{a,2}$ by ${\mathcal{I}}^{a,3}$. This will rely on a combination of Corollary~\ref{P:long_short_ann} and
Theorem~\ref{thm:short_long}, which will involve an intermediate 
 configuration $\widetilde{\mathcal I}$ comprising shorter (fragmented) trajectories of length $L' \ll L$ and 
appropriate intensity.
	
For convenience, rather than working with ${\mathcal{I}}^{a,2}$ directly, we will work with a `larger' configuration $\mathcal I^{f, L}$ (cf.~\eqref{eq:J-k-a-2-dom} below) instead. Let 
\begin{equation}\label{def:f}
\textstyle f(x) = \frac 1A\, v_-\, g^0(x) + \frac{\delta_1}{80 A}\,v_-1_{B_{R + 4L}}(x), \, x \in \Z^d,\end{equation}
and $L' = {L}{(\log L)^{-4\gamma}}$. One readily verifies that $f$ is supported on a subset of $B_{R + 4L}$ and 
that $2v \ge f \ge    cv \, (\log L)^{-8}$ pointwise on $B_{R + 4L}$ for $L \ge C$. One then straightforwardly deduces from \eqref{def:I'ka} and \eqref{def:f} that for all $L \ge C$,
	\begin{equation}\label{eq:J-k-a-2-dom}{\mathcal{I}}^{a,2} \text{ (under $\P$) }
		\leq_{\text{st.}} \mathcal I^{f, L};
	\end{equation}
the benefit of the second term in \eqref{def:f}, which guarantees ellipticity of $f$ but is unnecessary for \eqref{eq:J-k-a-2-dom} to hold, will soon become clear.
Recall now that ${\mathcal{I}}^{a, 1,1}\stackrel{\text{law}}{=} \mathcal{I}^{\rho}$ with $\rho$ as in 
\eqref{def:v}. Corollary~\ref{P:long_short_ann} is in 
effect for this choice of $\rho$ and with $L'$ as defined below 
\eqref{def:f}, in view of the discussion following \eqref{def:f}, Lemma~\ref{L:Cobst-ver}, and the fact that $v'= v(1-3 (\log L)^{-4})(1- (\log L)^{-4}) \geq u$ for $L \geq C$ by assumption on $v$, so that condition \eqref{eq:def_M-quant} holds with $v'$ in place of $u$ by monotonicity. 
Thus, combining Corollary~\ref{P:long_short_ann} with \eqref{eq:J-k-a-2-dom},  it follows 
that for all $L \geq C( u, \gamma)$, there exists a coupling $\mathbb Q_{1}^a$ of $({\mathcal{I}}^{a,2}\cup 
	{\mathcal{I}}^{a, 1,1})$ and $(\widetilde{\mathcal I} \cup 
	{\mathcal{I}}^{a, 1,1})$, with the two configurations sampled independently for either pair, where $\widetilde{\mathcal I} \stackrel{\text{law}}{=}  \mathcal I^{(1 + l^{-\frac12})P_L^{L'}(f), L'}$, and such that for $B = B_R$,
	\begin{equation}\label{eq:couple_L_by_L'}\mathbb Q_{1}^a \big[ \mathscr{C}^{\partial}_{B}\big(\mathcal V\big({\mathcal{I}}^{a,2}\cup 
		{\mathcal{I}}^{a, 1,1} \big)  \big)  \supset  \mathscr{C}^{\partial}_{B}\big(\mathcal V \big( \widetilde{\mathcal I} \cup 
		{\mathcal{I}}^{a, 1,1}\big)  
		\big)   \big] \geq 1 - C(u \vee 1)(R + L)^d e^{- c (\log L)^{\gamma}}.
	\end{equation} 
		
Next we replace $\widetilde{\mathcal I}$ by ${\mathcal{I}}^{a,3}$. To this effect, we apply 
Theorem~\ref{thm:short_long} with $2L$ playing the role of $L$ (and thus $2l$ replacing $l$), 
$K = R + 4L$ and $$ \textstyle f = f_2= \frac{1 +  l^{-\frac12}}{1 - \Cr{C:sprinkle-easy} l^{-\frac12}} \bar f, \quad \text{ where } \bar f = \frac 1A\, v_- 1_{B_{R + 2L}}(\cdot) + \frac{\delta_1}{40 A}\,v_-1_{B_{R + 6L}}(\cdot).$$ 
As we now explain, this yields for $L \geq C(u, \gamma)$ a coupling of $\widetilde{\mathcal I}$ and ${\mathcal{I}}^{a,3}$ such that
\begin{equation}\label{eq:couple_L'_by_2L}\mathbb Q_{2}^a \big[\mathscr{C}^{\partial}_{B}\big(\mathcal V(\widetilde{\mathcal I})  \big)  \supset  \mathscr{C}^{\partial}_{B}\big(\mathcal V\big({\mathcal{I}}^{a,3}\big) \big)   \big] \geq 1 - C(u \vee 1) (R + L)^de^{- c (\log L)^{\gamma}}.
\end{equation}
Indeed, one readily verifies that the function $f_2$ above \eqref{eq:couple_L'_by_2L} satisfies 
$c(\log L)^{-8} \leq f_2 \leq u$ pointwise on $B_{R + 5L}$, so that the
conditions of Theorem~\ref{thm:short_long} are met for our choice of parameters when $L \ge 
C$, whence Theorem~\ref{thm:short_long} provides a coupling of 
$\mathcal{I}_1=\mathcal{I}^{f_2, 2L}$ and $\mathcal{I}_2 =\mathcal I^{(1 -  \Cr{C:sprinkle-easy}l^{-1/2})P_{2L}^{L'}(f_2 
1_{B_{K}}), L'}=\mathcal I^{(1 + l^{-1/2})P_{2L}^{L'}(\bar f 1_{B_{R + 4L}}), L'}$, with $l =\frac L{L'}$, such that the inclusion 
$\mathcal{V}(\mathcal{I}_1)\subset \mathcal{V}(\mathcal{I}_2)$ holds with probability as in \eqref{eq:couple_L'_by_2L} on account of \eqref{eq:short_long} and by choice of $L'$. 
But  
\begin{align*}
P_L^{L'}(f) &\stackrel{\eqref{eq:f'}}{=} l^{-1} \sum_{0 \le k < l}P_{kL'} (f) \stackrel{\eqref{def:f}}{=} (lA)^{-1} \sum_{0 \le k < l}P_{kL'} (v_- g^0) + \frac{\delta_2}{80} (lA)^{-1}\sum_{0 \le k < l}P_{kL'} (v_- 1_{B_{R+4L}})\\
& \stackrel{\eqref{eq:defgh^ka}}{\le}  (lA)^{-1} \sum_{0 \le k < l}P_{kL'}\Big(\frac{1+P_L}{2} \Big) (v_-\,1_{B_{R + 2L}}) + \frac{\delta_1}{40}(lA)^{-1} \sum_{0 \le k < l} P_{kL'}\Big(\frac{1+P_L}{2} \Big)(v_- 1_{B_{R + 4L}}), 
\end{align*}
and the second line is readily seen to equal $= (2l)^{-1} \sum_{0 \le k < 2l} P_{kL'} (\bar f 1_{B_{R + 4L}})$
with $\bar f$ as above. Since $\widetilde{\mathcal I}$ has the same law as $\mathcal I^{(1 +  l^{-1/2})P_L^{L'}(f), L'}$, one immediately infers from this that $\widetilde{\mathcal I} \leq_{\text{st.}} \mathcal{I}_2$. Moreover, by choice of $L'$ and since $l=\frac L{L'}$, keeping in mind that $\delta_1 =  (\log L)^{-4}$ whilst $\gamma > 10$, one readily sees that 
whenever $L \geq C$,
\begin{equation*} \textstyle f_2 \le (1 + C (\delta_1)^2) \bar f \le  \textstyle \frac {v_-}A\, 1_{B_{R + 2L}}(\cdot ) + \frac{v_- \delta_1}{A} 1_{B_{R+ 6L}}(\cdot ) 
	\stackrel{\eqref{eq:defgh^ka}}{=} v_-(h^{a+1}-  h^{a}),
\end{equation*}
	whence $\mathcal{I}_1 \leq_{\text{st.}} {\mathcal{I}}^{a,3} $ in view of 
	\eqref{def:I'ka}. Applying Lemma~\ref{lem:concatenation} twice to concatenate the coupling of $\mathcal{I}_1$ and $\mathcal{I}_2$ supplied by Theorem~\ref{thm:short_long} with those implied by the two dominations $\widetilde{\mathcal I}\leq_{\text{st.}} \mathcal{I}_2$ and $\mathcal{I}_1 \leq_{\text{st.}} {\mathcal{I}}^{a,3} $ yields $\mathbb Q_{2}^a$ satisfying \eqref{eq:couple_L'_by_2L}, as desired.

	\medskip
	
	Having obtained $\mathbb Q_{1}^a$ and $\mathbb Q_{2}^a$ satisfying \eqref{eq:couple_L_by_L'} and \eqref{eq:couple_L'_by_2L} for each $0 \le a < A$, the remaining task is to extend and 
	concatenate these so as to produce a measure $\mathbb Q$ satisfying \eqref{eq:upper_L2L}. We 
	start by reconstructing for individual $a$'s the full configurations ${\mathcal{I}}^{a}$ and 
	${\mathcal{I}}^{a+1}$. In light of \eqref{eq:intermed_decomp}, \eqref{eq:I'ka_decomp} and the sets 
	coupled under $\mathbb Q_{1}^a$ and $\mathbb Q_{2}^a$, this boils down to 
	incorporating ${\mathcal{I}}^{a, 1,2}$, which is not involved in either of the two measures. The set ${\mathcal{I}}^{a, 1,2} = {\mathcal{I}}^{1,2}$ does not depend on $a$ (see  \eqref{eq:I'kaj}) and thus evolves trivially as $a\to (a+1)$. Hence feeding 
${\mathcal{I}}^{1,2}$ into \eqref{eq:I'ka_decomp} and subsequently \eqref{eq:intermed_decomp} yields the rewrite (still under $\P$)
	\begin{equation}\label{eq:tildeIdecomp}\begin{split}
			 {\mathcal I}^a &= \mathcal I^{1, 2} \cup  \widehat{\mathcal I}^a, \quad \qquad   \widehat{\mathcal I}^a = {\mathcal{I}}^{a,2}
			\cup {\mathcal{I}}^{a, 1,1}, 
			\\
			{\mathcal{I}}^{a+1} &= \mathcal I^{1, 2} \cup \overline{\mathcal I}^{a+1}, \quad  \overline{\mathcal I}^{a+1} = {\mathcal{I}}^{a,3} \cup 
			{\mathcal{I}}^{a, 1,1}, 
		\end{split}
	\end{equation} 
	valid	for all $0 \le a < A$, and 
	and each of the union is over independent sets. 
	
	Returning to $\mathbb Q_{1}^a$ and $\mathbb Q_{2}^a$, observe that the inclusion in \eqref{eq:couple_L'_by_2L} remains true if the pair $(\widetilde{\mathcal I}, {\mathcal{I}}^{a,3})$ is replaced by $(\widetilde{\mathcal I} \cup \mathcal{J}, {\mathcal{I}}^{a,3} \cup \mathcal{J})$, for arbitrary $\mathcal{J} \subset \Z^d$. We apply this observation to $\mathbb Q_{2}^a$ with the choice 
 $\mathcal{J}\stackrel{\text{law}}{=} {\mathcal{I}}^{a, 1,1} $, sampled independently and incorporated into 
 $\mathbb Q_{2}^a$ by 
	suitable extension. In view of \eqref{eq:tildeIdecomp}, \eqref{eq:couple_L_by_L'} asserts that the inclusion $\mathscr{C}^{\partial}_{B}(\mathcal V(  \widehat{\mathcal I}^a )  )  \supset  
	\mathscr{C}^{\partial}_{B}(\mathcal V ( \mathcal I' ))$, where $\mathcal I'\stackrel{\text{law}}{=} \widetilde{\mathcal I} \cup 
	{\mathcal{I}}^{a, 1,1}$, holds with $\mathbb Q_{1}^a$-probability $1 - C(u \vee 1)(R + L)^de^{- c (\log L)^{ \gamma}}$. In the same vein, \eqref{eq:couple_L'_by_2L} lifts to the event $ \mathscr{C}^{\partial}_{B}(\mathcal V ( \mathcal I' )) \supset \mathscr{C}^{\partial}_{B}(\mathcal V(  \overline{\mathcal I}^{a+1} )  )$ under $\mathbb Q_{2}^a$. 
	Thus, concatenating $\mathbb Q_{1}^a$ and $\mathbb Q_{2}^a$ by means of Lemma~\ref{lem:concatenation} produces a coupling $\mathbb Q_a$ of $(  \widehat{\mathcal I}^{a} ,   \overline{\mathcal I}^{a+1} )$ satisfying
	\begin{equation}\label{eq:couple_Qa}
		\mathbb Q^{a} \big[ \mathscr{C}^{\partial}_{B}\big(\mathcal 
		V(  \widehat{\mathcal I}^{a} )  \big)  \supset  \mathscr{C}^{\partial}_{B}\big(\mathcal V(  \overline{\mathcal I}^{a+1} ) \big)   
		\big] \geq 1 - C(u \vee 1)(R + L)^de^{- c (\log L)^{ \gamma}}.\end{equation}	
	Further chaining the couplings $\mathbb Q_a$'s over all $a$ with $0 \leq a < A$ by repeated application of Lemma~\ref{lem:concatenation} and extending the resulting measure by an independent sample of $\mathcal I^{1, 2}$, 
	we arrive in view of \eqref{eq:tildeIdecomp} at a coupling $\mathbb Q$
	of $\mathcal I^{1, 2} \cup  \widehat{\mathcal I}^0\stackrel{\text{law}}{=}  {\mathcal I}^0$ and $ {\mathcal I}^{1, 2} \cup  \overline{\mathcal I}^{A} \stackrel{\text{law}}{=}  {\mathcal I}^{A}$. Combining \eqref{eq:couple_Qa}, the same observation as following \eqref{eq:tildeIdecomp} and a union 
	bound, cf.~Remark~\ref{R:chain},2), it follows that
	\begin{equation}\label{eq:final_couple_Q}
		\mathbb Q \big[ \mathscr{C}^{\partial}_{B}\big(\mathcal 
		V( {\mathcal I}^{1,2} \cup  \widehat{\mathcal I}^{0} )  \big)  \supset  \mathscr{C}^{\partial}_{B}\big(\mathcal V( {\mathcal I}^{1,2} \cup  \overline{\mathcal I}^{A} )  \big)   \big] \geq 1 - 2A e^{- c (\log L)^{ \gamma}},
	\end{equation}
	with $B = B_R$. The measure $\mathbb Q$ is our final coupling satisfying \eqref{eq:upper_L2L} which is an immediate consequence of 
\eqref{eq:final_couple_Q} for $L \ge C(u, \gamma)$ upon recalling the observation following \eqref{eq:intermediate_config}.
\end{proof}

It remains to give the proof of Theorem~\ref{P:long_short} which employs 
Theorem~\ref{thm:long_short} in its annealed version as given by \eqref{eq:long-j-i_annealed} for the choice of environment $\mathcal{I}=\mathcal{I}^{\rho}$. The applicability of \eqref{eq:long-j-i_annealed} rests on a lower bound on the probability of the disconnection event $\mathscr{D}$ introduced in \eqref{e:obs-final2}. The following result is key towards this. Recall the definitions of ${\rm Disc(y)}$ and ${\rm Disc}(S)$ from \eqref{eq:disc-coup} 
and the associated scales $\tilde{\ell}_{\mathcal O}= \ell_{\mathcal O}^{1.01}$ and $\bar \ell_{\mathcal O}> \tilde{\ell}_{\mathcal O}$. 
A set $S \subset \Z^d$ is 
called $r$-separated if the ($\ell^{\infty}$-) distance between any two points in $S$ is larger than 
$r$.
\begin{proposition}[Disconnection estimate] \label{P:traps}
	For all $\gamma > 10$, $K \geq 0$, $L \geq 1$, the following hold. If $\ell_{\mathcal O}$ satisfies \eqref{eq:obs-cond-scales}, $ \bar \ell_{\mathcal O}  \le L^{\frac{1}{10d}}$, 
 $\rho$ fulfills $(\textnormal{C}_{\textnormal{obst}})$ with parameters $(u, Du, \gamma, L,K)$  for some $D>1$,  then for all $L \ge L_0(\gamma, u)$ and any $L^{\frac{1}{2d}}$-separated set $S \subset B_{K + 1.5(L \vee \hat{L})}$, one has \begin{equation}\label{eq:trap}
\P_{\rho}[\mathrm{Disc}(S) = \emptyset]	\le  \big( 1- \P_{u'} [ {{\rm Disc}}(y)] + e^{-c'(u \wedge1) L^{c'}}\big)^{|S|} + C(u \vee 1)(K + L)^d e^{-c (\log L)^{\gamma}}
	\end{equation}
 for some $c' = c'(D) > 0$, where $u' = u(1 - C(\log L)^{-\gamma})$ and $\P_{u'}$ is the law of $\mathcal I^{u'}$.
\end{proposition}
With hopefully obvious notation, in writing $\P_{\rho}[\mathrm{Disc}(S) = \emptyset]$ in \eqref{eq:trap}, we mean that the relevant vacant set $\mathcal{V}= \mathcal{V}^{\rho}= \Z^d \setminus \mathcal{I}^{\rho}$ in \eqref{eq:disc-coup}, and similarly for $ \P_{u'} [ {{\rm Disc}}(y)] $, where $\mathcal{V}=\mathcal{V}^{u'}$. We will prove Proposition~\ref{P:traps} in the next paragraph. Assuming this result we can finish the
\begin{proof}[Proof of Theorem~\ref{P:long_short}] We aim to apply Theorem~\ref{thm:long_short} with the choices $\mathbf{P}=\mathbb{P}_{\rho}$, $\mathcal{I}= \mathcal{I}^{\rho}$ and the length scales $\ell_{\mathcal O}$ and $\bar{\ell}_{\mathcal O}$ as 
\begin{equation}\label{e:obs-l-choice}
\ell_{\mathcal O} = (\log L)^{100 \gamma}, \quad \bar \ell_{\mathcal O}  = L^{ \frac{1}{10d}};
\end{equation}
note that this also fixes the (enlarged) obstacle size $\tilde{\ell}_{\mathcal 
O}= {\ell}_{\mathcal 
O}^{1.01}$, see above \eqref{eq:obs-tilde}, which satisfies $\tilde{\ell}_{\mathcal 
O} < \bar \ell_{\mathcal O}  $ whenever $L \geq C(\gamma)$. Clearly, $\ell_{\mathcal O}$ given by \eqref{e:obs-l-choice} satisfies \eqref{eq:obs-cond-scales} and $\bar \ell_{\mathcal O}  \leq \widetilde{L}$ (see \eqref{e:obs-Ltilde-choice}). Moreover, on account of \eqref{e:couplings-params} and with $l=L/L'$, one has that $ \varepsilon = l^{-\frac12}$ (see \eqref{eq:long-j-i'}) satisfies $\varepsilon \geq C\ell_{\mathcal O}^{-1/100}  $ and $\varepsilon \geq \ell_{\mathcal{O}}^{-1/(d-2)} (\log L)^{\gamma/2}$ for any $L \geq C(\gamma)$.
In view of this, Theorem~\ref{thm:long_short} is in force and \eqref{eq:long-j-i} (see also~\eqref{eq:long-j-i_annealed}) immediately yields \eqref{e:hard-coup-final3}, provided we can argue that
\begin{equation}\label{eq:disc-finalDD}
\mathbb{P}_{\rho}[\mathscr{D}] \geq  1 - C(u \vee 1)(K + 
L)^de^{-c l^{1 / 4}}
\end{equation}
under the assumptions on $\rho$ inherent to Theorem~\ref{P:long_short}, namely, if $\rho$ satisfies $(\textnormal{C}_{\textnormal{obst}})$ with parameters 
$(u', u, \gamma, L, K)$, for some $u'< u$ and \eqref{eq:def_M-quant} holds with $u'' = u'(1- (\log L)^{-4})$ in place of $u$. For later reference, spelling out the latter part using the  scales \eqref{e:obs-l-choice} and \eqref{eq:disc-coup}, we find that 
\begin{equation}
\label{eq:traps14}
\P_{u''} [ {{\rm Disc}}(y)]
\ge  L^{-\frac{1}{10}} \text{ for all $y \in \Z^d$.}
\end{equation}

We will use Proposition~\ref{P:traps} to show \eqref{eq:disc-finalDD}. To ensure the required separatedness for disconnection events occurring within a given cell $\widetilde{C}$, we introduce the event $\widetilde{\mathscr{D}} \subset \mathscr{D}$ as follows. With the notation from \eqref{eq:disc-coup}, where the event $\mathrm{Disc}(y)$ under $\P_{\rho}$ refers to the disconnection event in the configuration $\mathcal{V}=\mathcal{V}^{\rho}$, we set
\begin{equation}
\label{e:obs-final-Dtilde}
\widetilde{\mathscr{D}}= \big\{ \mathrm{Disc}(\mathcal{L}(\widetilde{C})) \neq \emptyset \text{ for all } \widetilde{C} \in \widetilde{\mathcal{C}}\big\},
\end{equation}
where $\mathcal{L}(\widetilde{C}) \stackrel{\text{def.}}{=} (  \lceil L^{\frac1{2d}} +1 \rceil \Z^d ) \cap \widetilde{C}$  for  $\widetilde{C} \in \widetilde{\mathcal{C}}$. In words, \eqref{e:obs-final-Dtilde} requires that each cell $\widetilde{C}$ witnesses a disconnection event $\mathrm{Disc}(y)$ centered around a point $y$ belonging to the `lattice' $\mathcal{L}(\widetilde{C})$. With a view towards applying Proposition~\ref{P:traps}, we now set $S= \mathcal{L}(\widetilde{C})$ for a fixed cell $\widetilde{C} \in \widetilde{\mathcal{C}}$, which is $L^{1/2d}$-separated and readily seen to satisfy $|S| \geq L^{1/2}$. The estimate \eqref{eq:disc-finalDD} now follows upon bounding $\mathbb{P}_{\rho}[\mathscr{D}] \geq \mathbb{P}_{\rho}[\widetilde{\mathscr{D}}]$, feeding \eqref{e:obs-final-Dtilde}, applying a union bound over $\widetilde{\mathcal{C}}$, and using \eqref{eq:trap} together with
\eqref{eq:traps14} to bound $\P_{\rho}[\mathrm{Disc}(\mathcal{L}(\widetilde{C})) = \emptyset]$.
\end{proof}

\subsection{Disconnection estimate}\label{sec:disconnection}
In this section we supply the proof of Proposition~\ref{P:traps}. One issue is that the family of events $\{ {\rm Disc}(y): \, y \in S \}$ involved in the event $\mathrm{Disc}(S)$ is not at all independent. Indeed, by assumption the points in $S$ have a separation scale which is a (small) polynomial in $L$, whereas $\mathcal{I}^{\rho}$ involves trajectories having length of order $L$, cf.~Definition~\ref{def:background}. To generate decoupling, Theorem~\ref{thm:short_long} will be  used repeatedly 
to couple $\mathcal I^{\rho}$ with a configuration $\mathcal I^{\tilde \rho}$ involving suitably 
shortened trajectories.

\begin{proof}[Proof of Proposition~\ref{P:traps}:]
We will work throughout the proof under the weaker assumption that $\hat{L}$, which appears as part of $(\textnormal{C}_{\textnormal{obst}})$, be of the more general form specified in Remark~\ref{rmk:background},\ref{R:lhat}. We will apply Theorem~\ref{thm:short_long} repeatedly to $\mathcal I^{\rho}$, with $\rho$ satisfying $(\textnormal{C}_{\textnormal{obst}})$, see Definition~\ref{def:background}. First note that, by \eqref{eq:disconnect_background2} and \eqref{eq:prelim3.1}, the set $\mathcal I^{\rho}$ has the same law as
	$\mathcal I^{\frac12(1 + P_{\hat{L}/2})f_1, \frac{\hat{L}}{2}} \cup \mathcal I^{f_2, \hat{L}}$ where the latter two configurations are sampled independently. Now define the triplets $(L, N, g)$ for integer $k \geq 0$ (with $L$ and $N$ decreasing in $k$) as follows: 
\begin{equation}\label{eq:traps6.0.0}
L_0 = \hat{L},\quad N_0 = K + 5 (L \vee \hat{L}), \quad g_0 =  f_1 +  f_2,
\end{equation}
and for $k \geq 0$, define recursively
\begin{equation}
\label{eq:traps5}
l^{-1} \stackrel{\textnormal{def.}}{=} \frac{L_{k+1}}{L} =  \max \{ 2^{-m}: 2^{-m} \leq \tfrac1{10} \wedge  (\log (L)^{-4\gamma})\}, \text{ for }k= 0, 1, \dots,
\end{equation}
as well as, with $\Cr{C:sprinkle-easy}$ as in Theorem~\ref{thm:short_long},
\begin{equation}
	\label{eq:traps6.0}
	g_{k+1} \stackrel{\textnormal{def.}}{=} (1- \Cr{C:sprinkle-easy} l^{-1/2}) l^{-1}  \sum_{n=0}^{ l -1} P_{nL_{k+1}}(g1_{B_{N - L}}), \quad N_{k+1} = N - 2L. 
\end{equation}
The recursive definition in \eqref{eq:traps5} is valid since $L_1$ is dyadic by our particular choice of 
the scale $L_0 = \hat L$ as the difference of two dyadic integers $\hat L_1 - \hat L_2$ with $\hat L_2 
\ge 8L(\log \frac L8)^{-4\gamma}$. For later reference, we note that, with the above choices and $k^+ 
\stackrel{\textnormal{def.}}{=} \max\{ k : L \geq L^{\frac{1}{3d}}\}$,
\begin{equation}
	\label{eq:traps5.1}
	\Gamma \stackrel{\text{def.}}{=}  \prod_{k=0}^{k^+} (1-\Cr{C:sprinkle-easy} l^{-1/2})  \stackrel{\eqref{eq:traps5}}{\geq}\prod_{m = \lceil c \log L\rceil}^{\infty} (1-  
	\Cr{C:sprinkle-easy}m^{-2\gamma}) \geq 1 - C(\log L)^{\gamma},
\end{equation}
 for $L \ge C$. Also using that $l^{-1} \leq \frac1{10}$ and $L_0 = \hat{L}$, we obtain
	\begin{equation}
		\label{eq:traps6.1}
		N = N_0 -2L_0 \sum_{n=0}^{k-1}\Big( \prod_{i=0}^{k-2}l_i^{-1} \Big)  \geq N_0 - 3L_0 \geq K + 
		2.6 (L \vee \hat{L}), 
		\text{ for any $k \geq 1$}
	\end{equation}
	(where we interpret an empty product as 1). Now 
	notice that $k^+ \geq 1$ and 
	$L_{k+1} \ge L^{1 - c'}$, where $c'$ is from Theorem~\ref{thm:short_long}), for every $k$ whenever $L \ge C(\gamma)$, which will be tacitly assumed henceforth. We aim to apply Theorem~\ref{thm:short_long} for all $0\leq k < k^+$ with the choices $L = 
	L$, $L' = L_{k+1}$, 
	$K = N - L$ and $f = g$. This hinges on obtaining a suitable lower bound for $g$, see the hypotheses of Theorem~\ref{thm:short_long}. To this effect, we now proceed to verify that for all $x \in 
	B_{N}$, 
\begin{align}
&u \geq g(y) \geq \Big[ \prod_{i=0}^{k-1} (1- \Cr{C:sprinkle-easy}l_{i}^{-1/2})\Big] (\log L_0)^{-{{\gamma}}}, \text{ for $0\leq k \leq k^+$} \label{eq:traps6}.
\end{align}
	Indeed, for $k=0$, by definition of $g_0$ in \eqref{eq:traps6.0.0}, upper and lower 
	bounds in \eqref{eq:traps6} follow immediately from \eqref{eq:disconnect_background2} since 
	$N_0 = K + 5(L \vee \hat{L})$. Suppose now that \eqref{eq:traps6} holds for some $k $
	with $0\leq k <k^+$. The upper and lower bounds for $g_{k+1}$ then follow from 
	\eqref{eq:traps6.0}, by substituting the corresponding bounds for $g$ from 
	\eqref{eq:traps6} valid by induction hypothesis. In doing so, note that the indicator 
	function present in \eqref{eq:traps6.0} is inconsequential for the purpose of bounding $ g_{k+1}(x)$ when $x \in B_{N_{k+1}+L_{k+1}}$ 
	since one applies $P_t$ for $t=nL_{k+1} \leq L - L_{k+1}$ and 
	$N_{k+1} + L_{k + 1} + ( L - L_{k+1}) = N - L$ by the definition of $N$. 
	Together with \eqref{eq:traps5} and \eqref{eq:traps5.1}, \eqref{eq:traps6} implies that
\begin{equation}
\label{eq:traps8}
Du \geq g(x) \geq 
c(\log L)^{- {{{\gamma}}}}, \text{ for } x \in B_{N}, \text{ }0\leq k \leq k^+,
\end{equation}
whenever 
$L \ge C$. In view of \eqref{eq:traps8}, 
	for each $0\leq k < k^+$, and chaining the resulting couplings using Lemma~\ref{lem:concatenation} yields a coupling  $\mathbb{Q}$ of $\mathcal I^{\rho}$ and $\mathcal I^{\tilde{\rho}}$ (see \eqref{eq:prelim2} for notation), where $\tilde \rho:\mathbb{N}^\ast \times \Z^d \to \mathbb{R}_{+}$ with
	\begin{equation} \label{e:rho-tilde-disc}
		\tilde{\rho}(\ell, x)  \stackrel{\textnormal{def.}}{=} \textstyle\frac{g_{k^+}(x)}{L_{k^+}}1_{\{\ell = L_{k^+}\}}
	\end{equation}
	and the inclusion $\mathcal I^{\tilde{\rho}} \subset \mathcal I^{{\rho}}$ holds with $\mathbb{Q}$-probability at least 
	$1-e^{-c (\log L)^{\gamma}}$, for suitable $c > 0$. Since $\tilde{\rho}$ in 
	\eqref{e:rho-tilde-disc} involves trajectories of length $L_{k^+}-1$ and $\frac{L_{k^+}}{L^{1/2d}}\to 
	0$ as $L \to \infty$ be definition of $k^+$, the events $ {{\rm Disc}}(y)$, $y \in S$ are independent 
	under $\P_{\tilde \rho} $ for large enough $L$ since $S$ is $L^{1/2d}$-separated by assumption. 
	We will now argue that
\begin{equation}
\label{eq:traps14}
\P_{\tilde \rho} [ {{\rm Disc}}(y)] \ge 
\P_{u'} [ {{\rm Disc}}(y)] - e^{-c'(u' \wedge1) L^{c'}}, \text{ for all $y \in S$,}
\end{equation}
where $u' = 	u'(1 - C(\log L)^{-\gamma})$ and $c'$ depends on $D$ in addition to dimension $d$.
If \eqref{eq:traps14} holds then applying the coupling and using independence immediately gives \eqref{eq:trap}.
	
	To complete the proof, it thus remains to show \eqref{eq:traps14} with $\tilde{\rho}$ given by \eqref{e:rho-tilde-disc}.  This follows from a direct application of Proposition~\ref{prop2:cube} (with, say,  $N= L^{1/9d}$, $a=3$ and small enough $\varepsilon > 0$), 
	by which $\mathcal{V}^{\tilde{\rho}}$ inherits up to a 
	small coupling error the disconnection lower bound 
	$\P_{u'} [ {{\rm Disc}}(y)]$, provided we can 
	show that $\tilde{\rho}$ satisfies for all $x \in B_{K + 2.6 (L \vee \hat{L})}$,
	\begin{align}
		&\label{eq:traps11}
		\bar{\ell}_x^{\, \tilde \rho} = \sum_{\ell' \geq 0} E_x[\tilde \rho(\ell'+\mathbb{N}^*, X_{\ell'})]  
		\in \textstyle  
		[u(1 - C(\log L)^{\gamma}), u] \quad \mbox{ and }  \quad \tilde \rho(\mathbb{N}^*, x) \leq {Du}{L^{-\frac{1}{3d}}}.
	\end{align}
	The second of these conditions follows immediately by \eqref{eq:traps8} and \eqref{e:rho-tilde-disc}. In order to obtain the required estimates on $\bar{\ell}_x^{\, \tilde \rho}$ in \eqref{eq:traps11}, we first observe that, for all $x \in B_{N+L}$, 
	\begin{align}
	& g(y) =  \Big[ \prod_{j=0}^{k-1} 
	(1 - \Cr{C:sprinkle-easy} l_{j}^{-1/2})\Big]\frac{L}{L_0} \sum_{ 0 \leq n < \frac{L_0}{L}}(P_{nL} g_0)(x),  
	\text{ for $1\leq k \leq k^+$.} \label{eq:traps7}
\end{align}
	For $k=1$, \eqref{eq:traps7} is \eqref{eq:traps6.0}, noting that for points $x$ of 
	interest, the indicator function in \eqref{eq:traps6.0} can be omitted. For the induction 
	step, one uses again the observation that 
	$N_{k+1} + L_{k + 1} + ( L - L_{k+1}) = N - L$ together with the semigroup property for 
	$(P_n)_{n\geq 0}$. Now, for all $x \in B_{K + 2.6 (L\vee \hat{L})}$, 
	noting that \eqref{eq:traps7} is in force for any such $x$ due to \eqref{eq:traps6.1}, we have
	\begin{multline*}
		\sum_{\ell' \geq 0} E_x[\tilde \rho(\ell'+\mathbb{N}^*, X_{\ell'})] =\sum_{m \geq 1} 
		\sum_{0 \leq \ell' < m} E_x[\tilde \rho(m, X_{\ell'})] \stackrel{\eqref{e:rho-tilde-disc}}{=} \sum_{0 \leq \ell' < L_{k^+}}L_{k^+}^{-1} \big( P_{\ell'}g_{k^+} \big)(x)\\
		\stackrel{\eqref{eq:traps7}, \eqref{eq:traps5.1}}{\geq}  \frac{\Gamma}{L_0} \displaystyle\sum_{ 0 \leq n < \frac{L_0}{L_{k^+}}}  
		\sum_{ 0 \leq \ell' < L_{k^+}}\big(P_{nL_{k^+} + \ell'}g_0\big)(x)\stackrel{\eqref{eq:traps5}, \eqref{eq:traps6.0.0}
		}{=} \frac{\Gamma}{\hat{L}}\sum_{\ell'=0}^{\hat{L}-1}\big[ P_{\ell'}(f_1 + f_2)\big](x)
		\\
		\stackrel{\eqref{eq:disconnect_background2}}{=}  \Gamma   \sum_{\ell \in \{\hat{L}/2, \hat{L}\}}\sum_{0 \leq \ell' < \ell}E_x[\rho(\ell, X_{\ell'})] = \Gamma  \sum_{\ell' \geq 0} E_y[ 
		\rho(\ell'+\mathbb{N}^*, X_{\ell'})], 
	\end{multline*}
	from which the first condition in \eqref{eq:traps11} follows by \eqref{eq:disconnect_background0} and 
	\eqref{eq:traps5.1}. Overall \eqref{eq:traps11}  thus holds, which completes the proof.
\end{proof}

\textbf{Acknowledgements.} This work has received funding from the European Research Council 
(ERC) under the European Union’s Horizon 2020 research and innovation programme (grant 
agreement No.~757296). HDC acknowledges funding from the NCCR SwissMap, the Swiss FNS, and 
the Simons collaboration on localization of waves. SG’s research was supported by the SERB grant 
SRG/2021/000032, a grant from the Department of Atomic Energy, Government of India, under project 
12–R\&D–TFR–5.01–0500 and in part by a grant from the Infosys Foundation as a member of the 
Infosys-Chandrasekharan virtual center for Random Geometry. PFR thanks the IMO in Orsay for its hospitality during the final stages of this project, with funding from ERC Grant agreement No.~740943. FS has received funding from the 
European Research Council (ERC) under the European Union’s Horizon 2020 research and innovation 
program (grant agreement No 851565). During this period, AT has been supported by grants
``Projeto Universal'' (406250/2016-2) and ``Produtividade em Pesquisa'' (304437/2018-2) from
CNPq and ``Jovem Cientista do Nosso Estado'', (202.716/2018) from FAPERJ. SG, PFR, FS and AT
thank IH\'ES and the University of Geneva for their hospitality on several occasions.

\bibliography{biblicomplete}
\bibliographystyle{abbrv}

\end{document}